\documentclass[11pt]{amsart}              \parindent0pt
\title{Smoothing of singular intersections of ellipsoids: pyramitoid }    \author[E. Artal]{Enrique Artal Bartolo}
\address[E.~Artal]{Departamento de Matem\'{a}ticas, IUMA \\
Universidad de Zaragoza \\
C.~Pedro Cerbuna 12, 50009, Zaragoza, Spain}
\urladdr{\url{http://riemann.unizar.es/~artal}}
\email{\href{mailto:artal@unizar.es}{artal@unizar.es}}

\author[S. L\'{o}pez de Medrano]{Santiago L\'{o}pez de Medrano}
\address[S. L\'{o}pez de Medrano]{Instituto de Matem\'{a}ticas\\
Universidad Nacional Aut\'{o}noma de M\'{e}xico\\
04510 Ciudad de M\'{e}xico, M\'{e}xico}
\email{\href{mailto:santiago@matem.unam.mx}{santiago@im.unam.mx}}

\author[M.T. Lozano]{Mar\'{\i}a Teresa Lozano}
\address[M.T. Lozano]{Departamento de Matem\'{a}ticas, IUMA \\
Universidad de Zaragoza \\
C.~Pedro Cerbuna 12, 50009, Zaragoza, Spain}
\email{\href{mailto:tlozano@unizar.es}{tlozano@unizar.es}}
\date{\today}

\thanks{This research was partially supported by PID2024-156181NB-C33 funded by MICIU/AEI/ 10.13039/501100011033
and by FEDER, UE. The two first named authors were also partially supported by UNAM--Papiit grant IN106324.}

\usepackage{graphicx, amssymb,amsmath}
\usepackage[T1]{fontenc}
\usepackage[utf8]{inputenc}
\usepackage{enumitem}
\usepackage{hyperref}
\usepackage[labelfont=rm]{caption}
\usepackage[labelformat=simple]{subcaption}
\usepackage{relsize}
\usepackage{tikz}

\usepackage{pgf}
\usepackage{xcolor}

\usepackage{pgfplots}
\usepackage{tikz-3dplot}

\tdplotsetmaincoords{60}{115}

\usepgfplotslibrary{patchplots}

\usetikzlibrary{shapes,cd}
\pgfplotsset{compat=newest}

\newcommand{\connectedsum}{\mathop{\mathlarger{\mathlarger{\text{\Large\#}}}}}

\newcommand{\orb}{\text{\rm orb}}
\DeclareMathOperator{\iso}{Iso}

\newtheorem{thm}{Theorem}[section]
\newtheorem{proposition}[thm]{Proposition}
\newtheorem{coro}[thm]{Corollary}
\newtheorem{lema}[thm]{Lemma}
\theoremstyle{definition}
\newtheorem{dfn}[thm]{Definition}
\newtheorem{ejm}[thm]{Example}
\theoremstyle{remark}
\newtheorem{remark}[thm]{Remark}

\newcommand{\esencial}[1]{\mathcal{E}_{(#1)}}
\newcommand{\esencialp}[1]{\tilde{\mathcal{E}}_{(#1)}}
\newcommand{\alma}[1]{\mathcal{C}_{(#1)}}
\newcommand{\almap}[1]{\tilde{\mathcal{C}}_{(#1)}}
\newcommand{\domo}[1]{l\mathbf{Y}_{#1}}
\newcommand{\orbdomo}[1]{l\mathbf{Y}^{\orb}_{#1}}
\newcommand{\cw}[1]{\mathcal{CW}_{(#1)}}
\newcommand{\ct}[1]{\mathcal{T}_{(#1)}}
\newcommand{\bipy}[2]{b\mathbf{#1}_{#2}}
\newcommand{\girobipy}[2]{gb\mathcal{#1}_{#2}}
\newcommand{\sgirobipy}[2]{sgb\mathcal{#1}_{#2}}

\numberwithin{equation}{section}

\dedicatory{In memory of Fred Cohen}

\begin{document}
\begin{abstract}
  The goal of this work is to continue the study the smoothings of 3\nobreakdash-dimen\-sional manifolds with singularities obtained as real-moment angle manifolds of non simple right-angled Coxeter polyhedral orbifolds. They appear in the study of coaxial intersections of ellipsoids. In particular
we introduce the concept of $n$-pyramitoid generalizing the $n$-pyramid.  \end{abstract}

\maketitle

\tableofcontents

\section*{Introduction}

The surprising discovery in 2007 of the close relationship between the new theory of polyhedral products developed by Fred and his
coworkers~\cite{BBCG} and the study of the topology of intersections
of coaxial ellipsoids initiated by several authors some years earlier (see~\cite{LdM:21}) gave a new impulse to the second theory through the participation of algebraic topologists (where one of the most enthusiastic promoters was Fred himself) and through the appearance of new lines of research.

One of these lines is the study of the topology and geometry of the 3-dimensional cases initiated by the present authors, beginning with the study of the manifold associated to the dodecahedron~\cite{ALdML:16}, equivalent to the study of the polyhedral product associated with the icosahedron, and other related 3-manifolds. More recently, in~\cite{ALdeML2025}, we have studied the case of the octahedron which (not being a simple polyhedron) produces a manifold with singularities, but gives rise to other polyhedral products when they are smoothed.

Our goal in the present paper is to provide smoothings of the isolated singularities of $3$-dimensional intersections of coaxial ellipsoids, in particular, the ones coming from pyramids.

 The polyhedron $P$ associated to an intersection of coaxial ellipsoids is given by a set of linear equations coming from  the quadratic equations defining the ellipsoids. If $P$ has a
 non-simple vertex $v$ (i.e., of valence bigger than 3),
 then the intersection variety $Z(P)$ has singular points. A suitable small deformation of the equations produces a \emph{smoothing} of the variety. This smoothing consists of small changes on the polyhedron  at the  neighbourhoods of the non-simple vertices, in such a way that no new faces are created and that  the vertices become simple (valence 3). This change is  also called a \emph{smoothing} of $P$. Deformations of smaller collections of vertices (which give non-simple polyhedra and singular intersections) can also be considered.

  The $n$-pyramid  $\mathcal{Y}_n$ (that is, the pyramid with basis the $n$-gon) is
  associated with an intersection $Z(\mathcal{Y}_n)$ of ellipsoids
  in  $\mathbb{R}^{n+1}$. If $n>3$, $Z(\mathcal{Y}_n)$ has a singular isolated singularity. The case $n=4$ with a generic singularity was studied in \cite{ALdeML2025} and here we extend this study to the general case $n>4$. Every smoothing of the $n$-pyramid $\mathcal{Y}_n$ is a simple polyhedron with $n+1$ faces, one of them being a polygon of $n$ edges. So we still have a polyhedron with $n+1$ faces. The vertices in the basis of the pyramid are all simple, so they are not affected by the smoothing operation. And the basis is still an $n$-polygon that has an edge in common with any of the remaining faces. The same happens in the intermediate steps of a smoothing.

A polyhedron $P$ associated to a smooth intersection of ellipsoids carries  on an orbifold structure. Each face is a mirror face and two faces intersect in an edge with dihedral angle~$\frac{\pi}{2}$. The orbifold fundamental group $\pi _1^{\orb}(P)$, see~\S\ref{sec:graphs}, has the reflections on the faces $\{ x_1,...,x_n\}$ as generators, where $n$ is the number of faces of $P$, with the relations $x_i^2=1$. Each edge $l_{ij}$, intersection of the faces $i$ and $j$, give rise to a relation $x_ix_j=x_jx_i$. This is a  right-angled Coxeter group (\emph{RACG}). The \emph{defining graph} $G_P$ that associates  a vertex to every generator and an edge to each commutation relation\footnote{The defining graph is not the ``Coxeter graph''. The Coxeter graph has the same
vertex set~$V$; however, the edge set only considers edges for each pair $\{x_i,x_j\}$ such that $(x_i,x_j)^m =1,\,m>2$~\cite{Davis1983} } coincides with the $1$-skeleton of the dual polyhedron of $P$. Observe that the defining graph $G_{\mathcal{Y}_n}$ of the $n$-pyramid $\mathcal{Y}_n$ is isomorphic to the $1$-skeleton of its dual polytope which turns out to be an $n$-pyramid. Since we are interested in pyramids because they appear as truncations
around non-simple vertices, the orbifold structure where the basis is not a mirror is also interesting. Their smoothings
are handlebodies and we will relate them to Heegaard splittings.

The paper is organized as follows: In \S\ref{sec:graphs} we recall the concept of right-angled Coxeter group, RACG, associated to a finite simplicial graph $\Gamma$. We explain that, if  this defining graph $\Gamma$ is the dual graph of the $1$-skeleton of a polyhedron, $P$, an orbifold structure  in $P$ is defined.  In \S\ref{sec:pyramitoid} we define and study a class of polyhedra called \emph{pyramitoids} where a face (the basis) has a edge in common with every one of the others faces. The $2$-skeleton of the pyramitoid without the basis constitutes the \emph{dome}. The right-angled orbifold structure on pyramitoids are analyzed in \S\ref{sec:orbifolds}. Being $\mathbf{Y}_n$ a simple $n$-pyramitoid, a smoothing of the $n$-pyramid, the \emph{real moment-angle manifold} $Z(\mathbf{Y}_n)$ is obtained (Theorem \ref{tzpn}). We also consider the orbifold structure on $\mathbf{Y}_n$ where only the faces of the dome as mirror faces. We proof that its real moment-angle manifold is a handleboby $H_n$. Some graphs associated to the pyramitoid permit to obtain the core and a set of meridian discs of  $H_n$. In \S\ref{sec:bipyramitoid} bipyramitoids, the polyhedron obtained by the identification by the basis face of two $n$-pyramitoids are studied. By construction, a Heegaard splitting  of its real moment-angle manifold can be analyzed. Some examples are given. Some comments on the general case are given in \S\ref{sec:general}.

\section{Graphs, groups and orbifolds}\label{sec:graphs}

In this work, a graph $\Gamma$ is a finite simplicial graph, which is determined by a finite set
$V:=V(\Gamma)$, the set of \emph{vertices}, and a subset $E:=E(\Gamma)\subset\{e\subset V\mid \#e=2\}$,
the set of \emph{edges}. Sometimes we will put some extra structure in these graphs, e.g., a
weight on $V$ or $E$.

Given a graph $\Gamma$ with a weight $w:E\to\mathbb{Z}_{>1}$. We can associate some
groups to this weighted graph. The Artin group relative to $(\Gamma,w)$ is
\[
A(\Gamma,w):=
\langle
v\in V
\mid
\underbrace{v_1\cdot v_2\cdot v_1\cdot v_2\cdot\ldots}_{w(e)\text{ factors}} = \underbrace{v_2\cdot v_1\cdot v_2\cdot v_1\cdot\ldots}_{w(e)\text{ factors}}
\text{ for }e=\{v_1,  v_2\}\in E
\rangle.
\]
The Coxeter group $C(\Gamma,w)$ is defined in the same way, adding the relations $v^2=1$, for $v\in V$.
When $w$ is the constant function with value~$2$, the function will be dropped from the notation and
the groups $A(\Gamma), C(\Gamma)$ are the \emph{right angle} Artin (or Coxeter) groups, $RAAG$ (or $RACG$), associated to $\Gamma$.
The graph $\Gamma$ is called the \emph{defining graph} of $A(\Gamma)$ and $C(\Gamma)$.

The defining graph~$\Gamma$ is not what is called in the literature a Coxeter graph $\Gamma_C$. For the Coxeter graph
the weight is defined as $w_C:E_C\to\mathbb{Z}_{>2}\cup\{\infty\}$. The sets of vertices coincide but
\[
E_C:=\{e\in E\mid w(e)>2\}\cup\{e\subset V\mid \#e=2, e\notin E\},\quad
w_C(e):=
\begin{cases}
w(e)&\text{if }e\in E,\\
\infty&\text{otherwise}.
\end{cases}
\]

Orbifolds will appear in a natural way in this work. An orbifold is defined as a pair $X^\orb:=(X,\mathcal{A})$
where $X$ is a topological space and $\mathcal{A}$ is a \emph{maximal atlas} of \emph{orbifold charts}.
A chart is a $4$-tuple $(U,\tilde{U},\phi, G)$,  $U$ is an open set of $X$ and $\tilde{U}$ is an open set of either $\mathbb{R}^n$ or $\mathbb{R}^{n-1}\times\mathbb{R}_{\geq 0}$,
$\psi:\tilde{U}\to U$ is continuous, $G$ is a finite group of linear isomorphisms of $\mathbb{R}^n$ fixing~$\tilde{U}$
such that there is a homeomorphism $\tilde{\phi}:\tilde{U}/G\to U$ fitting in the following commutative diagram:
\[
\begin{tikzcd}
&\tilde{U}\ar[ld, "\pi_G" above left]\ar[rd, "\phi"]&\\
\tilde{U}/G\ar[rr,"\tilde{\phi}"]&&U.
\end{tikzcd}
\]
These charts have natural compatibility properties. In some cases we may ask that the compatibility properties live
in the differentiable category, or in the analytic category if the groups are formed by linear automorphisms of $\mathbb{C}^m$.
It is not hard to see that given $p\in X$ there is a chart such that $\#\phi^{-1}(p)=1$; for any such chart
the acting groups are conjugate and its conjugacy class can be denoted by $\iso_p$, the \emph{isotropy group of}~$p$.

We are interested in a special class
of orbifolds.

\begin{dfn}
An orbifold $X^\orb$
is a \emph{right-angled orbifold} if for any
chart the group $G$ is isomorphic to
$(\mathbb{Z}/2)^k$, for some $0\leq k\leq n$,
acting by reflections on $k$ coordinate hyperplanes.
If for some $p\in X$, $\iso_p=(\mathbb{Z}/2)^k$, we say that $p$ is a \emph{smooth point} if $k=0$,
a \emph{mirror point} if $k=1$, and a \emph{$k$-singular point} if $k\geq 2$. In dimension~$2$,
a $2$-singular point is called a corner reflector of order~$2$.
\end{dfn}

We start with two examples
$X_1^\orb, X_2^\orb$ of $1$-dimensional
right-angled orbifolds where $X_1=X_2=[-1, 1]\subset \mathbb{R}$. For $X_1^\orb$, $\pm 1$ are mirror points,
and for $X_2$ only $-1$ is mirror point.

Orbifold fundamental groups can be defined as the usual ones, either as \emph{homotopy classes} of
\emph{orbifold} loops~\cite{HDu} or as the group of deck transformations of the universal cover of an orbifold~\cite{bmp:03}.

Using an orbifold version of the Seifert-van Kampen Theorem, we have that
$\pi_1^{\orb}(X_1^\orb; 0)$ is isomorphic to $\mathbb{Z}/2*\mathbb{Z}/2$
(resp. $\pi_1^{\orb}(X_2^\orb; 0)\cong\mathbb{Z}/2$),
which is the
$RACG$ with defining graph a set of two vertices
(one vertex).

\begin{ejm}\label{ejm:dim2}
Given $n\geq 2$, we can consider
a (combinatorial) $n$-gon $p_n$ in $\mathbb{R}^2$ (the edges may not be linear). We consider
a right-angled orbifold structure
$X_1^\orb$ where the open edges in the boundary are mirrors and the vertices
are corner reflectors of order~$2$. It is not hard
to see that $\pi_1^{\orb}(X_1^\orb)$
is the
$RACG$ with defining graph $\Gamma_n$, the dual graph of $\partial p_n$.
Note that for $n>2$, $\Gamma_n$ is the boundary of an $n$-polygon~$q_n$.

We consider another right-angled orbifold structure $X_2^\orb$ on $p_n$ where
one of the open edges is formed by smooth boundary
points (and its vertices are mirror boundary points). It is not hard
to see that $\pi_1^{\orb}(X_2^\orb)$
is the
$RACG$ with defining graph $\check{\Gamma}_n$
(obtained from $\Gamma_n$ erasing one vertex
and its surrounding edges).
\begin{figure}[ht]
    \centering
    \begin{tikzpicture}
    \foreach \a in {0, ...,7}
    {
    \coordinate (A\a) at ({\a*72-54}:1);
    }
    \filldraw[fill=gray!20!white, draw=white] (A0) -- (A1) -- (A2) -- (A3) -- (A4) -- cycle;
    \foreach \a[evaluate={\b=\a+1}] in {0, ..., 4}
    {
    \draw[blue, line width=1.2pt] (A\a) -- (A\b);
    \fill[red] (A\a) circle[radius=.1cm];
    }
    \node at (-1,1) {$X_1^\orb$};
    \foreach \a[evaluate={\b=\a+1}] in {0, ..., 6}
    {
    \coordinate (B\a) at ($.5*(A\a)+.5*(A\b)$);
    }
    \foreach \a[evaluate={\b=\a+1}] in {0, ..., 5}
    {
    \fill[green!80!black] (B\a) circle[radius=.1cm];
    \draw[green!80!black, line width=1.2pt] (B\a) -- (B\b);
    }
     \node[green!80!black, right] at (-27:1) {$\Gamma_n$};
     \begin{scope}[xshift=5cm]
         \foreach \a in {0, ...,7}
    {
    \coordinate (A\a) at ({\a*72-54}:1);
    }
    \filldraw[fill=gray!20!white, draw=white] (A0) -- (A1) -- (A2) -- (A3) -- (A4) -- cycle;
    \foreach \a[evaluate={\b=\a+1}] in {0, ..., 3}
    {
    \draw[blue, line width=1.2pt] (A\a) -- (A\b);
    }
    \foreach \a in {1, ..., 3}
    {
    \fill[red] (A\a) circle[radius=.1cm];
    }
    \draw[brown, line width=1.2pt] (A4) -- (A5);
    \fill[magenta] (A0) circle[radius=.1cm];
    \fill[magenta] (A4) circle[radius=.1cm];
    \node at (-1,1) {$X_2^\orb$};
    \foreach \a[evaluate={\b=\a+1}] in {0, ..., 6}
    {
    \coordinate (B\a) at ($.5*(A\a)+.5*(A\b)$);
    }
    \foreach \a[evaluate={\b=\a+1}] in {0, ..., 2}
    {
    \draw[green!80!black, line width=1.2pt] (B\a) -- (B\b);
    }
    \foreach \a[evaluate={\b=\a+1}] in {0, ..., 3}
    {
    \fill[green!80!black] (B\a) circle[radius=.1cm];
    }
     \node[green!80!black, right] at (-27:1) {$\check{\Gamma}_n$};
     \end{scope}
    \end{tikzpicture}
    \caption{$2$-dimensional orbifolds $X_1^\orb$, $X_2^\orb$ for $n=5$.}
    \label{fig:racg2}
\end{figure}
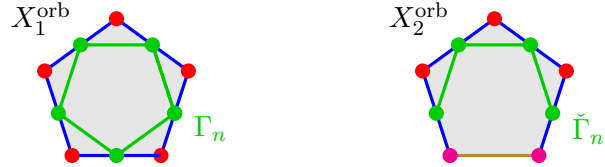
\end{ejm}

\begin{ejm}\label{ejm:dim3}
Let $P$ be a combinatorial simple polyhedron in $\mathbb{R}^3$ with $n$ faces.
 We can associate a right-angled orbifold $P^\orb$ structure where interior points in  $P$ are smooth points, interior points in each face are  mirror points,  points in the interior of each edge are corner reflector points and vertices are $3$-singular points. The orbifold fundamental group $\pi _1^{\orb}(P^\orb))$ has $n$ generators $\{ x_1,...,x_n\}$ associated to reflections
 on the faces of~$P$; as they come from reflections they satisfy the relations $x_i^2=1$. Each edge $l_{ij}$ is the intersection of the faces $i$ and $j$, and gives rise to a relation $x_ix_j=x_jx_i$. Hence, $\pi _1^{\orb}(P^\orb))$ is a
$RACG$. The defining graph $G_P$  coincides with the $1$-skeleton $\Gamma_P$ of the dual polytope of $P$: one vertex for each $2$-face and one edge for each $1$-face.

\begin{figure}[ht]
    \centering
    \begin{tikzpicture}
    \coordinate (A) at (-150:1);
    \coordinate (B) at (-30:1);
    \coordinate (C) at (90:1);
    \coordinate (D) at (0, 0);
    \draw (A) -- (B) -- (C) -- cycle;
    \foreach \x in {A, B, C}
    {
    \draw[dashed] (D) -- (\x);
    }
    \fill (A) circle[radius=.1cm];
    \fill (B) circle[radius=.1cm];
    \fill (C) circle[radius=.1cm];
    \fill[red] (D) circle[radius=.1cm];
     \begin{scope}[xshift=5cm]
    \coordinate (A) at (0:1);
    \coordinate (B) at (90:1);
    \coordinate (C) at (180:1);
    \coordinate (D) at (-90: 1);
    \coordinate (P) at (0,0);
    \coordinate (Q) at (2, 0);
    \draw (A) -- (B) -- (C) -- (D) -- cycle;
    \foreach \x in {A, B, C, D}
    {
    \draw (P) -- (\x);
    }
    \foreach \x in {A, B, D}
    {
    \draw[dashed] (Q) -- (\x);
    }
    \draw[dashed] (Q) to[out=90, in=0] ($1.25*(B)$) to[out=180, in=90] (C);
    \fill (A) circle[radius=.1cm];
    \fill (B) circle[radius=.1cm];
    \fill (C) circle[radius=.1cm];
    \fill (D) circle[radius=.1cm];
    \fill (P) circle[radius=.1cm];
    \fill[red] (Q) circle[radius=.1cm];
     \end{scope}
    \end{tikzpicture}
    \caption{Defining graphs of $G_P$ for the tetrahedron and the cube. Red vertices and dashed edges
    stand for the edges and vertices to be eliminated when considering the defining graph
    for $\pi _1^{\orb}(P_n^\orb)$.}
    \label{fig:racg3}
\end{figure}
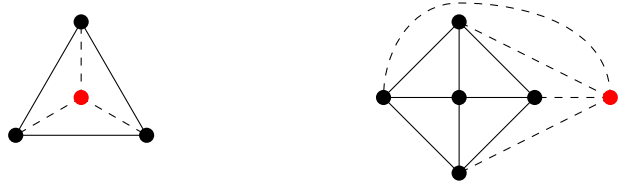

Fix one face, say the one associated to $x_n$. We  can consider another orbifold structure $P_n^\orb$ where the interior points
 of $x_n$ are boundary points. For the interior points of the edges (resp. vertices) of $x_n$, the \emph{chart} is given by
 $\mathbb{R}^2\times\mathbb{R}_\geq 0$ and the group is acting by the reflection on $x_2=0$ (resp. the reflections
 on $x_1=0$ and $x_2=0$).
\end{ejm}

\begin{remark}\label{rem:Z}
    A simple $3$-polytope $P$ (with $n$ faces) and its orbifold structure $P^\orb$ are associated to an intersection of ellipsoids.
    There are two ways to understand this association.

    From the orbifold point of view, let us consider the epimorphism $\pi_1^\orb(P^\orb)\to(\mathbb{Z}/2)^n$,
    which sends the generators associated to the faces to the canonical basis of $(\mathbb{Z}/2)^n$. This epimorphism
    induces a orbifold Galois cover $Z(P)\to P^\orb$ such that $Z(P)$ is a manifold and it is homeomorphic
    to an intersection of ellipsoids. The manifold $Z(P)$ is called the \emph{real moment-angle manifold} of~$P$, i.e a manifold
    admitting a quotient by $(\mathbb{Z}/2)^n$, such that $Z(P)/(\mathbb{Z}/2)^n\cong P$, and $n$ is the number
    of faces of $P$. This is related with the concept of \emph{small cover} of $P$, introduced by Davis and Januszkiewicz in \cite{DJ1991}, where $(\mathbb{Z}/2)^n$ is replaced by $(\mathbb{Z}/2)^3$.

    The other point of view consists in embedding $P$ in $\mathbb{R}^n$, where the intersection of $P$ with each coordinate
    hyperplane is a face. Then, $Z(P)$ is obtained as the union of $P$ with the images of $P$ by all the compositions of
    reflections with respect to the coordinate hyperplanes.
\end{remark}

\begin{ejm}
    Let us consider as $P_n$ a regular $n$-polygon centered at the origin. We can fix equations $a_j x + b_j y + c_j = 0$, $j=1,\dots,n$,
    for the lines supporting the edges such that
    \[
    P_n=\{(x, y)\in\mathbb{R}^2\mid a_j x + b_j y + c_j\geq 0,\ j=1,\dots,n\}
    \]
    and the sum of all equations equals~$1$. Let
    \[
    \begin{tikzcd}[row sep=0pt, /tikz/column 1/.append style={anchor=base east}, /tikz/column 2/.append style={anchor=base west}]
    \mathbb{R}^2\rar["\Phi_n"] &     \mathbb{R}^n  \\
    (x, y)\rar[mapsto]&(a_1 x + b_1 y + c_1,\dots, a_n x + b_n y + c_n).
    \end{tikzcd}
    \]
    The image of $P_n$ is the intersection of $\mathbb{R}_{\geq 0}^n$ with the image of $\Phi_n(\mathbb{R}^2)$ (which is contained in $r_1 + \dots r_n=1$). Since  $\Phi_n(\mathbb{R}^2)$ is an affine subspace we can express as the solution of a
    linear system.  Concrete expression can be found in~\cite{LMdlV}:
    \begin{equation}\label{eq:poln}
    \begin{aligned}
        r_1 + \dots + r_n &= 1\\
        r_i - r_{i+3} + (2\tau + 1) (r_{i+2} - r_{i+1})&=0\\
        i=1,\dots,n-3,&
    \end{aligned}
    \end{equation}
    where $\tau=\cos\frac{2\pi}{n}$.
    The edges of $\Phi_n(P_n)$ are the intersections with the coordinate hyperplanes. Then $Z(P_n)$ is the intersection of the $n-2$ coaxial quadrics replacing $r_i$ by~$x_i^2$.

    We can also consider the $n$-pyramid $\mathbf{P}_n$ with basis $P_n$ in $\mathbb{R}^2\equiv\{z=0\}$ and apex~$(0, 0, 1)$.
    The lateral faces are supported by the planes $a_j x + b_j y + c_j (1 - z) = 0$, $j=1,\dots,n$, while the basis is supported by $z=0$. The sum of these equations equal~$1$ and replacing $=$ by $\geq$ we obtain the inequalities determining $\mathbf{P}_n$. In a similar way,
    we consider
    \[
    \begin{tikzcd}[row sep=0pt, /tikz/column 1/.append style={anchor=base east}, /tikz/column 2/.append style={anchor=base west}]
    \mathbb{R}^2\rar["\Psi_{n}"] &     \mathbb{R}^{n+1}  \\
    (x, y, z)\rar[mapsto]&(a_1 x + b_1 y + c_1 - c_1 z,\dots, a_n x + b_n y + c_n - c_n z, z).
    \end{tikzcd}
    \]
    The image of $\mathbf{P}_n$ is the intersection of $\mathbb{R}_{\geq 0}^{n+1}$ with the image of $\Psi_n(\mathbb{R}^3)$ (which is contained in $r_1 + \dots r_n + r_{n + 1}=1$).
    If we replace the first equation in \eqref{eq:poln} by
    $r_1 + \dots + r_n + r_{n + 1}=1$ we obtain the equations of the affine subspace $\Phi_n(\mathbb{R}^3)$.
    The faces of $\Psi_n(\mathbf{P}_n)$ are the intersections with the coordinate hyperplanes. Then $Z(\mathbf{P}_n)$ is given as the intersection of $n-2$ coaxial ellipsoids.
\end{ejm}

\section{Pyramitoids}\label{sec:pyramitoid}

\begin{dfn}
An \emph{$n$-pyramitoid} $\mathbf{Y}_n$ is a 3-dimensional polyhedron  with $n+1$ faces, such that (at least) one of them (which we call a \emph{basis}) is a polygon with $n$ edges which, therefore, has an edge in common with each of the remaining faces of the pyramitoid (which we call \emph{lateral faces}, and the union of all the lateral faces will be called the \emph{dome} of the pyramitoid).  Choosing one basis, which we denote by $p_n$, we denote by $\domo{n}$ its dome\footnote{We will see soon that in most cases there is only one basis and when there are several of them, the corresponding domes are all equivalent}.

\end{dfn}
Every pyramid is a pyramitoid. The triangular prism is a simple 4-pyramitoid with any quadrangular face of the prism as basis.
There is an easy way to have a figure of a pyramitoid by drawing a projection $\pi_n:\mathbf{Y}_n\to p_n$ of the dome on a basis~$p_n$.
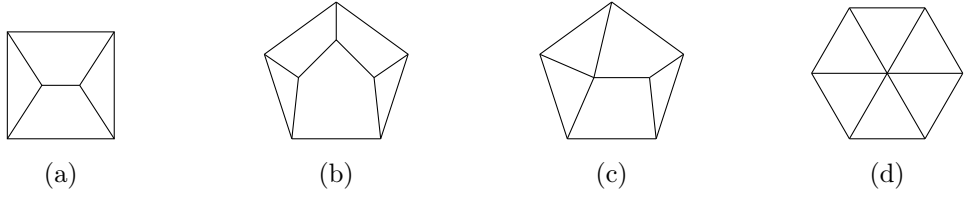
\begin{figure}[ht]
\begin{center}
\begin{subfigure}[b]{0.24\textwidth}
\centering
\begin{tikzpicture}
\draw (-135:1) rectangle (45:1);
\draw (-135:1) -- (-1/4, 0) -- (135:1);
\draw (-45:1) -- (1/4, 0) -- (45:1);
\draw (-1/4, 0) -- (1/4, 0);
\end{tikzpicture}
 \subcaption{}
\label{subfigd}
\end{subfigure}
\begin{subfigure}[b]{0.24\textwidth}
\centering
\begin{tikzpicture}
\coordinate (B1) at (-1/2,0);
\coordinate (B4) at (1/2,0);
\coordinate (B0) at (0,1/2);

\foreach \x in {0, ..., 5}
{
\coordinate (A\x) at (72*\x+90:1);
}
\foreach \x[count=\y] in {0, ..., 4}
{
\draw (A\x) -- (A\y);
}
\foreach \x in {0, 1, 4}
{
\draw (B\x) -- (A\x);
}

\draw (A3) -- (B4) -- (B0) -- (B1) -- (A2);
\end{tikzpicture}
 \subcaption{}
\label{subfigc}
\end{subfigure}
\begin{subfigure}[b]{0.24\textwidth}
\centering
\begin{tikzpicture}
\coordinate (A) at ({-1/3/sqrt(2)},{0/4/sqrt(2)});
\coordinate (B) at (1/2,0);
\draw (A) -- (B);
\foreach \x in {0, ..., 5}
{
\coordinate (A\x) at (72*\x+90:1);
}
\foreach \x[count=\y] in {0, ..., 4}
{
\draw (A\x) -- (A\y);
}
\foreach \x in {0, 1, 2}
{
\draw (A) -- (A\x);
}
\foreach \x in {3, 4}
{
\draw (B) -- (A\x);
}
\end{tikzpicture}
 \subcaption{}
\end{subfigure}
\begin{subfigure}[b]{0.24\textwidth}
\centering
\begin{tikzpicture}
\coordinate (O) at (0,0);
\foreach \x in {0, ..., 6}
{
\coordinate (A\x) at (60*\x:1);
}
\foreach \x[count=\y] in {0, ..., 5}
{
\draw (O) -- (A\x);
\draw (A\x) -- (A\y);
}
\end{tikzpicture}
 \subcaption{}
\label{subfig-hexagono}
\end{subfigure}
\caption{Pyramitoids. The examples \subref{subfigd} and \subref{subfigc}  are simple pyramitoids. }
\label{ejemplos}
\end{center}
\end{figure}

The image by $\pi_n$ of the $1$-skeleton of the pyramitoid on the basis determines a cellular decomposition of~$p_n$
which is combinatorially isomorphic to the dome,
as shown in Figure~\ref{ejemplos}.

\begin{dfn}
Given an $n$-pyramitoid $\mathbf{Y}_n$ and a basis~$p_n$
of $\mathbf{Y}_n$, the \emph{essential tree} $\esencial{\mathbf{Y}_n, p_n}$ is the
graph composed by all the vertices and by the edges not contained in~$p_n$.
The \emph{$p_n$-essential tree} $\esencialp{\mathbf{Y}_n, p_n}$ is the image of $\esencial{\mathbf{Y}_n}$ by~$\pi_n$.
The cellular decomposition of~$p_n$ determined by $\esencialp{\mathbf{Y}_n, p_n}$ is called the
\emph{essential decomposition of~$p_n$} and it is denoted by $\cw{\mathbf{Y}_n, p_n}$ ($p_n$ may be dropped from the notation).
\end{dfn}

We will drop the basis in the notation if no confusion is likely to arise.
$\esencial{\mathbf{Y}_n}$ is actually a tree, since, otherwise, there would be a circuit enclosing a face that would not touch the basis.
Each vertex  in the basis $p_n$ is in one edge of the essential tree $\esencial{\mathbf{Y}_n}$. This edge will be called a \emph{leaf} of the $\esencial{\mathbf{Y}_n}$.
Its image by $\pi_n$ is a  leaf of the $p_n$-essential tree $\esencialp{\mathbf{Y}_n, p_n}$.
Both trees are isomorphic and the $p_n$-essential tree $\esencial{\mathbf{Y}_n}$ by~$\pi_n$ will be considered together with its embedding in the plane.
\begin{dfn}
    Given an $n$-pyramitoid $\mathbf{Y}_n$ and a basis~$p_n$
of $\mathbf{Y}_n$, the \emph{core tree} $\alma{\mathbf{Y}_n, p_n}$ is the
essential tree minus the leaves.
The \emph{$p_n$-core tree} $\almap{\mathbf{Y}_n, p_n}$ is the image of $\alma{\mathbf{Y}_n, p_n}$ by~$\pi_n$ and it is also considered as an embedded graph in the plane.
\end{dfn}

\begin{ejm}
There is a particular interesting example, the $n$-books for $n>2$. An \emph{$n$-book} is the polyhedron obtained from the $n$-prism by collapsing one of the vertical faces into a horizontal segment. The $3$-book (the tetrahedron) has $4$ bases, the $4$-book (Figure~\ref{subfigd}) has $3$ bases and starting from $n=5$ (Figure~\ref{subfigc}) and all the $n$-books have $2$ bases.
\end{ejm}

\begin{lema}\label{lema:contar}
    Let $\mathbf{Y}_n$ be a simple pyramitoid (meaning that every vertex lies in exactly $3$ edges). Let $v, a$
    be the number of vertices and edges, respectively. Then:
    \begin{enumerate}[label=\rm(\arabic{enumi})]
        \item $v=2(n-2)$ and $a=3(n-1)$.
        \item The essential tree $\esencial{\mathbf{Y}_n}$ has $2n - 3$ edges and $2n-2$ vertices.
        \item\label{lema:contar-3} The core tree $\alma{\mathbf{Y}_n, p_n}$ has $n - 3$ edges and  $n-2$ vertices.
        \item A lateral face has at most $n$ edges.
        \item\label{lema:contar5} If $\mathbf{Y}_n$ is not an $n$-book, then it admits a unique basis~$p_n$.
    \end{enumerate}
\end{lema}

\begin{proof}
    Since every edge has two vertices and the pyramitoid is simple, we have $v=\frac{2a}{3}$.
    It follows from an Euler characteristic computation that $a=3n-3$ and $v=2n-2$.
    The essential tree $\esencial{\mathbf{Y}_n}$ is obtained eliminating the edges of the basis, hence, it has $2n-3$ edges
     and the core tree $\alma{\mathbf{Y}_n}$ has $n-3$ edges (by removing the $n$ leaves of $\esencial{\mathbf{Y}_n}$). Because $\esencial{\mathbf{Y}_n}$ and $\alma{\mathbf{Y}_n}$ are trees, its have  $2n-2$ and $n-2$ vertices repectivelly.

      A lateral face with $r$ edges has one in common with the basis and two surrounding leaves stemming from it, while the remaining $r-3$ lay in the core tree. Since this tree has $n-3$ edges, it follows that $r\le n$, so no lateral face has more that $n$ edges.

    If one lateral face $F$ has $n$ edges it is also a basis of $\mathbf{Y}_n$ and its $n-3$ edges not touching the first basis form the whole core tree. It follows that the rest of the $n-2$ leaves must join the $n-2$ vertices of the core tree with the $n-2$ vertices of the basis not in $F$, one by one, so $\mathbf{Y}_n$ is the $n$-book.
\end{proof}

The following lemma is straightforward:

\begin{lema}
  The result of truncating a vertex in the basis of an $n$-pyramitoid $\mathbf{Y}_{n}$ with basis $p_n$ is an $(n+1)$-pyramitoid $\mathbf{Y}_{n +1}$ with basis~$p_{n+1}$, where we have added a lateral triangular face.

  The tree $\esencial{\mathbf{Y}_{n +1}}$ (resp. $\esencialp{\mathbf{Y}_{n +1}}$) is obtained from $\esencial{\mathbf{Y}_{n}}$
  (resp. $\esencialp{\mathbf{Y}_{n}}$) by adding two leaves to the extremity of the leaf of $\esencial{\mathbf{Y}_{n}}$ involved in the truncation.

  The tree $\alma{\mathbf{Y}_{n +1}}$ (resp. $\almap{\mathbf{Y}_{n +1}}$) is obtained from $\alma{\mathbf{Y}_{n}}$
  (resp. $\almap{\mathbf{Y}_{n}}$) by adding the leaf of $\esencial{\mathbf{Y}_{n}}$ (resp. $\esencial{\mathbf{Y}_{n}}$) involved in the truncation.
\end{lema}

\begin{lema}\label{lem}
Let $\mathbf{Y}_{n}$ be a simple $n$-pyramitoid  with basis $p_n$. The following properties hold:
\begin{enumerate}[label=\rm(\arabic{enumi})]
  \item\label{lem1} The lateral faces are $m$-gons for $3\leq m\leq n$. There is at most one $n$-gon, and if
  there is one $n$-gon, $\mathbf{Y}_{n}$ is an $n$-book.
\item\label{lem3}  At least two lateral faces are triangles.
  \item\label{lem4} If two triangles are consecutive, $n=3$, i.e., $\mathbf{Y}_{n}$ is a tetrahedron.
\end{enumerate}
\end{lema}

\begin{proof}
The first part statement \ref{lem1} is a consequence of the fact that the faces of the $n$-pyramitoid are at most $n$-gons;
the second part is Lemma~\ref{lema:contar}\ref{lema:contar5}.
The two following possibilities prove \ref{lem3}:
\begin{enumerate}[label=\rm(\alph{enumi})]
        \item $\alma{\mathbf{Y}_{n}}$ is a vertex (so $\mathbf{Y}_{n}$ is a tetrahedron);
        \item $\alma{\mathbf{Y}_{n}}$ has at least two vertices of valence $1$, those vertices are of valence $3$ in $\esencial{\mathbf{Y}_{n}}$, therefore they are the vertices of triangles with the other two vertices in the basis.
\end{enumerate}

Finally, if two adjacent faces are triangles they have a  common vertex  $v\in\alma{\mathbf{Y}_{n}}$. Then  $\mathbf{Y}_{n}$ is a tetrahedron or the vertex $v$  is of valence bigger than $3$ in $\esencial{\mathbf{Y}_{n}}$ and therefore in $\mathbf{Y}_{n}$, then it is not simple. Hence, \ref{lem4} follows.
\end{proof}

The following result allows to recover all simple pyramitoids by truncations starting from a tetrahedron.

 \begin{thm}
For $n \geq 4$,  every simple $n$-pyramitoid $\mathbf{Y}_{n}$ is the result of truncating a vertex in the basis of a simple $(n-1)$-pyramitoid $\mathbf{Y}_{n-1}$.
 \end{thm}

 \begin{proof}
 Let be $\Delta$ one of the triangular lateral faces of  $\mathbf{Y}_{n}$. We can deform $\mathbf{Y}_{n}$ onto
 a new $(n-1)$-pyramitoid $\mathbf{Y}_{n-1}$ by contracting $\Delta$ onto a point. This deformation can be performed
 if $n\geq 4$.
\end{proof}

In order to obtain a classification giving
the complete list of all simple $n$-pyramitoids, we can use a \emph{label} of each one of them consisting of a sequence
$(b_1,\dots,b_n)$, in a clockwise cyclic order, where $b_1+3,\dots,b_n+3$ are the number of sides of the $n$ faces around the essential tree $\esencial{\mathbf{Y}_{n}}$, or equivalently, the faces around the boundary of the basis.

\begin{proposition}\label{prop:labels}
The cyclic label $(b_1,\dots,b_n)$ of a simple $n$-pyramitoid $\mathbf{Y}_{n}$ has the following properties:
\begin{enumerate}[label=\rm(\arabic{enumi})]
  \item\label{prop1} $0\leq b_i\leq n-3$
  \item\label{prop3}  At least two of the $b_i$'s are equal to $0$.
  \item\label{prop4} No adjacent pair in the cyclic order $(b_i, b_{i+1})$ can be $(0,0)$, except for the case $n=3$ (then the label is $(0,0,0)$ and $\mathbf{Y}_{n}$ is the tetrahedron).
  \item\label{prop2} $\displaystyle\sum_{i=1}^n b_i =2(n-3)$.
  \item\label{prop5} Let $\mathbf{Y}_{n+1}$ be the simple pyramitoid obtained by
  truncating the vertex in the $i^{\text{th}}$-position. Its label is
  $(b_1,\dots,b_{i-1}, b_{i} + 1, 0, b_{i+1} + 1, b_{i+2},\dots, b_n)$.
    \item\label{prop6} Let $\mathbf{Y}_{n-1}$ be the simple pyramitoid obtained by
  contracting the triangle associated to $b_i=0$. Its label is
  $(b_1,\dots,b_{i-2}, b_{i-1}-1, b_{i+1} - 1, b_{i+2},\dots, b_n)$.

\end{enumerate}
\end{proposition}

\begin{proof}
The third first statements are a consequence of Lemma~\ref{lem}.
For \ref{prop2}, we have that a  face around the essential tree has two vertices in common with the basis and the others are in the core tree. The number of vertices in the core tree is $n-2$ and they belong to~$3$ faces. Therefore
        \begin{equation*}
\sum_{i=1}^n b_i = 2n + 3(n-2) -3n = 2n-6.
        \end{equation*}
The two last statements are straightforward.
\end{proof}

\begin{remark}
    Note that not all the cyclic labels satisfying the previous proposition correspond to an actual pyramitoid. For example, if a cyclic label contains a consecutive triple $0, 1, 0$ or $1, 0, 1$, then the only option for the label is $(0, 1, 0, 1)$ and $n=4$.
\end{remark}

\section{Orbifold structures of pyramitoids}\label{sec:orbifolds}

We are going to consider two \emph{right-angled} orbifold structures in a simple pyramitoid~$\mathbf{Y}_n$.
First, the usual structure $\mathbf{Y}_n^{\orb}$ is the one where all its faces are mirror faces (this is the orbifold structure
which appears when dealing intersections of ellipsoids). Second, the  orbifold structure
$\orbdomo{n}$ where only the faces of the dome are mirror.

We have seen in Example~\ref{ejm:dim3} that the orbifold fundamental group associated to
right-angled orbifold structures in simple polytopes (with mirror points at all the faces, except eventually one)
is the right-angled Coxeter group with defining graph the $1$-skeleton of the dual polytope.
Let us describe these graphs in our case.

Let $\mathbf{Y}_n$ be a simple pyramitoid with basis $p_n$. We consider the cellular decomposition $\cw{\mathbf{Y}_n; p_n}$.
Let us consider the graph $\Gamma_n$ (associated to the basis $p_n$) defined in Example~\ref{ejm:dim2}. The graph $\Gamma_n$  is the boundary of the polygon~$q_n$.
Let us denote by $G_{\orbdomo{n}}$ the graph obtained by adding to $\Gamma_n$ one edge for each edge of the
$p_n$-core graph of $\mathbf{Y}_n$ as follows: Let $e$ be an edge of $\almap{\mathbf{Y}_n}$; it is the intersection of two $2$-cells
of $\cw{\mathbf{Y}_n; p_n}$. Each one of these cells contains a vertex of $\Gamma_n$, and we join these two vertices by a new edge, see
Figures{\rm~\ref{pyre4} and \rm~\ref{pyre5}}.

\begin{lema}\label{lema:min_tri}
Let $\mathbf{Y}_n$ be a simple pyramitoid.
\begin{enumerate}[label=\rm(\arabic{enumi})]
    \item\label{lema:min_tri1} The defining graph of the right-angled Coxeter group~$\pi_1^\orb(\orbdomo{})$ is $G_{l\mathbf{Y}_n}$.
    \item\label{lema:min_tr2} It determines a minimal triangulation $\ct{\mathbf{Y}_n; q_n}$ of the $n$-polygon~$q_n$ (with $n$ vertices), see
Figures{\rm~\ref{pyre4} and \rm~\ref{pyre5}}.
    \item\label{lema:min_tri3} The dual graph of $\ct{\mathbf{Y}_n; q_n}$ is $\almap{\mathbf{Y}_n}$.
\end{enumerate}
\end{lema}

\begin{proof}
\ref{lema:min_tri1} follows from the arguments in Example~\ref{ejm:dim3}.
The $n$ commutation relations of the consecutive faces are associated to the leaves of the essential tree;
the other $n-3$ relations corresponding to the edges of the core tree.

The edges of $G_{l\mathbf{Y}_n}$ not in $\Gamma_n$ are diagonals of $q_n$, which are pairwise non-intersecting.
Observe that $n-3$ is the maximal number of non-intersecting diagonals in a $n$-polygon and we obtain~\ref{lema:min_tr2}.

It is clear that the graph with a vertex for each triangle of the triangulation and an edge for each diagonal of the triangulation is the core tree, obtaining~\ref{lema:min_tri3}.
\end{proof}

Actually any minimal triangulation of $q_n$ determines a simple pyramitoid~$\mathbf{Y}_n$. Hence the number of different simple pyramitoids with basis an $n$-polygon~$p_n$ is equal to the number of minimal triangulation of $q_n$. In both cases,
we compute these numbers up to cyclic permutations (rotations). Otherwise speaking two labels are considered equivalent
if they coincide after a cyclic permutation but not if they only coincide after reversing the order.

 \begin{figure}[ht]
\begin{center}
\begin{tikzpicture}
\begin{scope}
\coordinate (A) at (0,0);
\coordinate (B) at (1.5,0);
\coordinate (C) at (1.5,1);
\coordinate (D) at (0,1);
\coordinate (E) at (.5,.5);
\coordinate (F) at (1,.5);

\fill[yellow] (D) -- (E) -- (F) --(C) -- cycle;
\fill[yellow] (A) -- (E) -- (F) --(B) -- cycle;
\draw (A) rectangle (C);
\draw[line width=1pt, blue] (A) -- (E);
\draw[line width=1pt, blue] (D) -- (E);
\draw[line width=1pt, red] (E) -- (F);
\draw[line width=1pt, blue] (B) -- (F);
\draw[line width=1pt, blue] (C) -- (F);
\draw ($1/3*(A) + 1/3*(D) + 1/3*(E)$) node {$\scriptstyle 1$} circle[radius=.15cm];
\draw ($1/3*(B) + 1/3*(C) + 1/3*(F)$) node {$\scriptstyle 3$} circle[radius=.15cm];
\draw ($1/4*(D) + 1/4*(C) + 1/4*(F) + 1/4*(E)$) node {$\scriptstyle 2$} circle[radius=.15cm];
\draw ($1/4*(A) + 1/4*(B) + 1/4*(F) + 1/4*(E)$) node {$\scriptstyle 4$} circle[radius=.15cm];
\end{scope}

\begin{scope}[xshift=3cm]
\coordinate (A) at (0,0);
\coordinate (B) at (1.5,0);
\coordinate (C) at (1.5,1);
\coordinate (D) at (0,1);
\coordinate (E) at (.5,.5);
\coordinate (F) at (1,.5);

\draw[line width=1pt, blue] (A) -- (E);
\draw[line width=1pt, blue] (D) -- (E);
\draw[line width=1pt, red] (E) -- (F);
\draw[line width=1pt, blue] (B) -- (F);
\draw[line width=1pt, blue] (C) -- (F);
\node at  ($1/3*(A) + 1/3*(D) + 1/3*(E)$) {};
\node at  ($1/3*(B) + 1/3*(C) + 1/3*(F)$) {};
\node at  ($1/4*(D) + 1/4*(C) + 1/4*(F) + 1/4*(E)$) {};
\node at  ($1/4*(A) + 1/4*(B) + 1/4*(F) + 1/4*(E)$) {};
\node at (.75, -.25) {$(0101)$};
\end{scope}

\begin{scope}[xshift=6cm]
\coordinate (A1) at (0,.5);
\coordinate (A2) at (.5,1);
\coordinate (A3) at (1, .5);
\coordinate (A4) at (.5,0);

\draw[line width=1pt] (A1) -- (A2) -- (A3) -- (A4) -- (A1) (A2) -- (A4);
\foreach \x in {1, ..., 4}{
\filldraw[fill=white] (A\x) circle[radius=.2];
\node at (A\x) {$\x$};
}
\draw[red, line width=1pt] ($.66*(A1) + .34*(A3)$) -- ($.34*(A1) + .66*(A3)$);
\end{scope}
\end{tikzpicture}
\caption{Essential and core trees, numerical label and defining graph $G_{l\mathbf{Y}_4}$ for $\mathbf{Y}_4$.  }\label{pyre4}
\end{center}
\end{figure}
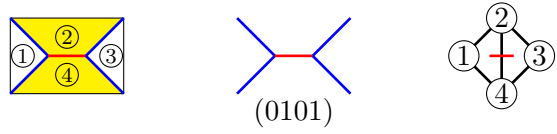

 \begin{figure}[ht]
\begin{center}
\begin{tikzpicture}
\begin{scope}
\foreach \x in {0, ..., 5}{
\coordinate (A\x) at (\x*72 + 18:1);
}

\draw (A0) -- (A1) -- (A2) -- (A3) -- (A4) -- (A5);
\coordinate (A) at (-1/3,0);
\coordinate (B) at (0, 1/3);
\coordinate (C) at (1/3,0);

\fill[yellow] (A0) -- (A1) -- (B) --(C) -- cycle;
\fill[yellow] (A2) -- (A1) -- (B) --(A) -- cycle;
\fill[green] (A3) -- (A4) -- (C) --(B) -- (A) -- cycle;

\draw[line width=1pt, blue] (A0) -- (C);
\draw[line width=1pt, blue] (A1) -- (B);
\draw[line width=1pt, blue] (A2) -- (A);
\draw[line width=1pt, blue] (A3) -- (A);
\draw[line width=1pt, blue] (A4) -- (C);

\draw[line width=1pt, red] (A) -- (B) -- (C);

\draw ($1/3*(A2) + 1/3*(A3) + 1/3*(A)$) node {$\scriptstyle 1$} circle[radius=.15cm];
\draw ($1/3*(A4) + 1/3*(A0) + 1/3*(C)$) node {$\scriptstyle 4$} circle[radius=.15cm];
\draw ($1/4*(A2) + 1/4*(A1) + 1/4*(A) + 1/4*(B)$) node {$\scriptstyle 2$} circle[radius=.15cm];
\draw ($1/4*(A0) + 1/4*(A1) + 1/4*(C) + 1/4*(B)$) node {$\scriptstyle 3$} circle[radius=.15cm];
\draw ($1/5*(A3) + 1/5*(A4) + 1/5*(A) + 1/5*(B) + 1/5*(C)$) node {$\scriptstyle 5$} circle[radius=.15cm];

\end{scope}

\begin{scope}[xshift=3cm]
\foreach \x in {0, ..., 5}{
\coordinate (A\x) at (\x*72 + 18:1);
}
\coordinate (A) at (-1/3,0);
\coordinate (B) at (0, 1/3);
\coordinate (C) at (1/3,0);

\draw[line width=1pt, blue] (A0) -- (C);
\draw[line width=1pt, blue] (A1) -- (B);
\draw[line width=1pt, blue] (A2) -- (A);
\draw[line width=1pt, blue] (A3) -- (A);
\draw[line width=1pt, blue] (A4) -- (C);

\draw[line width=1pt, red] (A) -- (B) -- (C);

\node at (0, -1) {$(01102)$};
\end{scope}

\begin{scope}[xshift=6cm]

\foreach \x in {0, ..., 5}{
\coordinate (A\x) at (-\x*72 -90:1);
}

\draw[line width=1pt] (A0) -- (A1) -- (A2) -- (A3) -- (A4) -- (A5) (A2) -- (A5) -- (A3);
\foreach \x in {1, ..., 5}{
\filldraw[fill=white] (A\x) circle[radius=.2];
\node at (A\x) {$\x$};
}
\draw[red, line width=1pt, yshift=-.25cm] (1/2, 0) -- (0, 1/2) -- (-1/2,0);
\end{scope}
\end{tikzpicture}

\caption{Essential and core trees, numerical label  and defining graph $G_{l\mathbf{Y}_5}$  for $\mathbf{Y}_5$.  }\label{pyre5}
\end{center}
\end{figure}
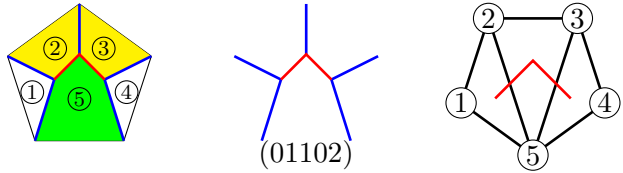

\begin{thm}
The number $ N_n$ of simple $n$-pyramitoids  for $n$ a prime number, is
\begin{equation*}
 \quad N_n= \frac{(n+1)...(2n-4)}{(n-2)!}, n\geq 5.
\end{equation*}
\end{thm}

\begin{proof}
Euler gave the formula to compute the total number of ways to triangulate a convex $n$-gon by non-intersecting diagonals. It is the $(n-2)^{\text{th}}$ Catalan number:
  \begin{equation*}
  \frac{n(n+1)...(2n-4)}{(n-2)!}.
\end{equation*}
    Recall that one pyramitoid produces one minimal triangulation in the dual polygon. One can identify
    the pyramitoids with the orbits of minimal triangulations by the action of the cyclic group
    $C_n$ of order~$n$ by $\frac{2k\pi }{n}$-rotation around the axis orthogonal to its basis.
    Note that these orbits have more than one element. Hence, if $n$ is prime all the orbits
    have $n$ elements and
\begin{equation*}
   N_n= \frac{1}{n}\frac{n(n+1)...(2n-4)}{(n-2)!}=\frac{(n+1)...(2n-4)}{(n-2)!}
\qedhere
\end{equation*}
\end{proof}

Then $N_5=1$, $N_7=6$, ... . For a general number $n$ a detailed study of the orbits of the cyclic group $C_n$ should  be made. For instance, in case $n=6$, the $14$ triangulations given by the Euler formula are distributed in $4$ orbits: one with $6$ elements, two with $3$ elements and one with $2$ elements. Therefore, there are $4$ different $6$-pyramitoids, see Fig. \ref{tri6} and \ref{pyre61}.

\begin{figure}[ht]
\begin{center}

\begin{tikzpicture}[scale=.75]
\foreach \x in {0, ..., 6}{
\coordinate (A\x) at (\x*60:1);
}
\newcommand\hexagono[6]{

\draw (A0) -- (A1) -- (A2) -- (A3) -- (A4) -- (A5) -- (A6);

\draw (A#1) -- (A#2);

\draw (A#3) -- (A#4);

\draw (A#5) -- (A#6);
}

\hexagono{2}{4}{2}{5}{2}{0}
\begin{scope}[xshift=2.2cm, yshift=0]
\foreach \x in {0, ..., 6}{
\coordinate (A\x) at (\x*60:1);
}
\hexagono{3}{5}{3}{6}{3}{1}
\end{scope}

\begin{scope}[xshift=4.4cm, yshift=0]
\foreach \x in {0, ..., 6}{
\coordinate (A\x) at (\x*60:1);
}
\hexagono{4}{0}{4}{1}{4}{2}
\end{scope}

\begin{scope}[xshift=6.6cm, yshift=0]
\foreach \x in {0, ..., 6}{
\coordinate (A\x) at (\x*60:1);
}
\hexagono{5}{1}{5}{2}{5}{3}
\end{scope}

\begin{scope}[xshift=8.8cm, yshift=0]
\foreach \x in {0, ..., 6}{
\coordinate (A\x) at (\x*60:1);
}
\hexagono{0}{2}{0}{3}{0}{4}
\end{scope}

\begin{scope}[xshift=11cm, yshift=0]
\foreach \x in {0, ..., 6}{
\coordinate (A\x) at (\x*60:1);
}
\hexagono{1}{3}{1}{4}{1}{5}
\end{scope}

\begin{scope}[xshift=-1.1cm, yshift=-2.2cm]
\foreach \x in {0, ..., 6}{
\coordinate (A\x) at (\x*60:1);
}
\hexagono{2}{0}{0}{3}{3}{5}
\end{scope}

\begin{scope}[xshift=1.1cm, yshift=-2.2cm]
\foreach \x in {0, ..., 6}{
\coordinate (A\x) at (\x*60:1);
}
\hexagono{3}{1}{1}{4}{4}{0}
\end{scope}

\begin{scope}[xshift=3.3cm, yshift=-2.2cm]
\foreach \x in {0, ..., 6}{
\coordinate (A\x) at (\x*60:1);
}
\hexagono{1}{5}{5}{2}{2}{4}
\end{scope}

\begin{scope}[xshift=7.7cm, yshift=-2.2cm]
\foreach \x in {0, ..., 6}{
\coordinate (A\x) at (\x*60:1);
}
\hexagono{5}{1}{1}{4}{4}{2}
\end{scope}

\begin{scope}[xshift=9.9cm, yshift=-2.2cm]
\foreach \x in {0, ..., 6}{
\coordinate (A\x) at (\x*60:1);
}
\hexagono{0}{2}{2}{5}{5}{3}
\end{scope}

\begin{scope}[xshift=12.1cm, yshift=-2.2cm]
\foreach \x in {0, ..., 6}{
\coordinate (A\x) at (\x*60:1);
}
\hexagono{1}{3}{3}{0}{0}{4}
\end{scope}

\begin{scope}[xshift=4.4cm, yshift=-4.4cm]
\foreach \x in {0, ..., 6}{
\coordinate (A\x) at (\x*60:1);
}
\hexagono{0}{2}{2}{4}{4}{0}
\end{scope}

\begin{scope}[xshift=6.6cm, yshift=-4.4cm]
\foreach \x in {0, ..., 6}{
\coordinate (A\x) at (\x*60:1);
}
\hexagono{1}{3}{3}{5}{5}{1}
\end{scope}

\end{tikzpicture}
\caption{Triangulations in the hexagon.  Each line represents an orbit by rotation and symmetry; in the second line there are two orbits by rotation.}\label{tri6}
\end{center}
\end{figure}
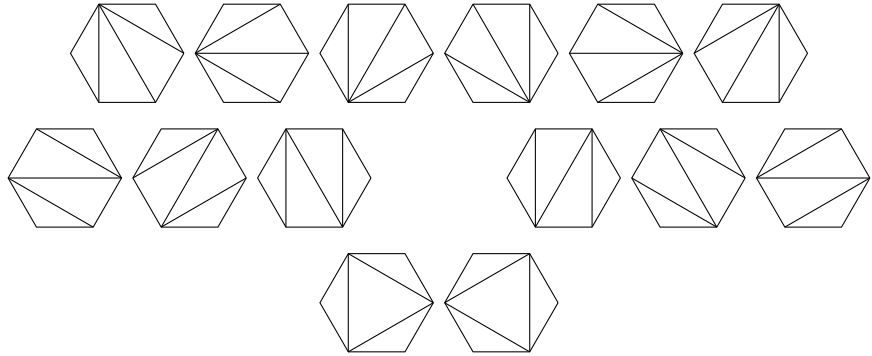

 \begin{figure}[ht]
\begin{center}
\begin{tikzpicture}
\begin{scope}
\foreach \x in {0, ..., 6}{
\coordinate (A\x) at (\x*60:1);
}

\draw (A0) -- (A1) -- (A2) -- (A3) -- (A4) -- (A5) -- (A6);
\coordinate (A) at (-1/2,0);
\coordinate (B) at (-1/3, 1/3);
\coordinate (C) at (1/3,1/3);
\coordinate (D) at (1/2,0);

\fill[yellow] (A3) -- (A) -- (B) --(A2) -- cycle;
\fill[yellow] (A2) -- (B) -- (C) --(A1) -- cycle;
\fill[yellow] (A0) -- (D) -- (C) --(A1) -- cycle;
\fill[blue!50!white] (A4) -- (A5) --  (D) -- (C) --(B) -- (A) -- cycle;

\draw[line width=1pt, blue] (A0) -- (D);
\draw[line width=1pt, blue] (A1) -- (C);
\draw[line width=1pt, blue] (A2) -- (B);
\draw[line width=1pt, blue] (A3) -- (A);
\draw[line width=1pt, blue] (A4) -- (A);
\draw[line width=1pt, blue] (A5) -- (D);

\draw[line width=1pt, red] (A) -- (B) -- (C) -- (D);

\draw ($1/3*(A3) + 1/3*(A4) + 1/3*(A)$) node {$\scriptstyle 1$} circle[radius=.15cm];
\draw ($1/3*(A5) + 1/3*(A0) + 1/3*(D)$) node {$\scriptstyle 5$} circle[radius=.15cm];
\draw ($1/4*(A2) + 1/4*(A3) + 1/4*(A) + 1/4*(B)$) node {$\scriptstyle 2$} circle[radius=.15cm];
\draw ($1/4*(A2) + 1/4*(A1) + 1/4*(C) + 1/4*(B)$) node {$\scriptstyle 3$} circle[radius=.15cm];
\draw ($1/4*(A0) + 1/4*(A1) + 1/4*(C) + 1/4*(D)$) node {$\scriptstyle 4$} circle[radius=.15cm];
\draw ($1/6*(A4) + 1/6*(A5) + 1/6*(A) + 1/6*(B) + 1/6*(C) + 1/6*(D)$) node {$\scriptstyle 6$} circle[radius=.15cm];

\end{scope}

\begin{scope}[xshift=3cm]
\foreach \x in {0, ..., 6}{
\coordinate (A\x) at (\x*60:1);
}

\coordinate (A) at (-1/2,0);
\coordinate (B) at (-1/3, 1/3);
\coordinate (C) at (1/3,1/3);
\coordinate (D) at (1/2,0);

\draw[line width=1pt, blue] (A0) -- (D);
\draw[line width=1pt, blue] (A1) -- (C);
\draw[line width=1pt, blue] (A2) -- (B);
\draw[line width=1pt, blue] (A3) -- (A);
\draw[line width=1pt, blue] (A4) -- (A);
\draw[line width=1pt, blue] (A5) -- (D);

\draw[line width=1pt, red] (A) -- (B) -- (C) -- (D);

\node at  ($1/3*(A3) + 1/3*(A4) + 1/3*(A)$) {$3$};
\node at   ($1/3*(A5) + 1/3*(A0) + 1/3*(D)$) {$3$};
\node at  ($1/4*(A2) + 1/4*(A3) + 1/4*(A) + 1/4*(B)$) {$4$};
\node at   ($1/4*(A2) + 1/4*(A1) + 1/4*(C) + 1/4*(B)$)  {$4$};
\node at   ($1/4*(A0) + 1/4*(A1) + 1/4*(C) + 1/4*(D)$) {$4$};
\node at   ($1/6*(A4) + 1/6*(A5) + 1/6*(A) + 1/6*(B) + 1/6*(C) + 1/6*(D)$)  {$6$};
\node at (0, -1) {$(011102)$};
\end{scope}

\begin{scope}[xshift=6cm]

\foreach \x in {0, ..., 6}{
\coordinate (A\x) at (-\x*60 -90:1);
}

\draw[line width=1pt] (A0) -- (A1) -- (A2) -- (A3) -- (A4) -- (A5) -- (A6) (A2) -- (A6) -- (A3) (A6) -- (A4);
\foreach \x in {1, ..., 6}{
\filldraw[fill=white] (A\x) circle[radius=.2];
\node at (A\x) {$\x$};
}

\end{scope}
\end{tikzpicture}

\begin{tikzpicture}
\begin{scope}
\foreach \x in {0, ..., 6}{
\coordinate (A\x) at (\x*60:1);
}

\draw (A0) -- (A1) -- (A2) -- (A3) -- (A4) -- (A5) -- (A6);
\coordinate (A) at (-1/2,1/4);
\coordinate (B) at (-1/4, 0);
\coordinate (C) at (1/4,0);
\coordinate (D) at (1/2,-1/4);

\fill[yellow] (A3) -- (A) -- (B) --(A4) -- cycle;
\fill[yellow] (A0) -- (D) -- (C) --(A1) -- cycle;
\fill[green] (A1) -- (C) -- (B) -- (A) --(A2) -- cycle;
\fill[green] (A4) -- (B) -- (C) -- (D) --(A5) -- cycle;

\draw[line width=1pt, blue] (A0) -- (D);
\draw[line width=1pt, blue] (A1) -- (C);
\draw[line width=1pt, blue] (A2) -- (A);
\draw[line width=1pt, blue] (A3) -- (A);
\draw[line width=1pt, blue] (A4) -- (B);
\draw[line width=1pt, blue] (A5) -- (D);

\draw[line width=1pt, red] (A) -- (B) -- (C) -- (D);

\draw ($1/3*(A3) + 1/3*(A2) + 1/3*(A)$) node {$\scriptstyle 2$} circle[radius=.15cm];
\draw ($1/3*(A5) + 1/3*(A0) + 1/3*(D)$) node {$\scriptstyle 5$} circle[radius=.15cm];
\draw ($1/4*(A4) + 1/4*(A3) + 1/4*(A) + 1/4*(B)$) node {$\scriptstyle 1$} circle[radius=.15cm];
\draw ($1/4*(A0) + 1/4*(A1) + 1/4*(C) + 1/4*(D)$) node {$\scriptstyle 4$} circle[radius=.15cm];
\draw ($1/5*(A1) + 1/5*(A2) + 1/5*(A) + 1/5*(B) + 1/5*(C)$) node {$\scriptstyle 3$} circle[radius=.15cm];
\draw ($1/5*(A4) + 1/5*(A5) + 1/5*(D) + 1/5*(B) + 1/5*(C)$) node {$\scriptstyle 6$} circle[radius=.15cm];

\end{scope}

\begin{scope}[xshift=3cm]
\foreach \x in {0, ..., 6}{
\coordinate (A\x) at (\x*60:1);
}

\coordinate (A) at (-1/2,1/4);
\coordinate (B) at (-1/4, 0);
\coordinate (C) at (1/4,0);
\coordinate (D) at (1/2,-1/4);

\draw[line width=1pt, blue] (A0) -- (D);
\draw[line width=1pt, blue] (A1) -- (C);
\draw[line width=1pt, blue] (A2) -- (A);
\draw[line width=1pt, blue] (A3) -- (A);
\draw[line width=1pt, blue] (A4) -- (B);
\draw[line width=1pt, blue] (A5) -- (D);

\draw[line width=1pt, red] (A) -- (B) -- (C) -- (D);

\draw ($1/3*(A3) + 1/3*(A2) + 1/3*(A)$) node {$3$};
\draw ($1/3*(A5) + 1/3*(A0) + 1/3*(D)$) node {$3$};
\draw ($1/4*(A4) + 1/4*(A3) + 1/4*(A) + 1/4*(B)$) node {$4$};
\draw ($1/4*(A0) + 1/4*(A1) + 1/4*(C) + 1/4*(D)$) node {$4$};
\draw ($1/5*(A1) + 1/5*(A2) + 1/5*(A) + 1/5*(B) + 1/5*(C)$) node {$5$};
\draw ($1/5*(A4) + 1/5*(A5) + 1/5*(D) + 1/5*(B) + 1/5*(C)$) node{$5$};
\node at (0, -1) {$(021021)$};
\end{scope}

\begin{scope}[xshift=6cm]

\foreach \x in {0, ..., 6}{
\coordinate (A\x) at (-\x*60 -90:1);
}

\draw[line width=1pt] (A0) -- (A1) -- (A2) -- (A3) -- (A4) -- (A5) -- (A6) (A1) -- (A3) -- (A6) -- (A4);
\foreach \x in {1, ..., 6}{
\filldraw[fill=white] (A\x) circle[radius=.2];
\node at (A\x) {$\x$};
}

\end{scope}
\end{tikzpicture}

\begin{tikzpicture}
\begin{scope}
\foreach \x in {0, ..., 6}{
\coordinate (A\x) at (\x*60:1);
}

\draw (A0) -- (A1) -- (A2) -- (A3) -- (A4) -- (A5) -- (A6);
\coordinate (A) at (0,0);
\coordinate (B) at (30:1/2);
\coordinate (C) at (150:1/2);
\coordinate (D) at (-90:1/2);

\fill[green] (A1) -- (B) -- (A) -- (C) --(A2) -- cycle;
\fill[green] (A4) -- (A3) -- (C) -- (A) -- (D) -- cycle;
\fill[green] (A5) -- (D) -- (A) -- (B) --(A0) -- cycle;

\draw[line width=1pt, blue] (A0) -- (B);
\draw[line width=1pt, blue] (A1) -- (B);
\draw[line width=1pt, blue] (A2) -- (C);
\draw[line width=1pt, blue] (A3) -- (C);
\draw[line width=1pt, blue] (A4) -- (D);
\draw[line width=1pt, blue] (A5) -- (D);

\draw[line width=1pt, red] (A) -- (B) (A) -- (C) (A) -- (D);

\draw ($1/3*(A3) + 1/3*(A2) + 1/3*(C)$) node {$\scriptstyle 2$} circle[radius=.15cm];
\draw ($1/3*(A5) + 1/3*(A4) + 1/3*(D)$) node {$\scriptstyle 6$} circle[radius=.15cm];
\draw ($1/3*(A0) + 1/3*(A1) + 1/3*(B)$) node {$\scriptstyle 4$} circle[radius=.15cm];
\draw ($1/5*(A1) + 1/5*(A2) + 1/5*(A) + 1/5*(B) + 1/5*(C)$) node {$\scriptstyle 3$} circle[radius=.15cm];
\draw ($1/5*(A4) + 1/5*(A3) + 1/5*(A) + 1/5*(D) + 1/5*(C)$) node {$\scriptstyle 1$} circle[radius=.15cm];
\draw ($1/5*(A5) + 1/5*(A0) + 1/5*(A) + 1/5*(D) + 1/5*(B)$) node {$\scriptstyle 5$} circle[radius=.15cm];

\end{scope}

\begin{scope}[xshift=3cm]
\foreach \x in {0, ..., 6}{
\coordinate (A\x) at (\x*60:1);
}

\coordinate (A) at (0,0);
\coordinate (B) at (30:1/2);
\coordinate (C) at (150:1/2);
\coordinate (D) at (-90:1/2);

\draw[line width=1pt, blue] (A0) -- (B);
\draw[line width=1pt, blue] (A1) -- (B);
\draw[line width=1pt, blue] (A2) -- (C);
\draw[line width=1pt, blue] (A3) -- (C);
\draw[line width=1pt, blue] (A4) -- (D);
\draw[line width=1pt, blue] (A5) -- (D);

\draw[line width=1pt, red] (A) -- (B) (A) -- (C) (A) -- (D);

\draw ($1/3*(A3) + 1/3*(A2) + 1/3*(C)$) node {$3$};
\draw ($1/3*(A5) + 1/3*(A4) + 1/3*(D)$) node {$3$};
\draw ($1/3*(A0) + 1/3*(A1) + 1/3*(B)$) node {$3$};
\draw ($1/5*(A1) + 1/5*(A2) + 1/5*(A) + 1/5*(B) + 1/5*(C)$) node {$5$};
\draw ($1/5*(A4) + 1/5*(A3) + 1/5*(A) + 1/5*(D) + 1/5*(C)$) node {$5$};
\draw ($1/5*(A5) + 1/5*(A0) + 1/5*(A) + 1/5*(D) + 1/5*(B)$) node {$5$};
\node at (0, -1.1) {$(020202)$};
\end{scope}

\begin{scope}[xshift=6cm]

\foreach \x in {0, ..., 6}{
\coordinate (A\x) at (-\x*60 -90:1);
}

\draw[line width=1pt] (A0) -- (A1) -- (A2) -- (A3) -- (A4) -- (A5) -- (A6) (A1) -- (A3) -- (A5) -- (A1);
\foreach \x in {1, ..., 6}{
\filldraw[fill=white] (A\x) circle[radius=.2];
\node at (A\x) {$\x$};
}

\end{scope}
\end{tikzpicture}

\caption{Essential and core trees, numerical label  and  defining graph $G_{l\mathbf{Y}_6}$  for $\mathbf{Y}_6$.
The reflection of the second gives another orbit: $(012012)$. }
\label{pyre61}
\end{center}
\end{figure}
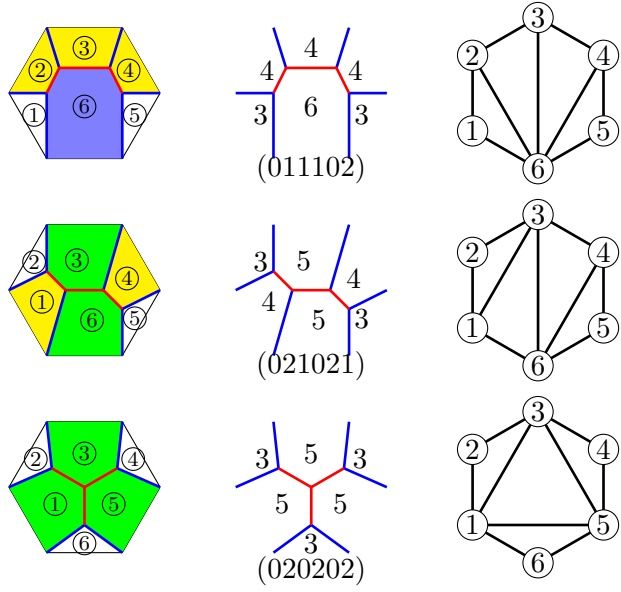

  The number of triangulations of a regular $n$-polygon modulo the action of the cyclic group is computed  in \cite{BR2012}. A reference for the integer sequence $\{N_n\}$ is $A001683$ in OEIS \cite{OEIS25}. The first values are: $N_4=1$, $N_5=1$, $N_6=4$, $N_7=6$, $N_8= 19$, see Fig. \ref{pyre4}, \ref{pyre5}, \ref{pyre61}, \ref{pyre71} and \ref{pyre81} where we are using different colors for faces in order to help identifying polygons as in Dutch web page \cite{Dutch20}. Triangles: white; quadrilaterals: yellow; pentagons: green; hexagons: light blue, \dots

\begin{figure}[ht]
\begin{center}
\begin{tikzpicture}
\begin{scope}
\foreach \x in {0, ..., 7}{
\coordinate (A\x) at (\x*360/7- 90/7:1);
}

\draw (A0) -- (A1) -- (A2) -- (A3) -- (A4) -- (A5) -- (A6) -- (A7);
\coordinate (A) at (180:1/2);
\coordinate (B) at (135:1/2);
\coordinate (C) at (90:1/2);
\coordinate (D) at (45:1/2);
\coordinate (E) at (0:1/2);

\fill[yellow] (A4) -- (A) -- (B) --(A3) -- cycle;
\fill[yellow] (A3) -- (B) -- (C) --(A2) -- cycle;
\fill[yellow] (A2) -- (C) -- (D) --(A1) -- cycle;
\fill[yellow] (A1) -- (D) -- (E) --(A0) -- cycle;
\fill[purple!50!white] (A5) -- (A) -- (B) -- (C) -- (D) -- (E) --(A6) -- cycle;

\draw[line width=1pt, blue] (A0) -- (E);
\draw[line width=1pt, blue] (A1) -- (D);
\draw[line width=1pt, blue] (A2) -- (C);
\draw[line width=1pt, blue] (A3) -- (B);
\draw[line width=1pt, blue] (A4) -- (A);
\draw[line width=1pt, blue] (A5) -- (A);
\draw[line width=1pt, blue] (A6) -- (E);

\draw[line width=1pt, red] (A) -- (B) -- (C)  -- (D) -- (E);

\draw ($1/3*(A5) + 1/3*(A4) + 1/3*(A)$) node {$\scriptstyle 1$} circle[radius=.15cm];
\draw ($1/3*(A0) + 1/3*(A6) + 1/3*(E)$) node {$\scriptstyle 6$} circle[radius=.15cm];
\draw ($1/4*(A1) + 1/4*(A0) + 1/4*(E) + 1/4*(D)$) node {$\scriptstyle 5$} circle[radius=.15cm];
\draw ($1/4*(A2) + 1/4*(A1) + 1/4*(C) + 1/4*(D)$) node {$\scriptstyle 4$} circle[radius=.15cm];
\draw ($1/4*(A3) + 1/4*(A2) + 1/4*(C) + 1/4*(B)$) node {$\scriptstyle 3$} circle[radius=.15cm];
\draw ($1/4*(A4) + 1/4*(A3) + 1/4*(B) + 1/4*(A)$) node {$\scriptstyle 2$} circle[radius=.15cm];
\draw ($1/7*(A5) + 1/7*(A6) + 1/7*(E) + 1/7*(D) + 1/7*(C) + 1/7*(B) + 1/7*(A)$) node {$\scriptstyle 7$} circle[radius=.15cm];

\end{scope}

\begin{scope}[xshift=3cm]
\foreach \x in {0, ..., 7}{
\coordinate (A\x) at (\x*360/7- 90/7:1);
}

\coordinate (A) at (180:1/2);
\coordinate (B) at (135:1/2);
\coordinate (C) at (90:1/2);
\coordinate (D) at (45:1/2);
\coordinate (E) at (0:1/2);

\draw[line width=1pt, blue] (A0) -- (E);
\draw[line width=1pt, blue] (A1) -- (D);
\draw[line width=1pt, blue] (A2) -- (C);
\draw[line width=1pt, blue] (A3) -- (B);
\draw[line width=1pt, blue] (A4) -- (A);
\draw[line width=1pt, blue] (A5) -- (A);
\draw[line width=1pt, blue] (A6) -- (E);

\draw[line width=1pt, red] (A) -- (B) -- (C)  -- (D) -- (E);

\draw ($1/3*(A5) + 1/3*(A4) + 1/3*(A)$) node {$3$};
\draw ($1/3*(A0) + 1/3*(A6) + 1/3*(E)$) node {$3$};
\draw ($1/4*(A1) + 1/4*(A0) + 1/4*(E) + 1/4*(D)$) node {$4$};
\draw ($1/4*(A2) + 1/4*(A1) + 1/4*(C) + 1/4*(D)$) node {$4$};
\draw ($1/4*(A3) + 1/4*(A2) + 1/4*(C) + 1/4*(B)$) node {$4$};
\draw ($1/4*(A4) + 1/4*(A3) + 1/4*(B) + 1/4*(A)$) node {$4$};
\draw ($1/7*(A5) + 1/7*(A6) + 1/7*(E) + 1/7*(D) + 1/7*(C) + 1/7*(B) + 1/7*(A)$) node {$7$};
\node at (0, -1.1) {$(011104)$};
\end{scope}

\begin{scope}[xshift=6cm]

\foreach \x in {0, ..., 7}{
\coordinate (A\x) at (-\x*360 / 7 -90:1);
}

\draw[line width=1pt] (A0) -- (A1) -- (A2) -- (A3) -- (A4) -- (A5) -- (A6) -- (A7) (A7) -- (A2) (A7) -- (A3) (A7)-- (A4) (A7) -- (A5);
\foreach \x in {1, ..., 7}{
\filldraw[fill=white] (A\x) circle[radius=.2];
\node at (A\x) {$\x$};
}

\end{scope}
\end{tikzpicture}

\begin{tikzpicture}
\begin{scope}
\foreach \x in {0, ..., 7}{
\coordinate (A\x) at (\x*360/7- 90/7:1);
}

\draw (A0) -- (A1) -- (A2) -- (A3) -- (A4) -- (A5) -- (A6) -- (A7);
\coordinate (A) at (180:1/2);
\coordinate (B) at (135:1/2);
\coordinate (C) at (90:1/2);
\coordinate (D) at (-45:1/2);
\coordinate (E) at (0:1/2);

\fill[yellow] (A4) -- (A) -- (B) --(A3) -- cycle;
\fill[yellow] (A3) -- (B) -- (C) --(A2) -- cycle;
\fill[yellow] (A6) -- (D) -- (E) --(A0) -- cycle;

\fill[green] (A2) -- (C) -- (D) -- (E)--(A1) -- cycle;

\fill[blue!50!white] (A5) -- (A) -- (B) -- (C) -- (D) --(A6) -- cycle;

\draw[line width=1pt, blue] (A0) -- (E);
\draw[line width=1pt, blue] (A1) -- (E);
\draw[line width=1pt, blue] (A2) -- (C);
\draw[line width=1pt, blue] (A3) -- (B);
\draw[line width=1pt, blue] (A4) -- (A);
\draw[line width=1pt, blue] (A5) -- (A);
\draw[line width=1pt, blue] (A6) -- (D);

\draw[line width=1pt, red] (A) -- (B) -- (C)  -- (D) -- (E);

\draw ($1/3*(A5) + 1/3*(A4) + 1/3*(A)$) node {$\scriptstyle 1$} circle[radius=.15cm];
\draw ($1/3*(A0) + 1/3*(A1) + 1/3*(E)$) node {$\scriptstyle 5$} circle[radius=.15cm];
\draw ($1/4*(A6) + 1/4*(A0) + 1/4*(E) + 1/4*(D)$) node {$\scriptstyle 6$} circle[radius=.15cm];
\draw ($1/5*(A2) + 1/5*(A1) + 1/5*(C) + 1/5*(D) + 1/5*(E)$) node {$\scriptstyle 4$} circle[radius=.15cm];
\draw ($1/4*(A3) + 1/4*(A2) + 1/4*(C) + 1/4*(B)$) node {$\scriptstyle 3$} circle[radius=.15cm];
\draw ($1/4*(A4) + 1/4*(A3) + 1/4*(B) + 1/4*(A)$) node {$\scriptstyle 2$} circle[radius=.15cm];
\draw ($1/6*(A5) + 1/6*(A6) + 1/6*(D) + 1/6*(C) + 1/6*(B) + 1/6*(A)$) node {$\scriptstyle 7$} circle[radius=.15cm];

\end{scope}

\begin{scope}[xshift=3cm]
\foreach \x in {0, ..., 7}{
\coordinate (A\x) at (\x*360/7- 90/7:1);
}

\coordinate (A) at (180:1/2);
\coordinate (B) at (135:1/2);
\coordinate (C) at (90:1/2);
\coordinate (D) at (-45:1/2);
\coordinate (E) at (0:1/2);

\draw[line width=1pt, blue] (A0) -- (E);
\draw[line width=1pt, blue] (A1) -- (E);
\draw[line width=1pt, blue] (A2) -- (C);
\draw[line width=1pt, blue] (A3) -- (B);
\draw[line width=1pt, blue] (A4) -- (A);
\draw[line width=1pt, blue] (A5) -- (A);
\draw[line width=1pt, blue] (A6) -- (D);

\draw[line width=1pt, red] (A) -- (B) -- (C)  -- (D) -- (E);

\draw ($1/3*(A5) + 1/3*(A4) + 1/3*(A)$) node {$3$};
\draw ($1/3*(A0) + 1/3*(A1) + 1/3*(E)$) node {$3$};
\draw ($1/4*(A6) + 1/4*(A0) + 1/4*(E) + 1/4*(D)$) node {$4$};
\draw ($1/5*(A2) + 1/5*(A1) + 1/5*(C) + 1/5*(D) + 1/5*(E)$) node {$5$};
\draw ($1/4*(A3) + 1/4*(A2) + 1/4*(C) + 1/4*(B)$) node {$4$};
\draw ($1/4*(A4) + 1/4*(A3) + 1/4*(B) + 1/4*(A)$) node {$4$};
\draw ($1/6*(A5) + 1/6*(A6) + 1/6*(D) + 1/6*(C) + 1/6*(B) + 1/6*(A)$) node {$6$};

\node at (0, -1.1) {$(0112013)$};
\end{scope}

\begin{scope}[xshift=6cm]

\foreach \x in {0, ..., 7}{
\coordinate (A\x) at (-\x*360 / 7 -90:1);
}

\draw[line width=1pt] (A0) -- (A1) -- (A2) -- (A3) -- (A4) -- (A5) -- (A6) -- (A7) (A7) -- (A2) (A7) -- (A3) (A7)-- (A4) -- (A6);
\foreach \x in {1, ..., 7}{
\filldraw[fill=white] (A\x) circle[radius=.2];
\node at (A\x) {$\x$};
}

\end{scope}
\end{tikzpicture}

\begin{tikzpicture}
\begin{scope}
\foreach \x in {0, ..., 7}{
\coordinate (A\x) at (\x*360/7- 90/7:1);
}

\draw (A0) -- (A1) -- (A2) -- (A3) -- (A4) -- (A5) -- (A6) -- (A7);
\coordinate (A) at (180:1/2);
\coordinate (B) at (-135:1/2);
\coordinate (C) at (90:1/2);
\coordinate (D) at (-45:1/2);
\coordinate (E) at (0:1/2);

\fill[yellow] (A4) -- (A) -- (B) --(A5) -- cycle;
\fill[yellow] (A6) -- (D) -- (E) --(A0) -- cycle;

\fill[green] (A2) -- (C) -- (D) -- (E)--(A1) -- cycle;
\fill[green] (A3) -- (A) -- (B) -- (C)--(A2) -- cycle;
\fill[green] (A5) -- (B) -- (C) -- (D) -- (A6) -- cycle;

\draw[line width=1pt, blue] (A0) -- (E);
\draw[line width=1pt, blue] (A1) -- (E);
\draw[line width=1pt, blue] (A2) -- (C);
\draw[line width=1pt, blue] (A3) -- (A);
\draw[line width=1pt, blue] (A4) -- (A);
\draw[line width=1pt, blue] (A5) -- (B);
\draw[line width=1pt, blue] (A6) -- (D);

\draw[line width=1pt, red] (A) -- (B) -- (C)  -- (D) -- (E);

\draw ($1/3*(A3) + 1/3*(A4) + 1/3*(A)$) node {$\scriptstyle 2$} circle[radius=.15cm];
\draw ($1/3*(A0) + 1/3*(A1) + 1/3*(E)$) node {$\scriptstyle 5$} circle[radius=.15cm];
\draw ($1/4*(A5) + 1/4*(A4) + 1/4*(A) + 1/4*(B)$) node {$\scriptstyle 1$} circle[radius=.15cm];
\draw ($1/4*(A0) + 1/4*(A6) + 1/4*(D) + 1/4*(E)$) node {$\scriptstyle 6$} circle[radius=.15cm];
\draw ($1/5*(A2) + 1/5*(A1) + 1/5*(C) + 1/5*(D) + 1/5*(E)$) node {$\scriptstyle 4$} circle[radius=.15cm];
\draw ($1/5*(A2) + 1/5*(A3) + 1/5*(C) + 1/5*(B) + 1/5*(A)$) node {$\scriptstyle 3$} circle[radius=.15cm];
\draw ($1/5*(A5) + 1/5*(A6) + 1/5*(C) + 1/5*(B) + 1/5*(D)$) node {$\scriptstyle 7$} circle[radius=.15cm];

\end{scope}

\begin{scope}[xshift=3cm]
\foreach \x in {0, ..., 7}{
\coordinate (A\x) at (\x*360/7- 90/7:1);
}

\coordinate (A) at (180:1/2);
\coordinate (B) at (-135:1/2);
\coordinate (C) at (90:1/2);
\coordinate (D) at (-45:1/2);
\coordinate (E) at (0:1/2);

\draw[line width=1pt, blue] (A0) -- (E);
\draw[line width=1pt, blue] (A1) -- (E);
\draw[line width=1pt, blue] (A2) -- (C);
\draw[line width=1pt, blue] (A3) -- (A);
\draw[line width=1pt, blue] (A4) -- (A);
\draw[line width=1pt, blue] (A5) -- (B);
\draw[line width=1pt, blue] (A6) -- (D);

\draw[line width=1pt, red] (A) -- (B) -- (C)  -- (D) -- (E);

\draw ($1/3*(A3) + 1/3*(A4) + 1/3*(A)$) node {$3$};
\draw ($1/3*(A0) + 1/3*(A1) + 1/3*(E)$) node {$3$};
\draw ($1/4*(A5) + 1/4*(A4) + 1/4*(A) + 1/4*(B)$) node {$4$};
\draw ($1/4*(A0) + 1/4*(A6) + 1/4*(D) + 1/4*(E)$) node {$4$};
\draw ($1/5*(A2) + 1/5*(A1) + 1/5*(C) + 1/5*(D) + 1/5*(E)$) node {$5$};
\draw ($1/5*(A2) + 1/5*(A3) + 1/5*(C) + 1/5*(B) + 1/5*(A)$) node {$5$};
\draw ($1/5*(A5) + 1/5*(A6) + 1/5*(C) + 1/5*(B) + 1/5*(D)$) node {$5$};

\node at (0, -1.1) {$(0220121)$};
\end{scope}

\begin{scope}[xshift=6cm]

\foreach \x in {0, ..., 7}{
\coordinate (A\x) at (-\x*360 / 7 -90:1);
}

\draw[line width=1pt] (A0) -- (A1) -- (A2) -- (A3) -- (A4) -- (A5) -- (A6) -- (A7) (A1) -- (A3) -- (A7) -- (A4) -- (A6);
\foreach \x in {1, ..., 7}{
\filldraw[fill=white] (A\x) circle[radius=.2];
\node at (A\x) {$\x$};
}

\end{scope}
\end{tikzpicture}

\begin{tikzpicture}
\begin{scope}
\foreach \x in {0, ..., 7}{
\coordinate (A\x) at (\x*360/7- 90/7:1);
}

\draw (A0) -- (A1) -- (A2) -- (A3) -- (A4) -- (A5) -- (A6) -- (A7);
\coordinate (A) at (180:1/3);
\coordinate (B) at (120:1/3);
\coordinate (C) at (120:2/3);
\coordinate (D) at (60:1/3);
\coordinate (E) at (0:1/3);

\fill[yellow] (A0) -- (E) -- (D) --(A1) -- cycle;

\fill[green] (A1) -- (D) -- (B) -- (C)--(A2) -- cycle;
\fill[green] (A3) -- (C) -- (B) -- (A)--(A4) -- cycle;

\fill[blue!50!white] (A5) -- (A) -- (B) -- (D) -- (E) -- (A6) -- cycle;

\draw[line width=1pt, blue] (A0) -- (E);
\draw[line width=1pt, blue] (A1) -- (D);
\draw[line width=1pt, blue] (A2) -- (C);
\draw[line width=1pt, blue] (A3) -- (C);
\draw[line width=1pt, blue] (A4) -- (A);
\draw[line width=1pt, blue] (A5) -- (A);
\draw[line width=1pt, blue] (A6) -- (E);

\draw[line width=1pt, red] (A) -- (B) -- (D) -- (E) (B) -- (C);

\draw ($1/3*(A5) + 1/3*(A4) + 1/3*(A)$) node {$\scriptstyle 1$} circle[radius=.15cm];
\draw ($1/3*(A0) + 1/3*(A6) + 1/3*(E)$) node {$\scriptstyle 6$} circle[radius=.15cm];
\draw ($1/3*(A3) + 1/3*(A2) + 1/3*(C)$) node {$\scriptstyle 3$} circle[radius=.15cm];
\draw ($1/4*(A0) + 1/4*(A1) + 1/4*(D) + 1/4*(E)$) node {$\scriptstyle 5$} circle[radius=.15cm];
\draw ($1/5*(A2) + 1/5*(A1) + 1/5*(C) + 1/5*(D) + 1/5*(B)$) node {$\scriptstyle 4$} circle[radius=.15cm];
\draw ($1/5*(A4) + 1/5*(A3) + 1/5*(C) + 1/5*(B) + 1/5*(A)$) node {$\scriptstyle 2$} circle[radius=.15cm];
\draw ($1/6*(A5) + 1/6*(A6) + 1/6*(A) + 1/6*(E) + 1/6*(B) + 1/6*(D)$) node {$\scriptstyle 7$} circle[radius=.15cm];

\end{scope}

\begin{scope}[xshift=3cm]
\foreach \x in {0, ..., 7}{
\coordinate (A\x) at (\x*360/7- 90/7:1);
}

\coordinate (A) at (180:1/3);
\coordinate (B) at (120:1/3);
\coordinate (C) at (120:2/3);
\coordinate (D) at (60:1/3);
\coordinate (E) at (0:1/3);

\draw[line width=1pt, blue] (A0) -- (E);
\draw[line width=1pt, blue] (A1) -- (D);
\draw[line width=1pt, blue] (A2) -- (C);
\draw[line width=1pt, blue] (A3) -- (C);
\draw[line width=1pt, blue] (A4) -- (A);
\draw[line width=1pt, blue] (A5) -- (A);
\draw[line width=1pt, blue] (A6) -- (E);

\draw[line width=1pt, red] (A) -- (B) -- (D) -- (E) (B) -- (C);

\draw ($1/3*(A5) + 1/3*(A4) + 1/3*(A)$) node {$3$};
\draw ($1/3*(A0) + 1/3*(A6) + 1/3*(E)$) node {$3$};
\draw ($1/3*(A3) + 1/3*(A2) + 1/3*(C)$) node {$3$};
\draw ($1/4*(A0) + 1/4*(A1) + 1/4*(D) + 1/4*(E)$) node {$4$};
\draw ($1/5*(A2) + 1/5*(A1) + 1/5*(C) + 1/5*(D) + 1/5*(B)$) node {$5$};
\draw ($1/5*(A4) + 1/5*(A3) + 1/5*(C) + 1/5*(B) + 1/5*(A)$) node {$5$};
\draw ($1/6*(A5) + 1/6*(A6) + 1/6*(A) + 1/6*(E) + 1/6*(B) + 1/6*(D)$) node {$6$};

\node at (0, -1.1) {$(0202103)$};
\end{scope}

\begin{scope}[xshift=6cm]

\foreach \x in {0, ..., 7}{
\coordinate (A\x) at (-\x*360 / 7 -90:1);
}

\draw[line width=1pt] (A0) -- (A1) -- (A2) -- (A3) -- (A4) -- (A5) -- (A6) -- (A7) (A2) -- (A7) -- (A3) -- (A5) -- (A7);
\foreach \x in {1, ..., 7}{
\filldraw[fill=white] (A\x) circle[radius=.2];
\node at (A\x) {$\x$};
}

\end{scope}
\end{tikzpicture}

\caption{Essential and core trees and numerical label  and defining graph $G_{l\mathbf{Y}_7}$  for $\mathbf{Y}_7$.
The vertical reflections of the second and fourth cases give another orbits: $(0310211)$ and $(0120203)$.}
\label{pyre71}
\end{center}
\end{figure}
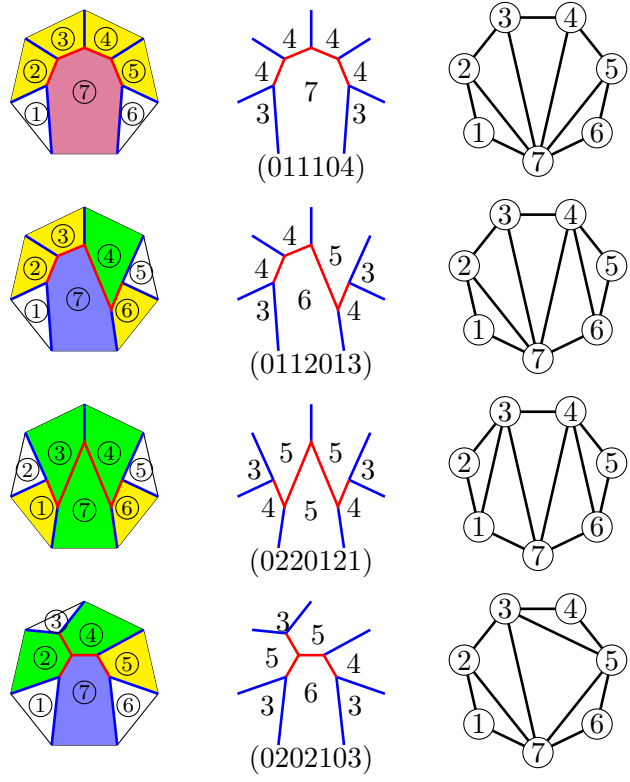

 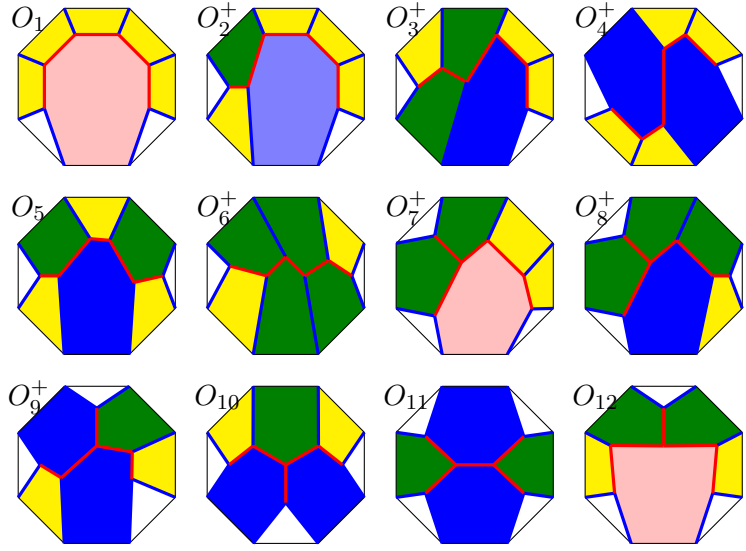
\begin{figure}[ht]
\begin{center}
\begin{tikzpicture}
\newcommand\octogono[1]{\node[regular polygon, draw, regular polygon sides=8, minimum size=#1cm] (p) at (0,0) {}
}
\begin{scope}
\foreach \x in {0, ..., 7}
{
\coordinate (A\x) at ({45*\x - 22.5}:1.125);
\coordinate (B\x) at (45*\x - 22.5:.75);
}
\fill[pink] (A7) -- (B0) -- (B1) -- (B2) -- (B3) -- (B4) -- (B5) -- (A6);
\foreach \x[evaluate= \x as \y using \x + 1] in {0,...,4}
{
\fill[yellow] (A\x) -- (A\y) -- (B\y) -- (B\x);
}
\octogono{2.25};
\foreach \x in {0,...,5}
{
\draw[blue, line width=1.2pt] (A\x) -- (B\x);
}
\foreach \x[evaluate= \x as \y using \x + 1] in {0,...,4}
{
\draw[red, line width=1.2pt] (B\x) -- (B\y);
}
\draw[blue, line width=1.2pt] (A7) -- (B0);
\draw[blue, line width=1.2pt] (A6) -- (B5);
\node at (-.9,.9) {$O_1$};
\end{scope}

\begin{scope}[xshift=2.5cm]
\foreach \x in {0, ..., 7}
{
\coordinate (A\x) at ({45*\x - 22.5}:1.125);
\coordinate (B\x) at (45*\x - 22.5:.75);
}
\coordinate(B4) at (180:.5);
\coordinate(B5) at (180:.75);
\fill[blue!50!white] (A7) -- (B0) -- (B1) -- (B2) -- (B3) -- (B4) -- (A6);
\foreach \x[evaluate= \x as \y using \x + 1] in {0,...,2}
{
\fill[yellow] (A\x) -- (A\y) -- (B\y) -- (B\x);
}
\fill[yellow] (A6) -- (B4) -- (B5) -- (A5);
\fill[green!50!black] (B5) -- (B4) -- (B3) -- (A3) -- (A4);
\octogono{2.25};
\foreach \x in {0,...,3, 5}
{
\draw[blue, line width=1.2pt] (A\x) -- (B\x);
}
\draw[blue, line width=1.2pt] (B5) -- (A4);
\foreach \x[evaluate= \x as \y using \x + 1] in {0,...,4}
{
\draw[red, line width=1.2pt] (B\x) -- (B\y);
}
\draw[blue, line width=1.2pt] (A7) -- (B0);
\draw[blue, line width=1.2pt] (A6) -- (B4);
\node at (-.9,.9) {$O_2^+$};
\end{scope}

\begin{scope}[xshift=5cm]
\foreach \x in {0, ..., 7}
{
\coordinate (A\x) at ({45*\x - 22.5}:1.125);
\coordinate (B\x) at (45*\x - 22.5:.75);
}
\coordinate(B3) at (150:.125);
\coordinate(B4) at (150:.5);
\coordinate(B5) at (180:.75);
\fill[blue] (A7) -- (B0) -- (B1) -- (B2) -- (B3) -- (B4) -- (A6);
\foreach \x[evaluate= \x as \y using \x + 1] in {0,1}
{
\fill[yellow] (A\x) -- (A\y) -- (B\y) -- (B\x);
}
\fill[yellow] (A3) -- (A4) -- (B5) -- (B4);
\fill[green!50!black]  (B4) -- (B3) -- (B2) -- (A2) -- (A3);
\fill[green!50!black]  (B3) -- (B4) -- (B5) -- (A5) -- (A6);
\octogono{2.25};
\foreach \x in {0,...,2, 5}
{
\draw[blue, line width=1.2pt] (A\x) -- (B\x);
}
\draw[blue, line width=1.2pt] (B4) -- (A3);
\draw[blue, line width=1.2] (B5) -- (A4);
\foreach \x[evaluate= \x as \y using \x + 1] in {0,...,4}
{
\draw[red, line width=1.2pt] (B\x) -- (B\y);
}
\draw[blue, line width=1.2pt] (A7) -- (B0);
\draw[blue, line width=1.2pt] (A6) -- (B3);
\node at (-.9,.9) {$O_3^+$};
\end{scope}
\begin{scope}[xshift=7.5cm]
\foreach \x in {0, ..., 7}
{
\coordinate (A\x) at ({45*\x - 22.5}:1.125);
\coordinate (B\x) at (45*\x - 22.5:.75);
}
\coordinate(B3) at (90:.5);
\coordinate(B4) at (-90:.5);
\fill[blue] (A7) -- (A0) -- (B1) -- (B2) -- (B3) -- (B4) ;
\fill[blue] (A3) -- (A4) -- (B5) -- (B6) -- (B4) -- (B3) ;
\foreach \x[evaluate= \x as \y using \x + 1] in {1,2,5}
{
\fill[yellow] (A\x) -- (A\y) -- (B\y) -- (B\x);
}
\fill[yellow] (A6) -- (A7) -- (B4) -- (B6);
\octogono{2.25};
\foreach \x in {1, 2, 3, 5, 6}
{
\draw[blue, line width=1.2pt] (A\x) -- (B\x);
}
\draw[blue, line width=1.2pt] (B1) -- (A0);
\draw[blue, line width=1.2] (B5) -- (A4);
\draw[blue, line width=1.2] (B4) -- (A7);
\foreach \x[evaluate= \x as \y using \x + 1] in {1,...,3}
{
\draw[red, line width=1.2pt] (B\x) -- (B\y);
}
\draw[red, line width=1.2pt] (B4) -- (B6);
\draw[red, line width=1.2pt] (B5) -- (B6);
\node at (-.9,.9) {$O_4^+$};
\end{scope}
\begin{scope}[xshift=0cm, yshift=-2.5cm]
\foreach \x in {0, ..., 7}
{
\coordinate (A\x) at ({45*\x - 22.5}:1.125);
\coordinate (B\x) at (45*\x - 22.5:.75);
}
\coordinate(B0) at (0:.875);
\coordinate(B1) at (-10:.5);
\coordinate(B2) at (70:.5);
\coordinate(B3) at (100:.5);
\coordinate(B4) at (180:.5);
\coordinate(B5) at (180:.75);
\fill[blue] (A6) -- (A7) -- (B1) -- (B2) -- (B3) -- (B4) ;
\fill[yellow] (A7) -- (A0) -- (B0) -- (B1) ;
\fill[yellow] (A2) -- (A3) -- (B3) -- (B2) ;
\fill[yellow] (A5) -- (A6) -- (B4) -- (B5) ;
\fill[green!50!black] (A1) -- (A2) -- (B2) -- (B1) -- (B0);
\fill[green!50!black] (A3) -- (A4) -- (B5) -- (B4) -- (B3);
\octogono{2.25};
\draw[blue, line width=1.2pt] (B0) -- (A0);
\draw[blue, line width=1.2pt] (B0) -- (A1);
\draw[blue, line width=1.2] (B2) -- (A2);
\draw[blue, line width=1.2] (B3) -- (A3);
\draw[blue, line width=1.2] (B5) -- (A4);
\draw[blue, line width=1.2] (B5) -- (A5);
\draw[blue, line width=1.2] (B4) -- (A6);
\draw[blue, line width=1.2] (B1) -- (A7);

\foreach \x[evaluate= \x as \y using \x + 1] in {0,...,4}
{
\draw[red, line width=1.2pt] (B\x) -- (B\y);
}

\node at (-.9,.9) {$O_5$};
\end{scope}
\begin{scope}[xshift=2.5cm, yshift=-2.5cm]
\foreach \x in {0, ..., 7}
{
\coordinate (A\x) at ({45*\x - 22.5}:1.125);
\coordinate (B\x) at (45*\x - 22.5:.75);
}
\coordinate(B0) at (0:.875);
\coordinate(B1) at (20:.6);
\coordinate(B2) at (0:.25);
\coordinate(B3) at (90:.25);
\coordinate(B4) at (180:.25);
\coordinate(B5) at (170:.75);
\fill[yellow] (A1) -- (A2) -- (B1) -- (B0) ;
\fill[yellow] (A5) -- (A6) -- (B4) -- (B5) ;
\fill[green!50!black] (A2) -- (A3) -- (B3) -- (B2) -- (B1);
\fill[green!50!black] (A3) -- (A4) -- (B5) -- (B4) -- (B3);
\fill[green!50!black] (A7) -- (A0) -- (B0) -- (B1) -- (B2);
\fill[green!50!black] (A6) -- (A7) -- (B2) -- (B3) -- (B4);
\octogono{2.25};
\draw[blue, line width=1.2pt] (B0) -- (A0);
\draw[blue, line width=1.2pt] (B0) -- (A1);
\draw[blue, line width=1.2] (B1) -- (A2);
\draw[blue, line width=1.2] (B3) -- (A3);
\draw[blue, line width=1.2] (B5) -- (A4);
\draw[blue, line width=1.2] (B5) -- (A5);
\draw[blue, line width=1.2] (B4) -- (A6);
\draw[blue, line width=1.2] (B2) -- (A7);

\foreach \x[evaluate= \x as \y using \x + 1] in {0,...,4}
{
\draw[red, line width=1.2pt] (B\x) -- (B\y);
}

\node at (-.9,.9) {$O_6^+$};
\end{scope}
\begin{scope}[xshift=5cm, yshift=-2.5cm]
\foreach \x in {0, ..., 7}
{
\coordinate (A\x) at ({45*\x - 22.5}:1.125);
\coordinate (B\x) at (45*\x - 22.5:.75);
}
\coordinate(B0) at (-30:.875);
\coordinate(B1) at (0:.65);
\coordinate(B2) at (70:.5);
\coordinate(B3) at (135:.25);
\coordinate(B4) at (135:.75);
\coordinate(B5) at (-135:.75);
\fill[yellow] (A0) -- (A1) -- (B1) -- (B0) ;
\fill[yellow] (A1) -- (A2) -- (B2) -- (B1) ;
\fill[green!50!black] (A2) -- (A3) -- (B4) -- (B3) -- (B2);
\fill[green!50!black] (A4) -- (A5) -- (B5) -- (B3) -- (B4);
\fill[pink] (A6) -- (A7) -- (B0) -- (B1) -- (B2) -- (B3) -- (B5);
\octogono{2.25};
\draw[blue, line width=1.2pt] (B0) -- (A0);
\draw[blue, line width=1.2pt] (B1) -- (A1);
\draw[blue, line width=1.2] (B2) -- (A2);
\draw[blue, line width=1.2] (B4) -- (A3);
\draw[blue, line width=1.2] (B4) -- (A4);
\draw[blue, line width=1.2] (B5) -- (A5);
\draw[blue, line width=1.2] (B5) -- (A6);
\draw[blue, line width=1.2] (B0) -- (A7);

\foreach \x[evaluate= \x as \y using \x + 1] in {0,...,3}
{
\draw[red, line width=1.2pt] (B\x) -- (B\y);
}
\draw[red, line width=1.2pt] (B3) -- (B5);

\node at (-.9,.9) {$O_7^+$};
\end{scope}
\begin{scope}[xshift=7.5cm, yshift=-2.5cm]
\foreach \x in {0, ..., 7}
{
\coordinate (A\x) at ({45*\x - 22.5}:1.125);
\coordinate (B\x) at (45*\x - 22.5:.75);
}
\coordinate(B0) at (0:.875);
\coordinate(B1) at (0:.65);
\coordinate(B2) at (70:.5);
\coordinate(B3) at (135:.25);
\coordinate(B4) at (135:.75);
\coordinate(B5) at (-135:.75);
\fill[yellow] (A7) -- (A0) -- (B0) -- (B1) ;
\fill[green!50!black] (A1) -- (A2) -- (B2) -- (B1) -- (B0);
\fill[green!50!black] (A2) -- (A3) -- (B4) -- (B3) -- (B2);
\fill[green!50!black] (A4) -- (A5) -- (B5) -- (B3) -- (B4);
\fill[blue] (A6) -- (A7) -- (B1) -- (B2) -- (B3) -- (B5);
\octogono{2.25};
\draw[blue, line width=1.2pt] (B0) -- (A0);
\draw[blue, line width=1.2pt] (B0) -- (A1);
\draw[blue, line width=1.2] (B2) -- (A2);
\draw[blue, line width=1.2] (B4) -- (A3);
\draw[blue, line width=1.2] (B4) -- (A4);
\draw[blue, line width=1.2] (B5) -- (A5);
\draw[blue, line width=1.2] (B5) -- (A6);
\draw[blue, line width=1.2] (B1) -- (A7);

\foreach \x[evaluate= \x as \y using \x + 1] in {0,...,3}
{
\draw[red, line width=1.2pt] (B\x) -- (B\y);
}
\draw[red, line width=1.2pt] (B3) -- (B5);

\node at (-.9,.9) {$O_8^+$};
\end{scope}
\begin{scope}[xshift=0cm, yshift=-5cm]
\foreach \x in {0, ..., 7}
{
\coordinate (A\x) at ({45*\x - 22.5}:1.125);
\coordinate (B\x) at (45*\x - 22.5:.75);
}
\coordinate(B0) at (-22:.5);
\coordinate(B1) at (20:.5);
\coordinate(B2) at (90:.25);
\coordinate(B3) at (90:.75);
\coordinate(B4) at (200:.5);
\coordinate(B5) at (180:.75);
\fill[yellow] (A0) -- (A1) -- (B1) -- (B0) ;
\fill[yellow] (A5) -- (A6) -- (B4) -- (B5) ;
\fill[green!50!black] (A1) -- (A2) -- (B3) -- (B2) -- (B1);
\fill[blue] (A3) -- (A4) -- (B5) -- (B4) -- (B2) -- (B3);
\fill[blue] (A6) -- (A7) -- (B0) -- (B1) -- (B2) -- (B4);
\octogono{2.25};
\draw[blue, line width=1.2pt] (B0) -- (A0);
\draw[blue, line width=1.2pt] (B1) -- (A1);
\draw[blue, line width=1.2] (B3) -- (A2);
\draw[blue, line width=1.2] (B3) -- (A3);
\draw[blue, line width=1.2] (B5) -- (A4);
\draw[blue, line width=1.2] (B5) -- (A5);
\draw[blue, line width=1.2] (B4) -- (A6);
\draw[blue, line width=1.2] (B0) -- (A7);

\foreach \x[evaluate= \x as \y using \x + 1] in {0, 1, 2, 4}
{
\draw[red, line width=1.2pt] (B\x) -- (B\y);
}
\draw[red, line width=1.2pt] (B2) -- (B4);

\node at (-.9,.9) {$O_9^+$};
\end{scope}
\begin{scope}[xshift=2.5cm, yshift=-5cm]
\foreach \x in {0, ..., 7}
{
\coordinate (A\x) at ({45*\x - 22.5}:1.125);
\coordinate (B\x) at (45*\x - 22.5:.75);
}
\coordinate(B0) at (0:.75);
\coordinate(B1) at (30:.5);
\coordinate(B2) at (90:0);
\coordinate(B3) at (150:.5);
\coordinate(B4) at (180:.75);
\coordinate(B5) at (-90:.5);
\fill[yellow] (A1) -- (A2) -- (B1) -- (B0) ;
\fill[yellow] (A3) -- (A4) -- (B4) -- (B3) ;
\fill[green!50!black] (A2) -- (A3) -- (B3) -- (B2) -- (B1);
\fill[blue] (A5) -- (A6) -- (B5) -- (B2) -- (B3) -- (B4);
\fill[blue] (A7) -- (A0) -- (B0) -- (B1) -- (B2) -- (B5);
\octogono{2.25};
\draw[blue, line width=1.2pt] (B0) -- (A0);
\draw[blue, line width=1.2pt] (B0) -- (A1);
\draw[blue, line width=1.2] (B1) -- (A2);
\draw[blue, line width=1.2] (B3) -- (A3);
\draw[blue, line width=1.2] (B4) -- (A4);
\draw[blue, line width=1.2] (B4) -- (A5);
\draw[blue, line width=1.2] (B5) -- (A6);
\draw[blue, line width=1.2] (B5) -- (A7);

\foreach \x[evaluate= \x as \y using \x + 1] in {0, ..., 3}
{
\draw[red, line width=1.2pt] (B\x) -- (B\y);
}
\draw[red, line width=1.2pt] (B2) -- (B5);

\node at (-.9,.9) {$O_{10}$};
\end{scope}
\begin{scope}[xshift=5cm, yshift=-5cm]
\foreach \x in {0, ..., 7}
{
\coordinate (A\x) at ({45*\x - 22.5}:1.125);
\coordinate (B\x) at (45*\x - 22.5:.75);
}
\coordinate(B0) at (-30:.75);
\coordinate(B1) at (30:.75);
\coordinate(B2) at (0:.25);
\coordinate(B3) at (180:.25);
\coordinate(B4) at (150:.75);
\coordinate(B5) at (210:.75);
\fill[green!50!black] (A0) -- (A1) -- (B1) -- (B2) -- (B0);
\fill[green!50!black] (A4) -- (A5) -- (B5) -- (B3) -- (B4);
\fill[blue] (A2) -- (A3) -- (B4) -- (B3) -- (B2) -- (B1);
\fill[blue] (A6) -- (A7) -- (B0) -- (B2) -- (B3) -- (B5);
\octogono{2.25};
\draw[blue, line width=1.2pt] (B0) -- (A0);
\draw[blue, line width=1.2pt] (B1) -- (A1);
\draw[blue, line width=1.2] (B1) -- (A2);
\draw[blue, line width=1.2] (B4) -- (A3);
\draw[blue, line width=1.2] (B4) -- (A4);
\draw[blue, line width=1.2] (B5) -- (A5);
\draw[blue, line width=1.2] (B5) -- (A6);
\draw[blue, line width=1.2] (B0) -- (A7);

\foreach \x[evaluate= \x as \y using \x + 1] in {1, 2, 3}
{
\draw[red, line width=1.2pt] (B\x) -- (B\y);
}
\draw[red, line width=1.2pt] (B0) -- (B2);
\draw[red, line width=1.2pt] (B3) -- (B5);

\node at (-.9,.9) {$O_{11}$};
\end{scope}
\begin{scope}[xshift=7.5cm, yshift=-5cm]
\foreach \x in {0, ..., 7}
{
\coordinate (A\x) at ({45*\x - 22.5}:1.125);
\coordinate (B\x) at (45*\x - 22.5:.75);
}
\coordinate(B0) at (-30:.75);
\coordinate(B1) at (20:.75);
\coordinate(B2) at (90:.25);
\coordinate(B3) at (160:.75);
\coordinate(B4) at (210:.75);
\coordinate(B5) at (90:.75);
\fill[yellow] (A0) -- (A1) -- (B1) -- (B0) ;
\fill[yellow] (A4) -- (A5) -- (B4) -- (B3) ;
\fill[green!50!black] (A1) -- (A2) -- (B5) -- (B2) -- (B1);
\fill[green!50!black] (A3) -- (A4) -- (B3) -- (B2) -- (B5);
\fill[pink] (A6) -- (A7) -- (B0) -- (B1) -- (B2) -- (B3) -- (B4);
\octogono{2.25};
\draw[blue, line width=1.2pt] (B0) -- (A0);
\draw[blue, line width=1.2pt] (B1) -- (A1);
\draw[blue, line width=1.2] (B5) -- (A2);
\draw[blue, line width=1.2] (B5) -- (A3);
\draw[blue, line width=1.2] (B3) -- (A4);
\draw[blue, line width=1.2] (B4) -- (A5);
\draw[blue, line width=1.2] (B4) -- (A6);
\draw[blue, line width=1.2] (B0) -- (A7);

\foreach \x[evaluate= \x as \y using \x + 1] in {0, ..., 3}
{
\draw[red, line width=1.2pt] (B\x) -- (B\y);
}
\draw[red, line width=1.2pt] (B2) -- (B5);

\node at (-.9,.9) {$O_{12}$};
\end{scope}
\end{tikzpicture}
\caption{The $19$ different simple $8$-pyramitoids.
We have to add the symmetric $O_j^-$ for $j=2,3,4,6,\dots,9$.}
\label{pyre81}
\end{center}
\end{figure}

Pyramids appear in many polyhedra as neighborhoods of isolated
non-simple vertices, so the study of pyramitoids (simple or not) can
be very useful for the study of the total or partial smoothings of the
intersections of ellipsoids corresponding to those polyhedra.


 \begin{thm}
 Every $n$-pyramitoid is in a neighbourhood of the $n$-pyramid  $\mathcal{Y}_n$ in their domain of deformations and therefore it is a (partial or total) smoothing of  $\mathcal{Y}_n$. On the other side, a complete smoothing of the $n$-pyramid $\mathcal{Y}_n$ becomes a simple $n$-pyramitoid.
 \end{thm}

\begin{proof}
It is clear that given a  $n$-pyramitoid, one can contract its core tree to a single point. The result is an $n$-pyramid $\mathcal{Y}_n$. We can also proceed by contracting
    each edge of the core tree separately and all the intermediate graphs correspond to $n$-pyramitoids in the deformation domain of $\mathcal{Y}_n$.
\end{proof}

One consequence of Proposition~\ref{prop:labels}\ref{prop5} and~\ref{prop6} is that we can recover all the simple pyramitoids
by successive truncations of a tetrahedron, i.e., pyramitoids are dual stacked polytopes. Let us recall the effect of one truncation.

 \begin{thm}[{\cite[\S9.9.2]{LdM2023}}]\label{tzpn}
 Let $\mathbf{Y}_n$ be an $n$-pyramitoid obtained from an $(n-1)$-pyramitoid $\mathbf{Y}_{n-1}$ by truncating a vertex in the basis. Call $Z_n= Z(\mathbf{Y}_n)$ and $Z_{n-1}= Z(\mathbf{Y}_{n-1})$ (see Remark{\rm~\ref{rem:Z}}). Then
 \begin{equation}\label{ZdePn}
   Z_n=Z_{n-1}\, \# \,Z_{n-1} \,\#_{i=1} ^{(2^{n - 3}- 1 )}  (\mathbb{S}^2 \times\mathbb{S}^1).
 \end{equation}
If $\mathbf{Y}_n$  is simple, then $Z_n$ is a connected sum of $b_n:=2^{n - 3} (n - 4) + 1$ copies of $\mathbb{S}^2 \times\mathbb{S}^1$.
 \end{thm}

 \begin{coro}
Let $\mathbf{Y}_n$ be a simple pyramitoid. Then
$
  \ker\left( \pi _1^{\orb} (\mathbf{Y}_n) \longrightarrow (\mathbb{Z}/2)^{n+1} \right)$
is the free group of $b_n$ generators.
\end{coro}

\begin{proof}[Proof of Theorem{\rm~\ref{tzpn}}]
  The equation \eqref{ZdePn} follows from  \cite[Theorem 2.1]{Gi-LdM}.
  Because $Z_3=Z(\mathbf{Y}_3)=S^3$, and $Z_4=Z(\mathbf{Y}_4)= \mathbb{S}^2 \times\mathbb{S}^1$, it follows that
  \begin{equation*}
    Z_n= \connectedsum_{i=1} ^{b_n}   (\mathbb{S}^2 \times\mathbb{S}^1),\qquad
       b_n=2 b_{n-1} + 2^{n-3}-1,\quad n>3, \quad b_3=0.
  \end{equation*}
   It is easy to prove by induction that
   \begin{equation}\label{bnan-3}
    b_n=(n-4)2^{n-3}+1,\quad n>3
  \end{equation}
  Indeed, \eqref{bnan-3} is true for $n=4$: $b_4=1$.  Let's suppose that \eqref{bnan-3} is true up to $n-1$ and let's verify it for $n$:
  \begin{eqnarray*}
    b_n&=&2 b_{n-1} + 2^{n-3}-1 = 2((n-5)2^{n-4}+1)+2^{n-3}-1= \\ &=&(n-5)2^{n-3}+2+2^{n-3}-1=
    (n-4)2^{n-3}+1
  \end{eqnarray*}
 Therefore, in the simple case
   \begin{equation*}
   Z_n=  \connectedsum _{i=1} ^{2^{n - 3} (n - 4)  + 1 } (\mathbb{S}^2 \times\mathbb{S}^1)_i,
 \end{equation*}
Note that $b_n = a_{n-3}$ where $\{ a_n\}$ is the integer sequence $A000337$ in OEIS \cite{OEIS25}. The integer sequence $b_n$ is also the one giving the genus of the $n$-cube (\cite{B-H1965}).
 \end{proof}

  Because we are interested in the smoothing of an isolated $n$-singular vertex in a general polyhedron, we study next the manifold associated to the lateral faces of a simple $n$\nobreakdash-pyramitoid.  This gives also a more constructive approach to the result obtained in Theorem \ref{tzpn}.
  Let
   $ \pi_n: Z_n \rightarrow \mathbf{Y}_n^{\orb}$
  be the orbifold cover map from the manifold $Z_n$ to the mirror $n$-pyramitoid.

  In the same way, let $ \pi_n^*: Z_n^* \rightarrow \orbdomo{n}$ be the orbifold
  cover map corresponding to the reflections of the faces of the dome.  We obtain a manifold $Z_n^*$ with boundary, is the preimage by $\partial Z_n^*=\pi_n^{-1}(p_n)={\pi_n^*}^{-1}(p_n)$ of the basis $p_n$ of the $n$-piramitoid $\mathbf{Y}_n$.

  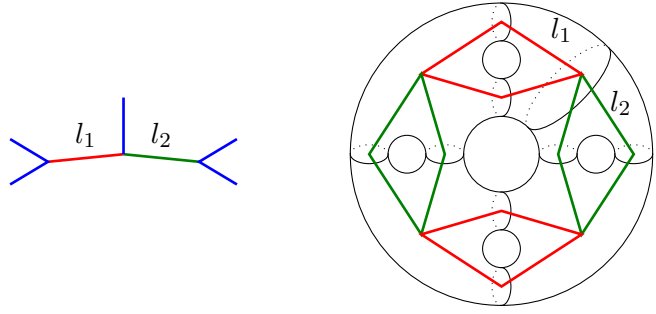
\begin{figure}[ht]
\begin{center}
\begin{tikzpicture}
\draw[red, line width=1pt] (-1, -.1) -- node[above, black] {$l_1$} (0,0);
\draw[green!50!black, line width=1pt] (0, 0) -- node[above, black] {$l_2$}(1, -.1);
\draw[blue, line width=1pt] (-1.5, .2)-- (-1, -.1) -- (-1.5, -.4);
\draw[blue, line width=1pt] (1.5, .2)-- (1, -.1) -- (1.5, -.4);
\draw[blue, line width=1pt] (0,0) -- (0, .75);
\begin{scope}[xshift=5cm]
\draw (0,0) circle[radius=.5cm];
\draw (0,0) circle[radius=2cm];

\draw (1.250,0) circle[radius=.25cm];
\draw (.5,0) arc[start angle=-180, end angle=0,x radius=.25cm, y radius=.125cm];
\draw[dotted] (.5,0) arc[start angle=180, end angle=0,x radius=.25cm, y radius=.125cm];
\draw (1.5,0) arc[start angle=-180, end angle=0,x radius=.25cm, y radius=.125cm];
\draw[dotted] (1.5,0) arc[start angle=180, end angle=0,x radius=.25cm, y radius=.125cm];

\draw (0,1.250) circle[radius=.25cm];
\draw (0,.5) arc[start angle=-90, end angle=90,y radius=.25cm, x radius=.125cm];
\draw[dotted] (0, .5) arc[start angle=270, end angle=90,y radius=.25cm, x radius=.125cm];
\draw (0,1.5) arc[start angle=-90, end angle=90,y radius=.25cm, x radius=.125cm];
\draw[dotted] (0, 1.5) arc[start angle=270, end angle=90,y radius=.25cm, x radius=.125cm];

\draw (-1.250,0) circle[radius=.25cm];
\draw (-1,0) arc[start angle=-180, end angle=0,x radius=.25cm, y radius=.125cm];
\draw[dotted] (-1,0) arc[start angle=180, end angle=0,x radius=.25cm, y radius=.125cm];
\draw (-2,0) arc[start angle=-180, end angle=0,x radius=.25cm, y radius=.125cm];
\draw[dotted] (-2,0) arc[start angle=180, end angle=0,x radius=.25cm, y radius=.125cm];

\draw (0,-1.250) circle[radius=.25cm];
\draw (0,-1) arc[start angle=-90, end angle=90,y radius=.25cm, x radius=.125cm];
\draw[dotted] (0, -1) arc[start angle=270, end angle=90,y radius=.25cm, x radius=.125cm];
\draw (0,-2) arc[start angle=-90, end angle=90,y radius=.25cm, x radius=.125cm];
\draw[dotted] (0, -2) arc[start angle=270, end angle=90,y radius=.25cm, x radius=.125cm];

\draw[rotate around={45:(0,0)}] (.5,0) arc[start angle=-180, end angle=0, x radius=.75cm, y radius=.25cm];
\draw[rotate around={45:(0,0)}, dotted] (.5,0) arc[start angle=180, end angle=0, x radius=.75cm, y radius=.25cm];

\draw[line width=1pt, green!50!black] (45:1.5) -- (0:.75) -- (-45: 1.5) -- (0:1.75) -- node[above=3pt, pos=.25, black] {$l_2$} cycle;
\draw[line width=1pt, red, rotate around={90:(0,0)}] (45:1.5) --  (0:.75) -- (-45: 1.5) -- node[above=3pt, pos=.25, black] {$l_1$} (0:1.75) --  cycle;
\draw[line width=1pt, green!50!black, rotate around={180:(0,0)}] (45:1.5) -- (0:.75) -- (-45: 1.5) -- (0:1.75) -- cycle;
\draw[line width=1pt, red, rotate around={-90:(0,0)}] (45:1.5) -- (0:.75) -- (-45: 1.5) -- (0:1.75) -- cycle;

\end{scope}
\end{tikzpicture}
\caption{The core tree $\alma{\mathbf{Y}_{5}}=l_1\cup l_2$ and the core graph $G_n$ in  $H_5= Z_5^*\subset Z_5$.  }\label{ZdeP5}
\end{center}
\end{figure}

We are going to study $Z_n^*$ and actually we are going to see that it is a handlebody.
Handlebodies are the pieces of Heegaard splittings of $3$-manifolds and
the key point to describe these splittings is to detect the so-called
cutting curves. To get these curves it is useful to see the handlebody as the
regular neighbourhood of a graph.

\begin{dfn}
A  \emph{core} $\mathcal{C}$ of  a handlebody $H_n$ is any minimal deformation retract of $H_n$. If $\mathcal{C}$ is a graph, we say that $\mathcal{C}$ is a \emph{core graph} of the handlebody.
\end{dfn}

 \begin{thm}\label{hn}
    The manifold $Z_n^*$ is a handlebody $H_{b_n}$ having ${\pi_n^*}^{-1} (\alma{\mathbf{Y}_{n}})$ as core graph.
  \end{thm}

\begin{proof}
The manifold $Z_n^*$ is a neighborhood of the graph $G_n:={\pi_n^*}^{-1}(\alma{\mathbf{Y}_{n}})$, see Figure~\ref{ZdeP5}; recall that $\pi_n^*$ is an $2^n$-orbifold cover. Let~$l_i$, $i=1, \dots ,n-3$, be the edges of $\alma{\mathbf{Y}_{n}}$.
Since $l_i$ is in exactly two faces of the dome, ${\pi_n^*}^{-1}$ if formed by $2^{n-2}$ edges. These edges give rise
to $2^{n-4}$ circles, each one  formed by $4$ copies of $l_i$.

If $l_i\cap l_j\neq\emptyset$, each one of these $l_i$-circles intersect two $l_j$-circles in one point as in
Figure~\ref{ZdeP5}. Therefore the graph  $G_n$ is connected and becomes the core of a handlebody $H_{c_n}$ of genus $c_n$. The boundary of $H_{c_n}$ is   the surface $F_{c_n}$ generated by the reflection of the basis polygon $p_n$ of $\mathbf{Y}_n$ on its edges. The genus $c_n$ of this surface can be computed using the  Euler characteristic:
\begin{equation*}
  2-2c_n = \chi (F_{c_n}) = n 2^{n-2}-n 2^{n-1} + 2^n  \quad \Longrightarrow \quad c_n= 2^{n - 3} (n - 4)  + 1= b_n
\qedhere
\end{equation*}
\end{proof}

\begin{coro}
The genus of the surface generated by reflection on the $n$ edges of a right-angle geometric $n$-polygon is  $b_n = 2^{n - 3}(n - 4) + 1,$ that is the term $a_{n-3}$ of  the integer sequence $\{a_n\}$ =  $A000337$ in OEIS \cite{OEIS25}. The same property holds for the genus of the graph $G_n$ associated to a simple $n$-pyramitoid.
\end{coro}

 \begin{coro}\label{2hn}
    The manifold $Z_n$ is the double of the $b_n$-handlebody $Z_n^*$.
  \end{coro}

\begin{proof}
To obtain  the manifold $Z_n$ we only need to reflect the handlebody $H_{b_n}$ on its boundary  surface $F_{b_n}$, that is, to construct the double of $H_{b_n}$ which is a connected sum of $b_n$ copies of  $\mathbb{S}^2 \times\mathbb{S}^1$. This is the result obtained in Theorem \ref{tzpn} by a different approach.
More precisely, the covering $\pi_n$ can be constructed gluing two copies of the covering $\pi_n^*$.
\end{proof}

Any $n$-pyramitoid $\mathbf{Y}_n$ with basis $p_n$ is determined by its $p_n$-core tree $\almap{\mathbf{Y}_n}$, which determines the $p_n$-essential tree $\esencialp{\mathbf{Y}_n}$ and the cellular
subdivision induced by this last tree. We want to codify these data.

\begin{dfn}
The \emph{code} of an $n$-pyramitoid $\mathbf{Y}_n$ is the set $(p_n, (r_1,\dots,r_{n-3})\})$, where $p_n$ is the basis polygon and $r_i$, $i=1, \dots , n-3$, is  a set of disjoint lines,
where $r_i$  crosses transversally the edge $l_i^{p_n}$ of $\almap{\mathbf{Y}_n}$ joining the two sides of the basis polygon separated by $l_i^{p_n}$.
\end{dfn}

\begin{remark}
    Let $(r_1,\dots,r_{n-3})$ be a family of pairwise disjoint $n-3$ arcs in~$p_n$ such that each arc has its endpoints
    in the interior of two non-consecutive edges. This a necessary condition that a code must satisfy. Is it sufficient?
\end{remark}

We are going to use the code  $(p_n, (r_1,\dots,r_{n-3}))$ to encode  the handlebody $H_{b_n}$.

  \begin{figure}[ht]
\begin{center}
\begin{tikzpicture}
\begin{scope}[scale=1.5]
\foreach \x in {0, ..., 4}{
\coordinate (A\x) at (\x*360/4 + 45:1);
}

\coordinate (A) at (180:1/3);
\coordinate (B) at (0:1/3);

\fill[yellow] (A0) -- (B) -- (A) --(A1) -- cycle;
\fill[yellow] (A3) -- (B) -- (A) --(A2) -- cycle;

\draw (A0) -- (A1) -- (A2) -- (A3) -- (A4);

\draw[line width=1pt, blue] (A0) -- (B);
\draw[line width=1pt, blue] (A1) -- (A);
\draw[line width=1pt, blue] (A2) -- (A);
\draw[line width=1pt, blue] (A3) -- (B);

\draw[line width=1pt, red] (A) -- (B);
\draw[line width=1pt, green] ($.5*(A0) + .5 *(A1)$) -- ($.5*(A2) + .5 *(A3)$);

\draw (-1/2, 0) node {$1$};
\draw (1/2, 0) node {$3$};
\draw (-1/3, .5) node {$2$};
\draw (1/3, -.5) node {$4$};

\end{scope}

\begin{scope}[xshift=3cm]
\foreach \x in {0, ..., 4}{
\coordinate (A\x) at (-\x*360/4 -90:1);
}

\draw[line width=1pt] (A0) -- (A1) -- (A2) -- (A3) -- (A4)  -- (A2) ;
\foreach \x in {1, ..., 4}{
\filldraw[fill=white] (A\x) circle[radius=.2];
\node at (A\x) {$\x$};
}

\draw[red, line width=1pt] (1/2, 0) -- (-1/2,0);

\end{scope}

\begin{scope}[xshift=6cm, scale=1.25]
\foreach \x in {0, ..., 5}{
\coordinate (A\x) at (\x*360/5 + 90:1);
}

\coordinate (A) at (180:1/2);
\coordinate (B) at (90:1/2);
\coordinate (C) at (0:1/2);

\fill[yellow] (A0) -- (B) -- (C) --(A4) -- cycle;
\fill[yellow] (A0) -- (B) -- (A) --(A1) -- cycle;

\fill[green] (A2) -- (A) -- (B) -- (C) --(A3) -- cycle;

\draw (A0) -- (A1) -- (A2) -- (A3) -- (A4) -- (A5);

\draw[line width=1pt, blue] (A0) -- (B);
\draw[line width=1pt, blue] (A1) -- (A);
\draw[line width=1pt, blue] (A2) -- (A);
\draw[line width=1pt, blue] (A3) -- (C);
\draw[line width=1pt, blue] (A4) -- (C);

\draw[line width=1pt, red] (A) -- (B) -- (C);
\draw[line width=1pt, green!70!black] ($.5*(A0) + .5 *(A1)$) to[bend left=15pt] ($2/3*(A2) + 1/3*(A3)$);
\draw[line width=1pt, green!70!black] ($.5*(A0) + .5 *(A4)$) to[bend right=15pt] ($1/3*(A2) + 2/3*(A3)$);

\draw (-2/3, -1/6) node {$1$};
\node[right =3pt] at (A1) {$2$};
\node[above right =0pt] at (B) {$3$};
\draw (2/3, -1/6) node {$4$};
\draw (0, -1/2) node {$5$};
\end{scope}

\begin{scope}[xshift=9cm]
\foreach \x in {0, ..., 5}{
\coordinate (A\x) at (-\x*360/5 -90:1);
}

\draw[line width=1pt] (A3) -- (A0) -- (A1) -- (A2) -- (A3) -- (A4) -- (A5)  -- (A2) ;
\foreach \x in {1, ..., 5}{
\filldraw[fill=white] (A\x) circle[radius=.2];
\node at (A\x) {$\x$};
}

\draw[red, line width=1pt] (1/2, 0) -- (0, 1/2) -- (-1/2,0);

\end{scope}
\end{tikzpicture}
\caption{The code (green lines) and triangulations in $\mathbf{Y}_4$ and  $\mathbf{Y}_5$.  }\label{codytri}
\end{center}
\end{figure}
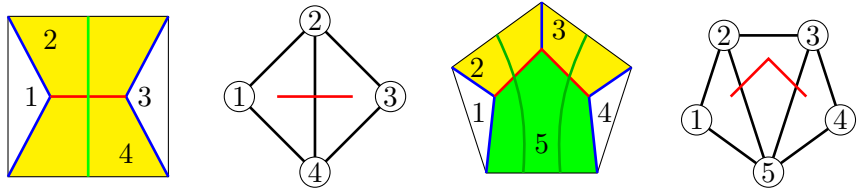

\begin{lema} There is a one-to-one correspondence between the set of codes of an $n$-pyramitoid and the triangulations of
the dual polygon~$q_n$, see Figure{\rm~\ref{codytri}}.
\end{lema}

\begin{proof}
This is true because the dual tree of a triangulation of~$q_n$  (one vertex for each triangle in the triangulation and one edge for each interior edge) is exactly~$\almap{\mathbf{Y}_n}$.
\end{proof}

  \begin{figure}[ht]
\begin{center}
\begin{tikzpicture}
\begin{scope}[scale=1.5]
\foreach \x in {0, ..., 5}{
\coordinate (A\x) at (\x*72 + 18:1);
}

\draw (A0) -- (A1) -- (A2) -- (A3) -- (A4) -- (A5);
\coordinate (A) at (-1/2,0);
\coordinate (B) at (0, 1/4);
\coordinate (C) at (1/2,0);

\draw[green, line width=1pt] ($.5*(A1) + .5*(A2)$) node[below right=3pt, black] {$r_1$} to[bend left=20pt] ($2/3*(A3) + 1/3*(A4)$);
\fill[white] ($.5*(A) + .5*(B)$) circle[radius=.1cm];

\draw[green, line width=1pt] ($.5*(A1) + .5*(A0)$) node[below left=3pt, black] {$r_2$} to[bend right=20pt] ($1/3*(A3) + 2/3*(A4)$);
\fill[white] ($.5*(A) + .5*(B)$) circle[radius=.1cm];

\draw[line width=1pt, blue] (A0) -- (C);
\draw[line width=1pt, blue] (A1) -- (B);
\draw[line width=1pt, blue] (A2) -- (A);
\draw[line width=1pt, blue] (A3) -- (A);
\draw[line width=1pt, blue] (A4) -- (C);

\draw[line width=1pt, red] (A) node[above] {$l_1$} -- (B) ;
\draw[line width=1pt, green!50!black]  (B) -- (C) node[above] {$l_2$} ;

\draw ($1.25/2*(A2) + 1.25/2*(A3)$) node {$1$};
\draw ($1.25/2*(A4) + 1.25/2*(A0)$) node {$4$};
\draw ($1.25/2*(A2) + 1.25/2*(A1)$) node {$2$};
\draw ($1.25/2*(A0) + 1.25/2*(A1)$) node {$3$};
\draw ($1.25/2*(A3) + 1.25/2*(A4)$) node {$5$};
\end{scope}

\begin{scope}[ scale=1.25, xshift=4cm, yshift=-1cm]

\draw (-2,0) arc[start angle=180, end angle=90, radius=2cm] coordinate[pos=.2] (P1) coordinate[pos=.8] (P2);

\draw (-1.5,0) arc[start angle=180, end angle=0, radius=.25cm] coordinate[pos=.25] (Q1)  coordinate[pos=.75] (Q2);
\draw (-.5,0) arc[start angle=180, end angle=90, radius=.5cm]  coordinate[pos=1/3] (R1)  coordinate[pos=2/3] (R2);
\draw (-2,0) arc[start angle=-180, end angle=0, x radius=.25cm, y radius=.125cm];
\draw[dotted] (-2,0) arc[start angle=180, end angle=0, x radius=.25cm, y radius=.125cm];
\draw (-1,0) arc[start angle=-180, end angle=0, x radius=.25cm, y radius=.125cm];
\draw[dotted] (-1,0) arc[start angle=180, end angle=0, x radius=.25cm, y radius=.125cm];

\draw (0,1.5) arc[start angle=90, end angle=270, radius=.25cm]   coordinate[pos=.25] (S1)  coordinate[pos=.75] (S2);
\draw (-.5,0) arc[start angle=180, end angle=90, radius=.5cm];
\draw (0, 2) arc[start angle=90, end angle=-90, y radius=.25cm, x radius=.125cm];
\draw[dotted] (0, 2) arc[start angle=90, end angle=270, y radius=.25cm, x radius=.125cm];
\draw (0,1) arc[start angle=90, end angle=-90, y radius=.25cm, x radius=.125cm];
\draw[dotted] (0,1) arc[start angle=90, end angle=270, y radius=.25cm, x radius=.125cm];

\coordinate (V) at (135:1.7) ;
\draw[line width=1, green!50!black] (V) -- node[pos=.25, above] {$l_2$} (0,1.75) (V) -- node[pos=.5, above] {$l_2$} (0,.75);
\draw[line width=1, red] (V) -- node[pos=.25, left] {$l_1$} (-1.75, 0) (V) -- node[pos=.5, right] {$l_1$} (-.75, 0);

\draw[green, line width=1] (P1) to[bend left] (Q1);
\draw[green, line width=1, dotted] (P1) to[bend right] (Q1);

\draw[green, line width=1] (Q2) to[bend left] (R1);
\draw[green, line width=1, dotted] (Q2) to[bend right] (R1);

\draw[green, line width=1] (P2) to[bend left] (S1);
\draw[green, line width=1, dotted] (P2) to[bend right] (S1);

\draw[green, line width=1] (R2) to[bend left] (S2);
\draw[green, line width=1, dotted] (R2) to[bend right] (S2);

\end{scope}

\begin{scope}[xshift=8cm]
\begin{scope}
\draw (-2,0) arc[start angle=180, end angle=90, radius=2cm] coordinate[pos=.2] (P1) coordinate[pos=.8] (P2);

\draw (-1.5,0) arc[start angle=180, end angle=0, radius=.25cm] coordinate[pos=.25] (Q1)  coordinate[pos=.75] (Q2);
\draw (-.5,0) arc[start angle=180, end angle=90, radius=.5cm]  coordinate[pos=1/3] (R1)  coordinate[pos=2/3] (R2);

\draw (0,1.5) arc[start angle=90, end angle=270, radius=.25cm]   coordinate[pos=.25] (S1)  coordinate[pos=.75] (S2);
\draw (-.5,0) arc[start angle=180, end angle=90, radius=.5cm];

\coordinate (V) at (135:1.7) ;
\draw[line width=1,  green!50!black] (V) -- node[pos=.25, above] {$\scriptstyle l_2$} (0,1.75) (V) -- (0,.75) ;
\draw[line width=1, red] (V) --  node[pos=.25, left] {$\scriptstyle l_1$} (-1.75, 0) (V) -- (-.75, 0);

\draw[green, line width=1] (P1) to[bend left] (Q1);
\draw[green, line width=1, dotted] (P1) to[bend right] (Q1);

\draw[green, line width=1] (Q2) to[bend left] (R1);
\draw[green, line width=1, dotted] (Q2) to[bend right] (R1);

\draw[green, line width=1] (P2) to[bend left] (S1);
\draw[green, line width=1, dotted] (P2) to[bend right] (S1);

\draw[green, line width=1] (R2) to[bend left] (S2);
\draw[green, line width=1, dotted] (R2) to[bend right] (S2);
\end{scope}

\begin{scope}[xscale=-1]
\draw (-2,0) arc[start angle=180, end angle=90, radius=2cm] coordinate[pos=.2] (P1) coordinate[pos=.8] (P2);

\draw (-1.5,0) arc[start angle=180, end angle=0, radius=.25cm] coordinate[pos=.25] (Q1)  coordinate[pos=.75] (Q2);
\draw (-.5,0) arc[start angle=180, end angle=90, radius=.5cm]  coordinate[pos=1/3] (R1)  coordinate[pos=2/3] (R2);

\draw (0,1.5) arc[start angle=90, end angle=270, radius=.25cm]   coordinate[pos=.25] (S1)  coordinate[pos=.75] (S2);
\draw (-.5,0) arc[start angle=180, end angle=90, radius=.5cm];

\coordinate (V) at (135:1.7) ;
\draw[line width=1, green!50!black] (V) -- (0,1.75) (V) -- (0,.75);
\draw[line width=1, red] (V) -- (-1.75, 0) (V) -- (-.75, 0);

\draw[green, line width=1] (P1) to[bend left] (Q1);
\draw[green, line width=1, dotted] (P1) to[bend right] (Q1);

\draw[green, line width=1] (Q2) to[bend left] (R1);
\draw[green, line width=1, dotted] (Q2) to[bend right] (R1);

\draw[green, line width=1] (P2) to[bend left] (S1);
\draw[green, line width=1, dotted] (P2) to[bend right] (S1);

\draw[green, line width=1] (R2) to[bend left] (S2);
\draw[green, line width=1, dotted] (R2) to[bend right] (S2);
\end{scope}

\begin{scope}[yscale=-1]
\draw (-2,0) arc[start angle=180, end angle=90, radius=2cm] coordinate[pos=.2] (P1) coordinate[pos=.8] (P2);

\draw (-1.5,0) arc[start angle=180, end angle=0, radius=.25cm] coordinate[pos=.25] (Q1)  coordinate[pos=.75] (Q2);
\draw (-.5,0) arc[start angle=180, end angle=90, radius=.5cm]  coordinate[pos=1/3] (R1)  coordinate[pos=2/3] (R2);

\draw (0,1.5) arc[start angle=90, end angle=270, radius=.25cm]   coordinate[pos=.25] (S1)  coordinate[pos=.75] (S2);
\draw (-.5,0) arc[start angle=180, end angle=90, radius=.5cm];

\coordinate (V) at (135:1.7) ;
\draw[line width=1, green!50!black] (V) --  node[pos=.25, below] {$\scriptstyle l_2$}   (0,1.75)  (V) --  (0,.75);
\draw[line width=1, red]  (V) -- (-1.75, 0) (V) --  (-.75, 0);

\draw[green, line width=1] (P1) to[bend left] (Q1);
\draw[green, line width=1, dotted] (P1) to[bend right] (Q1);

\draw[green, line width=1] (Q2) to[bend left] (R1);
\draw[green, line width=1, dotted] (Q2) to[bend right] (R1);

\draw[green, line width=1] (P2) to[bend left] (S1);
\draw[green, line width=1, dotted] (P2) to[bend right] (S1);

\draw[green, line width=1] (R2) to[bend left] (S2);
\draw[green, line width=1, dotted] (R2) to[bend right] (S2);
\end{scope}

\begin{scope}[scale=-1]
\draw (-2,0) arc[start angle=180, end angle=90, radius=2cm] coordinate[pos=.2] (P1) coordinate[pos=.8] (P2);

\draw (-1.5,0) arc[start angle=180, end angle=0, radius=.25cm] coordinate[pos=.25] (Q1)  coordinate[pos=.75] (Q2);
\draw (-.5,0) arc[start angle=180, end angle=90, radius=.5cm]  coordinate[pos=1/3] (R1)  coordinate[pos=2/3] (R2);

\draw (0,1.5) arc[start angle=90, end angle=270, radius=.25cm]   coordinate[pos=.25] (S1)  coordinate[pos=.75] (S2);
\draw (-.5,0) arc[start angle=180, end angle=90, radius=.5cm];

\coordinate (V) at (135:1.7) ;
\draw[line width=1, green!50!black] (V) --  (0,1.75) (V) -- (0,.75) ;
\draw[line width=1, red] (V) --  node[pos=.25, right] {$\scriptstyle l_1$}  (-1.75, 0) (V) -- (-.75, 0);

\draw[green, line width=1] (P1) to[bend left] (Q1);
\draw[green, line width=1, dotted] (P1) to[bend right] (Q1);

\draw[green, line width=1] (Q2) to[bend left] (R1);
\draw[green, line width=1, dotted] (Q2) to[bend right] (R1);

\draw[green, line width=1] (P2) to[bend left] (S1);
\draw[green, line width=1, dotted] (P2) to[bend right] (S1);

\draw[green, line width=1] (R2) to[bend left] (S2);
\draw[green, line width=1, dotted] (R2) to[bend right] (S2);
\end{scope}

\end{scope}
\end{tikzpicture}
\caption{The code ($r_1$ and $r_2$) and the core tree ($l_1$ and $l_2$) in $\mathbf{Y}_5$, and two partial covers  with the  preimages of the code and the core.}
\label{cod5}
\end{center}
\end{figure}
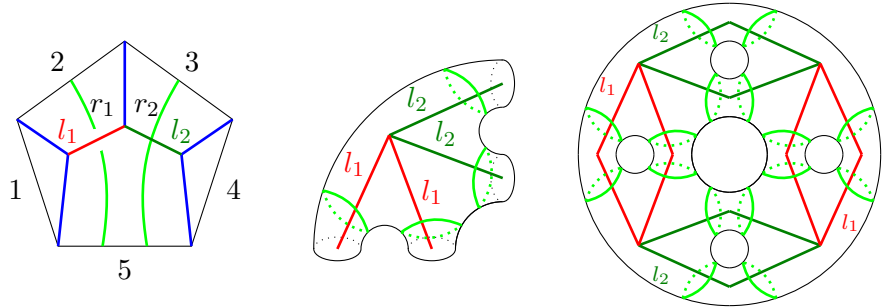

Let us fix a pyramitoid $\mathbf{Y}_n$ with a code $(r_1,\dots, r_{n-3})$. We number clockwise the edges of the basis~$p_n$ (Figure~\ref{cod5}).
This numbering induces another one of the $2$-cells of the cellular decomposition of $p_n$ and also of the
faces of the dome $\domo{n}$. We consider the indices $1\leq j_1<\dots<j_s\leq n$ of the edges containing end points
of the code (i.e., corresponding to non triangular faces).
Let $\rho_i$ be the reflection on the $i$-face. The reflections $\rho_{j_1},\dots,\rho_{j_s}$ generate a group
of order~$2^s$.
For example, the central  drawing in Figure~\ref{cod5} shows the result after the
action on $\mathbf{Y}_5$ of the group generated by $\rho_2,\rho_3,\rho_5$ of $\mathbf{Y}_5$.
The circles,  $(\pi_{n_{\mid H_{b_n}}})^{-1}(r_i)$,  generated by the line  $r_i$ in the surface $F_{b_n}$ bound   meridian disks of the core circles generated by $l_i$ in the core graph $G_n$. Therefore these circles  identify the handlebody $H_{b_n}$ from the surface $F_{b_n}$, see Figure~\ref{cod5}.

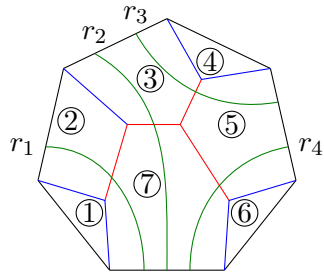
\begin{figure}[ht]
\begin{center}
\begin{tikzpicture}[scale=1.75]
\foreach \x in {0,...,7}
{
\coordinate (A\x) at ({180/7*(2*\x+3) - 90}:1);
}
\foreach \x[evaluate= \x as \y using \x + 1] in {0,...,6}
{
\draw (A\x) -- (A\y);
}
\coordinate (B0) at ({90/7*4- 90}:.6);
\coordinate (B1) at ({720/7 + 90/7*4- 90}:.6);
\coordinate (B2) at ({-720/7 + 90/7*4- 90}:.6);
\coordinate (B3) at (.1, .2);
\coordinate (B4) at (-.3, .2);
\draw[red] (B0) -- (B3) -- (B4) -- (B2) (B3) -- (B1);
\draw[blue] (A0) -- (B0) -- (A6) (A1) -- (B1) -- (A2) (A4) -- (B2) -- (A5) (A3) -- (B4);
\draw[green!50!black] ($.7*(A0) + .3*(A1)$) node[right,black] {$r_4$} to[out=-170, in=90] ($.7*(A6) + .3*(A5)$)
($.7*(A1) + .3*(A0)$) to[out=-170, in=-60] ($.7*(A2) + .3*(A3)$) node[above,black] {$r_3$}
($.5*(A5) + .5*(A6)$) to[out=90, in=-30] ($.7*(A3) + .3*(A2)$) node[above,black] {$r_2$}
($.7*(A4) + .3*(A3)$) node[left,black] {$r_1$} to[out=0, in=90] ($.7*(A5) + .3*(A6)$);
\draw ($1/4*(A0)+1/4*(A6)+1/2*(B0)$) node {$6$} circle [radius=.1];
\draw ($1/4*(A4)+1/4*(A5)+1/2*(B2)$) node {$1$} circle [radius=.1];
\draw ($1/4*(A1)+1/4*(A2)+1/2*(B1)$) node {$4$} circle [radius=.1];
\draw ($(A4)+ (.25,.45)$) node {$2$} circle [radius=.1];
\draw ($1/4*(A3)+1/4*(A2)+.7*(B3)$) node {$3$} circle [radius=.1];
\draw ($1/4*(A1)+1/4*(A0)+1/2*(B3)$) node {$5$} circle [radius=.1];
\draw ($(B4)+ (.15,-.45)$) node {$7$} circle [radius=.1];
\end{tikzpicture}
\caption{Example of code of a $\mathbf{Y}_7$ with the associated cell decomposition.}
\label{fcodes1}
\end{center}
\end{figure}

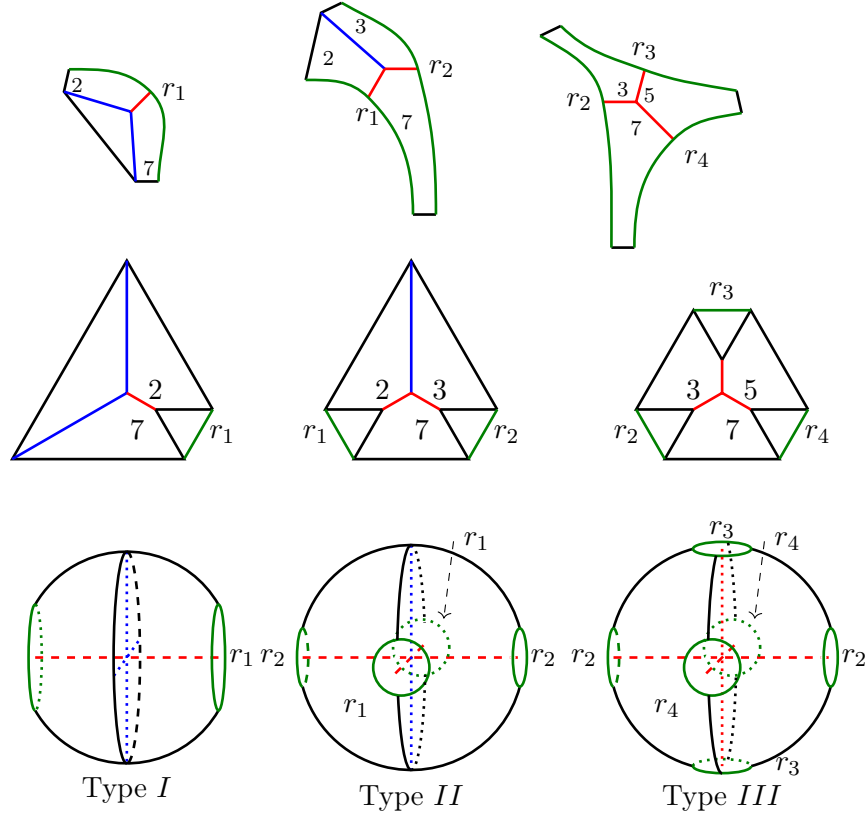
\begin{figure}[ht]
\begin{center}
\begin{tikzpicture}[scale=1.75]
\begin{scope}[yshift=.5cm, xshift=-.25cm]
\foreach \x in {0,...,7}
{
\coordinate (A\x) at ({180/7*(2*\x+3) - 90}:1);
}
\coordinate (B2) at ({-720/7 + 90/7*4- 90}:.6);
\coordinate (C) at ($(B2) + .15*(1, 1)$);
\coordinate (D5) at ($.8*(A5) + .2*(A6)$);
\coordinate (D4) at ($.8*(A4) + .2*(A3)$);
\draw[line width=1pt] (D4) -- (A4) -- (A5) -- (D5);
\draw[blue,line width=1pt]  (A4) -- (B2) -- (A5);
\draw[red,line width=1pt] (B2) -- (C);
\draw[green!50!black,line width=1pt] (D5) to[out=90, in=-45]  (C) node[right, black] {$r_1$}  to[out=135, in=0]  (D4);
\node at ($(A4) + (.1,.05)$) {$\scriptstyle 2$};
\node at ($(A5) + (.1,.1)$) {$\scriptstyle 7$};
\end{scope}

\begin{scope}[xshift=1.5cm, yshift=.25cm]
\foreach \x in {0,...,7}
{
\coordinate (A\x) at ({180/7*(2*\x+3) - 90}:1);
}
\coordinate (B4) at (-.3, .2);
\coordinate (C) at ($(B2) + .15*(1, 1)$);
\coordinate (D3) at ($.8*(A3) + .2*(A2)$);
\coordinate (D4) at ($.6*(A4) + .4*(A3)$);
\coordinate (D5) at ($.6*(A5) + .4*(A6)$);
\coordinate (D6) at ($.6*(A6) + .4*(A5)$);
\coordinate (R2) at ($(B4) +(0:.25)$);
\coordinate (R1) at ($(B4) +(-120:.25)$);
\draw[line width=1pt] (D4) -- (A3) -- (D3) (D5) -- (D6);
\draw[blue,line width=1pt]  (A3) -- (B4);
\draw[red,line width=1pt] (B4) -- (R1) (B4) -- (R2);
\draw[green!50!black,line width=1pt] (D5) to[out=90, in=-45] (R1) node[below, black] {$r_1$} to[out=135, in=0] (D4);
\draw[green!50!black,line width=1pt] (D6) to[out=90, in=-75] (R2) node[right, black] {$r_2$} to[out=105, in=-30] (D3);
\node at ($(A4) + (.25,.5)$) {$\scriptstyle 2$};
\node at ($(A3) + (.3,-.1)$) {$\scriptstyle 3$};
\node at ($(A5) + (.3,.7)$) {$\scriptstyle 7$};
\end{scope}

\begin{scope}[xshift=3cm]
\foreach \x in {0,...,7}
{
\coordinate (A\x) at ({180/7*(2*\x+3) - 90}:1);
}
\coordinate (B3) at (.1, .2);
\coordinate (D3) at ($.8*(A3) + .2*(A2)$);
\coordinate (D2) at ($.6*(A3) + .4*(A2)$);
\coordinate (D1) at ($.6*(A1) + .4*(A0)$);
\coordinate (D0) at ($.4*(A1) + .6*(A0)$);
\coordinate (D5) at ($.6*(A5) + .4*(A6)$);
\coordinate (D6) at ($.6*(A6) + .4*(A5)$);
\coordinate (R2) at ($(B3) +(-.25,0)$);
\coordinate (R3) at ($(B3) +(75:.25)$);
\coordinate (R4) at ($(B3) +(-45:.4)$);
\draw[line width=1pt] (D3) -- (D2)  (D0) -- (D1) (D5) -- (D6);
\draw[red,line width=1pt] (B3) -- (R2) (B3) -- (R3) (B3) -- (R4);
\draw[green!50!black,line width=1pt] (D5) to[out=90, in=-80] (R2) node[left, black] {$r_2$} to[out=100, in=-30] (D3) (D6) to[out=90, in=-135] (R4)  node[below right, black] {$r_4$} to[out=45, in=190] (D0) (D1) to[out=-190, in=-20] (R3) node[above, black] {$r_3$} to[out=160, in=-45] (D2);
\node at ($(B3) + (.1,.05)$) {$\scriptstyle 5$};
\node at ($(B3) + (-.1,.1)$) {$\scriptstyle 3$};
\node at ($(B3) + (0,-.2)$) {$\scriptstyle 7$};
\end{scope}

\begin{scope}[yshift=-2cm, xshift=-.75cm]
\coordinate (A0) at (90:1);
\coordinate (A1) at (-150:1);
\coordinate (A2) at (-30:1);
\coordinate (O) at (0,0);
\coordinate (P2O) at ($.25*(A2) + .75*(O)$);
\coordinate (P20) at ($.75*(A2) + .25*(A0)$);
\coordinate (P21) at ($.75*(A2) + .25*(A1)$);

\draw[line width=1pt] (P20) -- (A0) -- (A1) -- (P21) -- (P2O) -- (P20);
\draw[line width=1pt, blue] (A0) -- (O) -- (A1);
\draw[line width=1pt, red] (O) -- (P2O) node[above, black] {$2$} node[below left, black] {$7$};
\draw[line width=1pt, green!50!black] (P21) -- node[right, black] {$r_1$} (P20);
\end{scope}

\begin{scope}[yshift=-2cm, xshift=1.4cm]
\coordinate (A0) at (90:1);
\coordinate (A1) at (-150:1);
\coordinate (A2) at (-30:1);
\coordinate (O) at (0,0);
\coordinate (P2O) at ($.25*(A2) + .75*(O)$);
\coordinate (P20) at ($.75*(A2) + .25*(A0)$);
\coordinate (P21) at ($.75*(A2) + .25*(A1)$);
\coordinate (P1O) at ($.25*(A1) + .75*(O)$);
\coordinate (P10) at ($.75*(A1) + .25*(A0)$);
\coordinate (P12) at ($.75*(A1) + .25*(A2)$);

\draw[line width=1pt] (P20) -- (A0) -- (P10) -- (P1O) -- (P12) -- (P21) -- (P2O) -- (P20);
\draw[line width=1pt, blue] (A0) -- (O);
\draw[line width=1pt, red] (O) -- (P2O) node[above, black] {$3$}  node[below left, black] {$7$} (O) -- (P1O)node[above, black] {$2$} ;
\draw[line width=1pt, green!50!black] (P21) -- node[right, black] {$r_2$} (P20)
(P12) -- node[left, black] {$r_1$} (P10);
\end{scope}

\begin{scope}[yshift=-2cm, xshift=3.75cm]
\coordinate (A0) at (90:1);
\coordinate (A1) at (-150:1);
\coordinate (A2) at (-30:1);
\coordinate (O) at (0,0);
\coordinate (P2O) at ($.25*(A2) + .75*(O)$);
\coordinate (P20) at ($.75*(A2) + .25*(A0)$);
\coordinate (P21) at ($.75*(A2) + .25*(A1)$);
\coordinate (P1O) at ($.25*(A1) + .75*(O)$);
\coordinate (P10) at ($.75*(A1) + .25*(A0)$);
\coordinate (P12) at ($.75*(A1) + .25*(A2)$);
\coordinate (P0O) at ($.25*(A0) + .75*(O)$);
\coordinate (P01) at ($.75*(A0) + .25*(A1)$);
\coordinate (P02) at ($.75*(A0) + .25*(A2)$);

\draw[line width=1pt] (P20) -- (P02) -- (P0O) -- (P01) -- (P10) -- (P1O) -- (P12) -- (P21) -- (P2O) -- (P20);
\draw[line width=1pt, red] (O) -- (P2O) node[above, black] {$5$}  node[below left, black] {$7$} (O) -- (P1O)node[above, black] {$3$} (O) -- (P0O) ;
\draw[line width=1pt, green!50!black] (P21) -- node[right, black] {$r_4$} (P20)
(P12) -- node[left, black] {$r_2$} (P10) (P02) -- node[above, black] {$r_3$} (P01);
\end{scope}

\begin{scope}[yshift=-4cm, xshift=-.75cm, scale=.8]
\draw[line width=1pt] (30: 1) arc[start angle=30, end angle=150, radius=1];
\draw[line width=1pt] (-30: 1) arc[start angle=-30, end angle=-150, radius=1];
\draw[line width=1pt, red, dashed] ({-sqrt(3)/2}, 0) -- ({sqrt(3)/2}, 0);
\draw[line width=1pt, green!50!black] ({-sqrt(3)/2}, 1/2) arc[start angle=90, end angle=270, x radius=.0625, y radius=.5];
\draw[line width=1pt, dotted, green!50!black] ({-sqrt(3)/2}, 1/2) arc[start angle=90, end angle=-90, x radius=.0625, y radius=.5];
\draw[line width=1pt, blue, dotted] (0, 1) -- (0, -1) (190:.125 and 1) -- (10:.125 and 1);
\draw[line width=1pt, green!50!black] ({sqrt(3)/2}, 0) ellipse [x radius=-.0625cm,y radius=.5cm] node[right,black] {$r_1$};
\draw[line width=1pt] (0,1) arc[start angle=90, end angle=270, x radius=.125, y radius=1];
\draw[line width=1pt,dashed] (0,1) arc[start angle=90, end angle=-90, x radius=.125, y radius=1];
\node at (-90:1.25) {Type $I$};
\end{scope}

\begin{scope}[yshift=-4cm, xshift=1.4cm, scale=.85]
\draw[line width=1pt] (15: 1) arc[start angle=15, end angle=165, radius=1];
\draw[line width=1pt] (-15: 1) arc[start angle=-15, end angle=-165, radius=1];
\draw[line width=1pt, red, dashed] ({-cos(15)}, 0) -- ({cos(15)}, 0)
(-135:.2) -- (45:.2);
\draw[line width=1pt, green!50!black] (165:1) arc[start angle=90, end angle=270, x radius=.05, y radius={sin(15)}] node[left, black, pos=.5] {$r_2$};
\draw[line width=1pt, dashed, green!50!black] (165:1) arc[start angle=90, end angle=-90, x radius=.05, y radius={sin(15)}];
\fill[white] ({cos(15)-.0625},0) circle [radius=.025cm];
\draw[line width=1pt, green!50!black] ({cos(15)}, 0) ellipse [x radius=-.0625,
y radius={sin(15)}] node[right,black] {$r_2$};
\draw[line width=1pt, blue, dotted] (90:1) -- (-90:1);
\draw[line width=1pt, green!50!black] (-135:.125) circle [radius=.25] node [below left=8pt, black] {$r_1$};
\draw[line width=1pt, green!50!black, dotted] (45:.125) circle [radius=.25];
\fill[white] (-.125,0) circle [radius=.025cm];
\draw[line width=1pt] (0,1) arc[start angle=90, end angle=171, x radius=.125, y radius=1]
(0,-1) arc[start angle=-90, end angle=-160, x radius=.125, y radius=1];
\draw[line width=1pt,dotted] (0,1) arc[start angle=90, end angle=18, x radius=.125, y radius=1] (0,-1) arc[start angle=-90, end angle=-9, x radius=.125, y radius=1];
\draw[->, dashed] (70:1.1) node[right] {$r_1$} -- (50:.45);
\node at (-90:1.25) {Type $II$};
\end{scope}

\begin{scope}[yshift=-4cm, xshift=3.75cm, scale=.85]
\draw[line width=1pt] (15: 1) arc[start angle=15, end angle=75, radius=1];
\draw[line width=1pt] (165: 1) arc[start angle=165, end angle=105, radius=1];
\draw[line width=1pt] (-15: 1) arc[start angle=-15, end angle=-75, radius=1];
\draw[line width=1pt] (-165: 1) arc[start angle=-165, end angle=-105, radius=1];
\draw[line width=1pt, red, dashed] ({-cos(15)}, 0) -- ({cos(15)}, 0)
(-135:.2) -- (45:.2);
\draw[line width=1pt, green!50!black] ($(0,{-cos(15)}) + (180:{sin(15)} and .0625)$) arc[start angle=180, end angle=360, x radius={sin(15)}, y radius=.0625] node[right=5pt, black, pos=1] {$r_3$};
\draw[line width=1pt, green!50!black, dotted] ($(0,{-cos(15)}) + (180:{sin(15)} and .0625)$) arc[start angle=180, end angle=0, x radius={sin(15)}, y radius=.0625];
\draw[line width=1pt, green!50!black] (0, {cos(15)}) ellipse [x radius={sin(15)},
y radius=.0625] node[above,black] {$r_3$};

\draw[line width=1pt, green!50!black] (165:1) arc[start angle=90, end angle=270, x radius=.05, y radius={sin(15)}] node[left, black, pos=.5] {$r_2$};
\draw[line width=1pt, dashed, green!50!black] (165:1) arc[start angle=90, end angle=-90, x radius=.05, y radius={sin(15)}];
\fill[white] ({cos(15)-.0625},0) circle [radius=.025cm];
\draw[line width=1pt, green!50!black] ({cos(15)}, 0) ellipse [x radius=-.0625,
y radius={sin(15)}] node[right,black] {$r_2$};
\draw[line width=1pt, red, dotted] (90:{cos(15)}) -- (-90:{cos(15)});
\draw[line width=1pt, green!50!black] (-135:.125) circle [radius=.25] node [below left=8pt, black] {$r_4$};
\draw[line width=1pt, green!50!black, dotted] (45:.125) circle [radius=.25];
\fill[white] (-.125,0) circle [radius=.025cm];
\draw[line width=1pt] (115: .125 and 1) arc[start angle=115, end angle=171, x radius=.125, y radius=1]
(-161: .125 and 1.03) arc[start angle=-161, end angle=-90, x radius=.125, y radius=1.03];
\draw[line width=1pt,dotted] (65: .125 and 1.125) arc[start angle=65, end angle=18, x radius=.125, y radius=1.125]
(-65: .125 and 1) arc[start angle=-65, end angle=-9, x radius=.125, y radius=1];
\draw[->, dashed] (70:1.1) node[right] {$r_4$} -- (50:.45);
\node at (-90:1.25) {Type $III$};
\end{scope}
\end{tikzpicture}
\caption{The three possible cells, the corresponding subpyramitoids and its orbifold covers obtained by the reflections on  some faces of their domes.}
\label{fcode2}
\end{center}
\end{figure}

The code of a pyramitoid determines a cellular decomposition which is interesting. The proof
of the following proposition is elementary and Figures~\ref{fcodes1} and~\ref{fcode2} are a good illustration of the result.

\begin{proposition}\label{prop:desc_code}
    The code $(p_n, (r_1,\dots,r_{n-3}))$ of a pyramitoid $\mathbf{Y}_n$ determines a cellular decomposition of $p_n$ such that
    \begin{enumerate}[label=\rm(\alph{enumi})]
    \item The vertices of the decomposition are the intersections $r_j\cap\partial p_n$, and the vertices of $p_n$.
    \item There are two types of edges: the curves $r_j$ and the subdivision of the edges of $p_n$ determined by the intersections of the curves of the code and $\partial p_n$.
    \item The $2$-cells are the closure of the connected components of $p_n\setminus(r_1\cup\dots\cup r_{n-3})$.
\end{enumerate}
    There are three types of $2$-cells of this decomposition,  types $I$, $II$, or $III$, determined by
the valence ($1, 2$, or $3$)
    of the unique vertex of $\almap{\mathbf{Y}_n}$ contained in the cell.

    This decomposition induces a decomposition of $\mathbf{Y}_n$ in subpyramitoids having the $2$-cells of the decomposition
    as bases. Let $F$ be such a cell and let $\mathbf{Y}_F$ the corresponding subpyramitoid, see Figure{\rm~\ref{fcode2}}.

    \begin{enumerate}[label=\rm(\roman{enumi})]
    \item If $F$ is of type $I$ we have a $4$-pyramitoid with label $(0, 1, 0, 1)$; one of the triangular lateral faces is in $\mathbf{Y}_n$,
    and the basis of the other one is in the code.
    \item If $F$ is of type $II$ we have a $5$-pyramitoid with label $(0, 1, 1, 0, 2)$; the bases of the two triangular lateral faces are in the code.
    \item If $F$ is of type $III$ we have a $6$-pyramitoid with label $(0, 2, 0, 2, 0, 2)$; the bases of the three triangular lateral faces are in the code.
    \end{enumerate}

\end{proposition}

\begin{remark}\label{rem:tetraedros_truncados}
    Observe that every subpyramitoid $\mathbf{Y}_F$ is a $3$-pyramid with a $1$, $2$ or $3$  truncated vertices respectively and the basis of the cutting triangle in each truncation is a code line, see Figures~\ref{fcode2} .
\end{remark}

\begin{dfn}
    For $j=1,2,3$ we define \emph{$j$-size} $m_j$ of a code as de number of cells of types $I,II$, or $III$.
\end{dfn}

\begin{remark}
    The size of a code is $n-2=m_1+m_2+m_3$ and $m_1$ is the number of triangular faces (equivalently the number of $0$ in the label). Note also that $m_j$ coincides
    with the number of vertices of valence~$j$ in the tree~$\alma{\mathbf{Y}_n}$.
    It is easy to check that $m_2=n-2m_1$ and $m_3=m_1-2$.
\end{remark}

An \emph{essential circle} in a handlebody is a non-disconnecting simple closed curve in the boundary which bounds a \emph{meridian disk} in the handlebody.

\begin{proposition}
The code $(p_n, (r_1,\dots,r_{n-3})\})$ of the pyramitoid $\mathbf{Y}_n$ defines a collection of $(n - 3)\times 2^{n-2}$ essential circles in $F_{b_n}$ that are the boundary of meridian disks of the handlebody~$H_{b_n}$, such that $H_{b_n}$ minus these disks is a union of $(n-2)\times 2^{n-3}$ $3$-balls.
\end{proposition}

\begin{proof}
Each $r_i$ gives rise to $2^{n-2}$ circles formed by $4$ copies of $r_i$, hence
the number of circles. Each of these circles bounds a disk in the handlebody.
The portions of the handlebody limited by these disks are obtained from the subpyramitoids of Proposition~\ref{prop:desc_code}.

 Let $F$ be a $2$-cell in $p_n$. Let us denote by $j=1,2,3$ whether $F$ is of type $I,II,III$.
 Note that in Figure~\ref{fcode2}, we can see the subpyramitoid $\mathbf{Y}_F$ determined by $F$ as a $3$-pyramid truncated at $1,2$, or $3$ vertices of the basis.

 We introduce an orbifold structure $\mathbf{Y}_F^{\orb,3}$ on $\mathbf{Y}_F$ such that the mirror faces are the ones coming
 from the lateral faces of the $3$-pyramid.
 Let $\pi_F:\mathbb{B}^3_F\to \mathbf{Y}_F^{\orb,3}$ be the universal abelian orbifold cover. Geometrically, we are taking the union of the images of $\mathbf{Y}_F$ by the group of reflections on the  lateral faces of $\mathbf{Y}_F$ totally or partially contained in lateral faces of the original pyramitoid $Y_n$. As we can see in Figure~\ref{fcode2}, $\mathbb{B}^3_F$ is a $3$-ball, with
$2j$ distinguished disks on  the boundary.
With this process we obtain as many $3$-balls as the core tree has vertices. This number is $(n-2)$, see Lemma~\ref{lema:contar}\ref{lema:contar-3}.

If we consider the reflections along the group generated by all the lateral faces of $\orbdomo{\mathbf{Y}_n}$
the preimage of each subpyramitoid are $2^{n-3}$ copies of the above balls. Let us justify this assertion.
The universal abelian cover of $\orbdomo{\mathbf{Y}_n}$ has Galois group $(\mathbb{Z}/2)^n$. Over each
subpyramitoid $\mathbf{Y}_F$ this cover factors through $\pi_F$ with group $(\mathbb{Z}/2)^3$. Hence the
$\pi^{-1}(\mathbf{Y}_F)$ is the disjoint union of $2^{n-3}$ copies of the $3$-ball $\mathbb{B}^3_F$. And
$H_{b_n}$ decomposes into $(n-2)\times 2^{n-3}$ $3$-balls. The distinguished disks are used to glue these balls
and they become meridian disks of the handlebody.
\end{proof}

This decomposition is not minimal but it is invariant by the action of the reflection group $(\mathbb{Z}/2)^n$.
It is possible to simplify it to get only $1$ ball, but we would lose the group action. The next result gives
a simplification respecting the group action. It is valid only for $n>4$ because
for $n=4$ there are only $2$ cells and they are of type~$I$, see Figure~\ref{fig:toro-heegaard}.
Each cell of this type produces two balls of type~$I$ as in Figure~\ref{fcode2}.

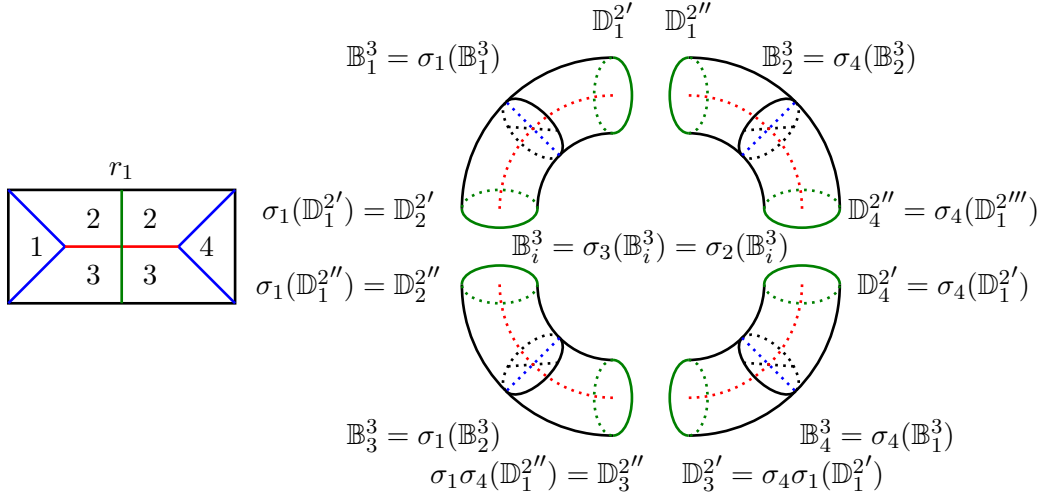
\begin{figure}[ht]
    \centering
    \begin{tikzpicture}
\begin{scope}[xshift=-6.5cm, yshift=-.5cm, scale=.75]
\draw[line width=1pt] (-2, -1) rectangle (2, 1);
\draw[line width=1pt, blue] (-2, 1) -- (-1, 0) -- (-2, -1)
(2, 1) -- (1, 0) -- (2, -1);
\draw[line width=1pt, red]  (-1, 0) -- (1, 0);
\draw[line width=1pt, green!50!black]  (0,-1) -- (0, 1) node[above, black] {$r_1$};
\node at (-1.5,0) {$1$};
\node at (1.5,0) {$4$};
\node at (-.5,.5) {$2$};
\node at (.5,.5) {$2$};
\node at (-.5,-.5) {$3$};
\node at (.5,-.5) {$3$};
\end{scope}
\newcommand\cuarto[1]{
\begin{scope}[#1]
\draw[line width=1pt] (180:1) arc [start angle=180,end angle=90,radius=1];
\draw[line width=1pt] (180:2) arc [start angle=180,end angle=90,radius=2];
\draw[dotted, green!50!black, line width=1pt] (180:1) arc [start angle=0,end angle=180,x radius=.5, y radius=.25];
\draw[green!50!black, line width=1pt] (180:1) arc [start angle=360,end angle=180,x radius=.5, y radius=.25];
\draw[dotted, green!50!black, line width=1pt] (90:1) arc [start angle=270,end angle=90,x radius=.25, y radius=.5];
\draw[green!50!black, line width=1pt] (90:1) arc [start angle=-90,end angle=90,x radius=.25, y radius=.5];
\draw[rotate=-45, line width=1pt] (180:1) arc [start angle=0,end angle=180,x radius=.5, y radius=.25];
\draw[dotted, rotate=-45, line width=1pt] (180:1) arc [start angle=360,end angle=180,x radius=.5, y radius=.25];
\draw[dotted, line width=1pt, blue] (135:1) -- (135:2);
\draw[dotted, line width=1pt, shift={(135:1.5)}, rotate=-45] (-120:.5 and .25) -- (60:.5 and .25);
\draw[dotted, line width=1pt, red] (180:1.5) arc [start angle=180,end angle=90, radius=1.5] ;
\end{scope}
}
\cuarto{}
\cuarto{xscale=-1, xshift=-1cm}
\cuarto{yscale=-1, yshift=1cm}
\cuarto{scale=-1, yshift=1cm, xshift=-1cm}

\node at (.5,-.5) {$\mathbb{B}^3_i=\sigma_3(\mathbb{B}^3_i)=\sigma_2(\mathbb{B}^3_i)$};
\node at (-2.5,2) {$\mathbb{B}^3_1=\sigma_1(\mathbb{B}^3_1)$};
\node at (3,2) {$\mathbb{B}^3_2=\sigma_4(\mathbb{B}^3_2)$};
\node at (-2.5,-3) {$\mathbb{B}^3_3=\sigma_1(\mathbb{B}^3_2)$};
\node at (3.5,-3) {$\mathbb{B}^3_4=\sigma_4(\mathbb{B}^3_1)$};
\node at (0,2.5) {${\mathbb{D}^2_1}'$};
\node at (1,2.5) {${\mathbb{D}^2_1}''$};
\node at (-3.5,0) {$\sigma_1({\mathbb{D}^2_1}')={\mathbb{D}^2_2}'$};
\node at (-3.5,-1) {$\sigma_1({\mathbb{D}^2_1}'')={\mathbb{D}^2_2}''$};
\node at (-1,-3.5) {$\sigma_1\sigma_4({\mathbb{D}^2_1}'')={\mathbb{D}^2_3}''$};
\node at (2.25, -3.5) {${\mathbb{D}^2_3}'=\sigma_4\sigma_1({\mathbb{D}^2_1}')$};
\node at (4.375,0) {${\mathbb{D}^2_4}''=\sigma_4({\mathbb{D}^2_1}''')$};
\node at (4.375,-1) {${\mathbb{D}^2_4}'=\sigma_4({\mathbb{D}^2_1}')$};

\end{tikzpicture}

    \caption{Decomposition in balls for $n=4$.}
    \label{fig:toro-heegaard}
\end{figure}

\begin{coro}
    Let $n>4$. The handlebody $H_{b_n}$ can be seen as the union of the $(n - 2 - m_1)2^{n-3}$ $3$-balls, where $m_1$ is the $1$-size of the code.
    The number of meridian disks is $2^{n-3} (2 n - 6 - m_1)$.
    These balls come from the cells of types $II$ and $III$ and this decomposition is equivariant with respect to the
    action of the reflection group.
\end{coro}

\begin{proof}
    Some of the circles defined by the code lines  bound  isotopic meridian disks in the handlebody. In the type $I$ case, when we reflect on the triangular face of the pyramitoid, the two disks associated with copies of the code line $r_1$ are isotopic to a disk bounded
    by copies of the edge of the basis~$p_n$ in the triangle. Therefore  it is sufficient to consider the types $II$ and $III$.
    We have identified $2^{n-3} m_1$ pairs of isotopic meridian disks corresponding to type $I$ cells.
\end{proof}

A handleboby $H_g$ contains a collection $\{\mathbb{D}^2_1, \dots ,\mathbb{D}^2_g\}$ of pairwise disjoint properly embedded disks such that the result of cutting $H_g$ along $\mathbb{D}^2_1\cup\dots \cup \mathbb{D}^2_g$ is a $3$-ball $\mathbb{B}^3$ (\cite{Hempel1976}). In this way the handlebody $H_g$ is recovered
from a $3$-ball $\mathbb{B}^3$ with $g$ pairs
$({\mathbb{D}^2_i}',{\mathbb{D}^2_i}'')$
of disks on its boundary, by gluing up ${\mathbb{D}^2_i}'$ and ${\mathbb{D}^2_i}''$ by means
of an orientation-reversing homeomorphism; the disk $\mathbb{D}^2_i \subset H_g$ is the common image of
${\mathbb{D}^2_i}'$ and ${\mathbb{D}^2_i}''$.

In our case we have ended with a family of $(n - 2 - m_1)2^{n-3}$ $3$-balls $\mathbb{B}^3_{i}$, where $m_1$ is the $1$-size of the code. These balls are of types $II$ and $III$. We have
$2^{n-3} (2 n - 6 - m_1)$ pairs
$({\mathbb{D}^2_i}',{\mathbb{D}^2_i}'')$ of disks distributed in the boundary of these balls.
As in the classical case, we have orientation-reversing homeomorphisms
$h_i:{\mathbb{D}^2_i}'\to{\mathbb{D}^2_i}''$ which allow to recover $H_{b_n}$
as the disjoint union of the balls $\mathbb{B}^3_{i}$ with the \emph{gluings} $h_i$,
denoting by $\mathbb{D}^2_i\subset H_{b_n}$ the identified disks.
It is possible to simplify this construction with the prize of losing
the group action.

Let $\mathcal{D}$ be the family of $2^{n-3} (2 n - 6 - m_1)$ disks $\mathbb{D}^2_i$.
The collection $\mathcal{D}$  contains at least $b_n$ classes of isotopy.
In fact the code $(p_n, \{r_1,\dots,r_{n-3}\})$ of an $n$-pyramitoid contains all the information of the orbifold structure in $\orbdomo{n}$.

Let us consider the orbifold fundamental group
\[
\pi_1^\orb(p_n^{\orb};x_0)= \left\langle\gamma_1,\dots,\gamma_n\middle| \gamma_j^2=1, \gamma_j\cdot\gamma_{j+1}=\gamma_{j+1}\cdot\gamma_{j}, j=1,\dots,n\bmod{n}\right\rangle.
\]
where the basepoint $x_0$ is the center of  the polygon $p_n$ and
the generators $\gamma _j$ are supported by the segment from $x_0$ to the mid point $m_j$ of the $j$-edge of $p_n$, $\{ j=1,\dots ,n\}$.

There is an orbifold inclusion $i:p_n^{\orb}\hookrightarrow \orbdomo{n}$. The induced map on orbifold fundamental groups is surjective and the kernel is normally generated
by the following elements: Let $r_i$ be a path in the code going from $m_{j_i }$ to $m_{k_i}$.
The boundary of the disk associated to $r_i$ is
\[
(\gamma _{j_i }\gamma _{k_i })^2=[\gamma _{j_i },\gamma _{k_i }]
\]
which is the relation associated to the edge $l_i$ in $\alma{\mathbf{Y}_n}$ corresponding to the
intersection of two faces in $l_i$.
From the maps
\[
\begin{tikzcd}
\pi_1^{\orb}(p_n^{\orb}; x_0)\ar[rr, "i_*", two heads]\ar[rd, "\rho\circ i_*", two heads]&&
\pi_1^{\orb}(\orbdomo{\mathbf{Y}_n}; x_0)
\ar[ld, "\rho", two heads]\\
&(\mathbb{Z}/2)^n&
\end{tikzcd}
\]
we obtain two orbifold covers associated
to $\rho$ and $\rho\circ i_*$:
\[
\begin{tikzcd}
F_{b_n}\rar\dar[hookrightarrow]&
p_n^{\orb}
\dar[hookrightarrow]\\
H_{b_n}\rar[]&\orbdomo{\mathbf{Y}_n}.
\end{tikzcd}
\]
The
fundamental groups of $H_{b_n}$ and $F_{b_n}$
are well-known but actual presentations of them are needed when dealing with Heegaard splittings; the previous
maps allow us to fix concrete generators (and relations) of these groups.

 \section{Bipyramitoids}\label{sec:bipyramitoid}

\begin{dfn}
An \emph{$n$-bipyramitoid} $\bipy{Y}{n}$  is a polyhedron with $n$ faces such that there exists a plane $\Pi$ cutting each face in two polygons without ever crossing a vertex. The plane $\Pi$ divides $\bipy{Y}{n}$ into two  $n$-pyramitoids.
\end{dfn}

\begin{remark}
The \emph{$n$-bipyramitoid} $\bipy{Y}{n}$ can also be defined  by its  Hamiltonian property: the dual graph of the $1$-skeleton has a Hamiltonian cycle (the intersection of the plane $\Pi$ with the faces and edges in $\bipy{Y}{n}$), see the works of Whitney~\cite{W:31} and Helden~\cite{Helden}.
\end{remark}

It follows from the preceding definition  that the real moment-angle manifold cover of a simple bipyramitoid can be  studied using a Heegaard splitting.

\begin{thm}\label{tzdebipiramoide}
  Let $\bipy{Y}{n}$ be a simple $n$-bipyramitoid and let $Z(\bipy{Y}{n})$ be the manifold obtained by reflection on all its faces. Then $Z(\bipy{Y}{n})$ is the result of pasting together $2$~handlebodies $H_{b_{n}}$ by a homeomorphism defined by the code associated to  the two pyramitoids.
\end{thm}

\begin{proof}
  The equatorial plane $\Pi$ cuts  $\bipy{Y}{n}$ into two parts: north  $N(\bipy{Y}{n})$ and south $S(\bipy{Y}{n})$. Each of these two parts is by definition a $n$-pyramitoid. That is, each of these two parts, north and south, have the same  $n$-polygon $p_{n}$ as basis face.
  Reflecting $N(\bipy{Y}{n})$ (resp. $S(\bipy{Y}{n})$) on all their faces but the basis $p_{n}$, one obtain a handlebody $H_{b_{n}}$ (resp. $H'_{b_{n}}$). Then
 \begin{equation}\label{zbp}
   Z(\bipy{Y}{n})= H_{b_{n}} \bigcup _{\partial H_{b_{n}} \equiv \partial H'_{b_{n}}} H'_{b_{n}}
 \end{equation}
  Observe that the boundary of $\partial H_{b_{n}}$ (and $\partial H'_{b_{n}}$) is the surface $F_{b_{n}}$ generated by the reflection of $p_{n}$ on its edges. Then the above expression is a Heegaard splitting of $Z(\bipy{Y}{n})$ (\cite{Heegaard1898,Hempel1976}).  The surface $F_{b_{n}}=\partial H_{b_{n}}=\partial H'_{b_{n}} $ plus the codes $(p_{n}, \{r_i\})$ and $(p_{n}, \{r'_i\})$  define a Heegaard splitting of the manifold $Z(\bipy{Y}{n})$.
\end{proof}

\begin{remark}
    A Heegaard diagram of $Z(\bipy{Y}{n})$ can be obtained from the set of meridians of $H_{b_{n}}$: $\{ \pi_{n}^{-1} (r_i)\subset \partial H_{b_{n}} = F_{b_{n}}, i=1, \dots , 2n-3 \}$ and  the  set of meridians of $H'_{b_{n}}$:  $\{ \pi_{n}^{-1} (r'_i)\subset \partial H'_{b_{n}} =F_{b_{n}}, i=1, \dots , 2n-3 \}$. Note that the two sets of meridians are not minimal but nevertheless they determine $Z(\bipy{Y}{n})$.
\end{remark}

\begin{thm}
The fundamental group $\pi_1(Z(b\mathbf{Y}_n))$ is the free product of the two free groups, $G_{b_n}=\pi_1( H_{b_n})$ and $G'_{b_n}=\pi_1( H'_{b_{n}})$, with amalgamation given by monomorphisms
\begin{equation*}
 \varphi: \pi _1(F_{b_n})\longrightarrow G_{b_n}, \quad \psi: \pi _1(F_{b_n})\longrightarrow G'_{b_n}.
\end{equation*}
Alternatively, $\pi_1(Z(b\mathbf{Y}_n))$ is the quotient of
$\pi _1(F_{b_n}):=\ker\rho$, where
\begin{equation*}
    \begin{tikzcd}[row sep=0pt,/tikz/column 1/.append style={anchor=base east},/tikz/column 2/.append style={anchor=base west}]
    \left\langle{\gamma_i}, i\in\mathbb{Z}/n\middle|
    {[\gamma_i,\gamma_{i+1}]=1}, i\in\mathbb{Z}/n\right\rangle\rar["\rho"]&\displaystyle\bigoplus_{i\in\mathbb{Z}/2}(\mathbb{Z}/n)e_i\\
    \gamma_i\rar[mapsto]&e_i.
    \end{tikzcd}
\end{equation*}
by the relations $[\gamma_{i_j},\gamma_{i_k}]=1$, whenever $(i_j,i_k)$
are the indices of the extremities of $r_1,\dots,r_{n-3}$
and $r'_1,\dots,r'_{n-3}$.
\end{thm}

\begin{proof}
  This is a consequence of (\ref{zbp}).
\end{proof}

The homology group $H_1 (Z(b\mathbf{Y}_n))$ can be computed using the Mayer-Vietoris sequence.
\begin{equation}\label{M-V}
  H_1(F_{b_n})\overset{i_n -i'_n}\longrightarrow H_1(H_{b_n})\oplus H_1(H_{b_n}) \overset{j_n +j'_n}\longrightarrow H_1 (Z(b\mathbf{Y}_n))
\end{equation}
In another work we will study these homology groups as $\mathbb{Z}[(\mathbb{Z}/2)^n]$-modules.

\subsection{Examples}

\begin{ejm}
The tetrahedron $T$ is a $4$-bipyramitoid, it has $4$ faces and the line $L$ in Figure~\ref{ftetraNS} pass once through all the faces. The plane defined by the line $L$ divides $T$ into two parts, North and South, which are smoothings of the $4$-pyramid, in fact they are triangular prisms. Every one has a core tree formed by one edge, $l$ and $l'$ respectively.

 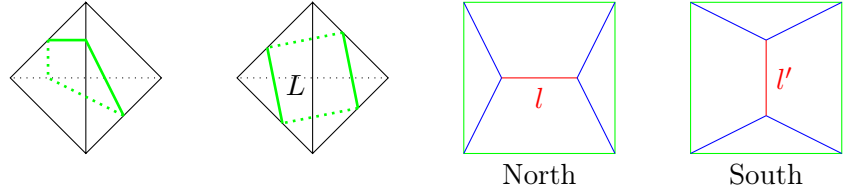
\begin{figure}[ht]
\begin{center}
\begin{tikzpicture}
	\begin{scope}
		\coordinate (A) at (1,0);
		\coordinate (B) at (0,1);
		\coordinate (C) at (-1,0);
		\coordinate (D) at (0,-1);
		\coordinate (AD) at ($ .5*(A) + .5*(D)$) ;
		\coordinate (BD) at ($ .75*(B) + .25*(D)$) ;
		\coordinate (BC) at ($ .5*(B) + .5*(C)$) ;
		\coordinate (AC) at ($ .25*(A) + .75*(C)$) ;

		\draw  (D) -- (B) -- (A) -- (D) -- (C) -- (B);
		\draw[dotted] (A) -- (C);
		\draw[green, line width=1pt] (AD) --  (BD) -- (BC);
		\draw[green, line width=1pt, dotted] (BC) --  (AC) -- (AD);
	\end{scope}

		\begin{scope}[xshift=3cm]
		\coordinate (A) at (1,0);
		\coordinate (B) at (0,1);
		\coordinate (C) at (-1,0);
		\coordinate (D) at (0,-1);
		\coordinate (AD) at ($ .6*(A) + .4*(D)$) ;
		\coordinate (BC) at ($ .4*(B) + .6*(C)$) ;
		\coordinate (CD) at ($ .6*(D) + .4*(C)$) ;
		\coordinate (AB) at ($ .4*(A) + .6*(B)$) ;

		\draw  (D) -- (B) -- (A) -- (D) -- (C) -- (B);
		\draw[dotted] (A) -- (C);
		\draw[green, line width=1pt] (AD) --  (AB)  (CD) -- node[right,black] {$L$} (BC);
		\draw[green, line width=1pt, dotted] (AD) --  (CD)  (AB)-- (BC);
	\end{scope}

			\begin{scope}[xshift=6cm]
		\coordinate (A) at (1,1);
		\coordinate (B) at (-1,1);
		\coordinate (C) at (-1,-1);
		\coordinate (D) at (1,-1);
		\coordinate (E) at (-1/2, 0) ;
		\coordinate (F) at (1/2,0) ;

		\draw[green]  (A) -- (B) -- (C) -- (D) -- cycle;
		\draw[blue]  (B) -- (E) -- (C)   (D) -- (F) -- (A);
		\draw[red] (E)  -- node[below] {$l$} (F);
\node[below] at (0,-1) {North};
	\end{scope}

				\begin{scope}[xshift=9cm]
		\coordinate (A) at (1,1);
		\coordinate (B) at (-1,1);
		\coordinate (C) at (-1,-1);
		\coordinate (D) at (1,-1);
		\coordinate (E) at (0,-1/2) ;
		\coordinate (F) at (0,1/2) ;

		\draw[green]  (A) -- (B) -- (C) -- (D) -- cycle;
		\draw[blue]  (D) -- (E) -- (C)   (B) -- (F) -- (A);
		\draw[red] (E)  -- node[right] {$l'$} (F);
		\node[below] at (0,-1) {South};
	\end{scope}
\end{tikzpicture}

\caption{The tetrahedron and it division by the line $L$.}\label{ftetraNS}

\end{center}
\end{figure}

 \begin{figure}[ht]
 \begin{center}
\begin{tikzpicture}

	\begin{scope}[xshift=0cm]
		\coordinate (A) at (1,1);
		\coordinate (B) at (-1,1);
		\coordinate (C) at (-1,-1);
		\coordinate (D) at (1,-1);
		\coordinate (E) at (-1/2, 0) ;
		\coordinate (F) at (1/2,0) ;

		\draw[green]  (A) -- (B) -- (C) -- (D) -- cycle;
		\draw[blue]  (B) -- (E) -- (C)   (D) -- (F) -- (A);

		\draw[orange, line width=2pt] ($.5*(A) + .5*(B)$)  -- node[right, pos=.25] {$r$} ($.5*(C) + .5*(D)$);
		\fill[white] (0,0) circle[radius=.2cm];
		\draw[red] (E)  node[below right]  {$l$} -- (F);
		\node[below] at (0,-1) {North};
	\end{scope}

	\begin{scope}[xshift=3cm]
		\coordinate (A) at (1,1);
		\coordinate (B) at (-1,1);
		\coordinate (C) at (-1,-1);
		\coordinate (D) at (1,-1);
		\coordinate (E) at (0,-1/2) ;
		\coordinate (F) at (0,1/2) ;

		\draw[green]  (A) -- (B) -- (C) -- (D) -- cycle;
		\draw[blue]  (D) -- (E) -- (C)   (B) -- (F) -- (A);
			\draw[green!50!black, line width=2pt] ($.5*(A) + .5*(D)$)  -- node[below, pos=.75] {$r'$} ($.5*(B) + .5*(C)$);
		\fill[white] (0,0) circle[radius=.2cm];
		\draw[red] (E)   -- (F)  node[below right] {$l'$};
		\node[below] at (0,-1) {South};
	\end{scope}

		\begin{scope}[xshift=6cm]
		\coordinate (A) at (1,1);
		\coordinate (B) at (-1,1);
		\coordinate (C) at (-1,-1);
		\coordinate (D) at (1,-1);
		\coordinate (E) at (0,-1/2) ;
		\coordinate (F) at (0,1/2) ;

		\draw[green]  (A) -- (B) -- (C) -- (D) -- cycle;
			\draw[orange, line width=2pt] ($.5*(A) + .5*(B)$)  -- node[right, pos=.25] {$r$} ($.5*(C) + .5*(D)$);

		\draw[green!50!black, line width=2pt] ($.5*(A) + .5*(D)$)  -- node[below, pos=.75] {$r'$} ($.5*(B) + .5*(C)$);

		\node[below] at (0,-1) {$p_4$};
	\end{scope}
\end{tikzpicture}

\caption{The code for north and south parts and in the polygon $p_4$.}  \label{ftetraCodigo}
\end{center}
\end{figure}
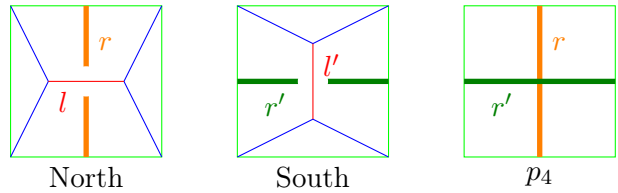

The manifold $Z_N$ obtained by reflecting North on all the faces but the basis $p_4$ is a solid torus $H_1$ with the circle formed by $4$ copies of $l$ as core. The circle formed by $4$ copies of $r$ is a meridian on the boundary $\partial H_1=F_1$ (Figure~\ref{ftetraF1}). In fact in $\partial H_1=F_1$, $r_1$ gives rise to $4$ parallel meridians. Only one is needed.  Analogously, $Z_S$ is a solid torus $H'_1$ with core generated by $l'$ and meridian generated by $r'$. The torus boundary $F_1$ of both solid torus  is made up by $2^4=16$ copies of $p_4$. In $F_1$ the pair composed by a meridians of $H_1$ and a meridian of $H'_1$ form an homology basis. Therefore:
\[
Z(T)= H_{1} \bigcup _{\partial H_1 \equiv \partial H'_{1}} H'_{1} = \mathbb{S}^3.
\]
 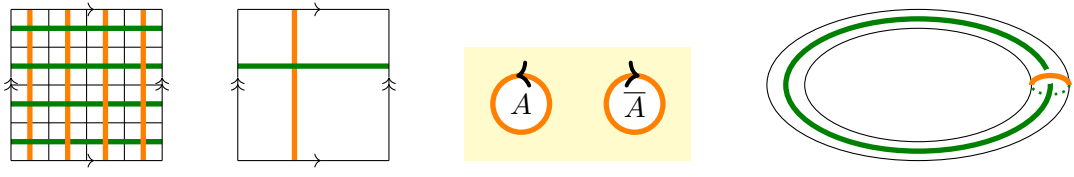
\begin{figure}[ht]
 \begin{center}
\begin{tikzpicture}
\begin{scope}[scale=2]
\foreach \x in {1, 2, 3}{
	\draw (\x/4, 0) -- (\x/4, 1);
	\draw (0,\x/4) -- (1,\x/4);
}
\foreach \x in {0, ..., 3}{
	\draw[orange, line width=2pt] (\x/4 + 1/8, 0) -- (\x/4+ 1/8, 1);
	\draw[green!50!black, line width=2pt]  (0,\x/4+ 1/8) -- (1,\x/4+ 1/8);
}
\draw[->] (.55, 0) -- (1,0) (0,0) -- (.55,0) ;
\draw[->] (.55, 1) -- (1,1) (0,1) -- (.55,1) ;
\draw[->>] (0,.55) -- (0,1) (0,0) -- (0,.55) ;
\draw[->>] (1,.55) -- (1,1) (1,0) -- (1,.55) ;
\end{scope}

\begin{scope}[xshift=3cm, scale=2]

		\draw[orange, line width=2pt] (3/8, 0) -- (3/8, 1);
		\draw[green!50!black, line width=2pt]  (0,5/8) -- (1,5/8);
	\draw[->] (.55, 0) -- (1,0) (0,0) -- (.55,0) ;
	\draw[->] (.55, 1) -- (1,1) (0,1) -- (.55,1) ;
	\draw[->>] (0,.55) -- (0,1) (0,0) -- (0,.55) ;
	\draw[->>] (1,.55) -- (1,1) (1,0) -- (1,.55) ;
\end{scope}

\begin{scope}[xshift=6cm]
\fill[yellow!25!white] (0,0) rectangle (3, 1.5);
\fill[white] (.75,.75)  circle[radius=.375cm];
\node at (.75,.75) {$A$};
\draw[orange, line width=2pt] (.75,.75)  circle[radius=.375cm];
\draw[orange, line width=1.5pt,arrows= {->[color=black]}] (1.125, 0.75) arc[start angle=0, end angle=100, radius=.375];

\fill[white] (2.25,.75)  circle[radius=.375cm];
\node at (2.25,.75) {$\overline{A}$};
\draw[orange, line width=2pt] (2.25,.75)  circle[radius=.375cm];
\draw[orange, line width=1.5pt,arrows= {->[color=black]}] (1.875, 0.75) arc[start angle=180, end angle=80, radius=.375];

\end{scope}

\begin{scope}[xshift=12cm, yshift=1cm]
\draw (0,0) ellipse[x radius=2cm, y radius=1cm];
\draw (0,0) ellipse[x radius=1.5cm, y radius=.75cm];
\draw[line width=2pt, green!50!black] (0,0) ellipse[x radius=1.75cm, y radius=.875cm];
\fill[white] (7.5:1.75 and .875) circle[radius=.1];
\draw[line width=2pt, orange] (2,0)  arc[start angle=0, end angle=180, x radius=.25cm, y radius=.125cm];
\draw[line width=1pt, green!50!black, dotted] (2,0)  arc[start angle=0, end angle=-180, x radius=.25cm, y radius=.125cm];

\end{scope}
\end{tikzpicture}
\caption{Meridians of $H_1$ and $H'_1$ in their boundary $F_1$ and three ways to represent the Heegard spliting.}  \label{ftetraF1}
\end{center}
\end{figure}
\end{ejm}

\begin{ejm}

The triangular prism $TP$ is a $5$-bipyramitoid, it has $5$ faces and the plane drawn in Figure~\ref{prisma} passes once though all the faces. The plane  divides $TP$ into two parts, North and South, which are smoothings of the $5$-pyramid. Every one has a core tree formed by two edges,  $(l_1, l_2)\subset $ North and $(l'_1, l'_2)\subset $ South.

 \begin{figure}[ht]
\begin{center}
\begin{tikzpicture}
	\begin{scope}[scale=1.5]
		\foreach \x in {0, ..., 5}{
			\coordinate (A\x) at (\x*72 + 18:1);
		}

		\draw (A0) -- (A1) -- (A2) -- (A3) -- (A4) -- (A5);
		\coordinate (A) at (-1/2,0);
		\coordinate (B) at (0, 1/4);
		\coordinate (C) at (1/2,0);

		\draw[line width=1pt, blue] (A0) -- (C);
		\draw[line width=1pt, blue] (A1) -- (B);
		\draw[line width=1pt, blue] (A2) -- (A);
		\draw[line width=1pt, blue] (A3) -- (A);
		\draw[line width=1pt, blue] (A4) -- (C);

		\draw[line width=1pt, red] (A) node[above] {$l_1$} -- (B) -- (C) node[above] {$l_2$} ;

		\draw ($1.25/2*(A2) + 1.25/2*(A3)$) node {$1$};
		\draw ($1.25/2*(A4) + 1.25/2*(A0)$) node {$4$};
		\draw ($1.25/2*(A2) + 1.25/2*(A1)$) node {$2$};
		\draw ($1.25/2*(A0) + 1.25/2*(A1)$) node {$3$};
		\draw ($1.25/2*(A3) + 1.25/2*(A4)$) node {$5$};
		\draw[below] ($1.25/2*(A3) + 1.25/2*(A4)$) node {North};
	\end{scope}

	\begin{scope}[xshift=4cm, scale=1.5]
	\foreach \x in {0, ..., 5}{
		\coordinate (A\x) at (\x*72 + 18:1);
	}

	\draw (A0) -- (A1) -- (A2) -- (A3) -- (A4) -- (A5);
	\coordinate (A) at (-1/2,0);
	\coordinate (B) at (0, 1/4);
	\coordinate (C) at (1/2,0);

	\draw[line width=1pt, blue] (A0) -- (C);
	\draw[line width=1pt, blue] (A1) -- (B);
	\draw[line width=1pt, blue] (A2) -- (A);
	\draw[line width=1pt, blue] (A3) -- (A);
	\draw[line width=1pt, blue] (A4) -- (C);

	\draw[line width=1pt, red] (A) node[above] {$l'_1$} -- (B) -- (C) node[above] {$l'_2$} ;

	\draw ($1.25/2*(A2) + 1.25/2*(A3)$) node {$5$};
	\draw ($1.25/2*(A4) + 1.25/2*(A0)$) node {$3$};
	\draw ($1.25/2*(A2) + 1.25/2*(A1)$) node {$1$};
	\draw ($1.25/2*(A0) + 1.25/2*(A1)$) node {$2$};
	\draw ($1.25/2*(A3) + 1.25/2*(A4)$) node {$4$};
	\draw[below] ($1.25/2*(A3) + 1.25/2*(A4)$) node {South};
\end{scope}

	\begin{scope}[xshift=-4cm, yshift=-1cm]
	\foreach \x in {0, ..., 5}{
		\coordinate (A\x) at (\x*72 + 18:1);
	}
	\coordinate (A1) at (-1,3);
	\coordinate (A2) at (1,3);
	\coordinate (A3) at (0,2);

		\coordinate (B1) at (-1,0);
	\coordinate (B2) at (1,0);
	\coordinate (B3) at (0,-1);

	\coordinate (A23) at ($.25*(A2) + .75*(A3)$);
	\coordinate (B23) at ($.5*(B2) + .5*(B3)$);
	\coordinate (A13) at ($.5*(A1) + .5*(A3)$);
	\coordinate (B12) at ($.25*(B2) + .75*(B1)$);
	\coordinate (C1) at ($.5*(A1) + .5*(B1)$);

	\fill[yellow!10!white] (B23) -- (A23) -- (A13) -- (C1) -- (B12) -- cycle;

	\draw  (A1) -- node[above] {$l'_1$} (A2) -- (A3) -- (A1) -- (B1) --node[left] {$l_2$}  (B3) -- (B2) -- node[right] {$l'_2$}  (A2) (A3) -- node[left] {$l_1$} (B3);
	\draw[dotted] (B1) -- (B2);

	\draw[green!50!black, line width=2pt] (B23) --node[right, black] {$2$}  (A23) --node[above, black] {$1$}  (A13) -- node[above, black] {$5$} (C1) ;
\draw[green!50!black, line width=2pt, dotted]	 (C1) -- (B12) -- (B23);

\draw[->] ($(C1)+(-.2,-.5)$) node[left] {4} to[bend right] ($.5*(C1) + .5*(B12) +(-.1,0)$);
\draw[->] ($(B23)+(0,-.5)$) node[right] {3} to[bend left] ($.75*(B23) + .25*(B12) +(0,-.1)$);

\end{scope}

\end{tikzpicture}

\caption{The triangular prism and it division by a plane into two $P_5$.}
\label{prisma}
\end{center}
\end{figure}
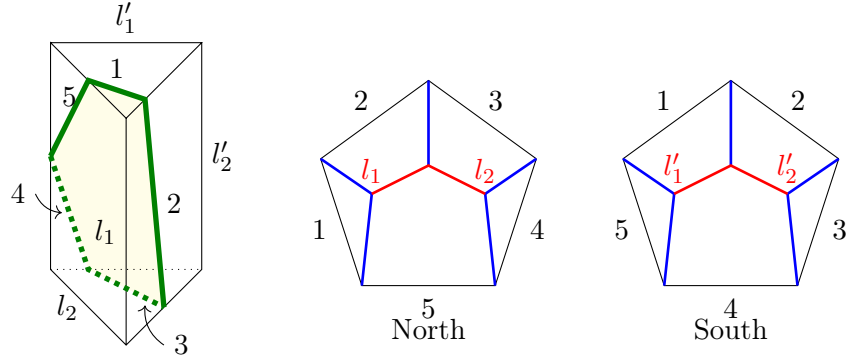

The manifold $Z_N$ obtained by reflecting North (and South) on all the faces but the basis $p_5$ is a handlebody $H_{5}$ (resp. $H'_{5}$), as is shown in Figure~\ref{cod5}. The union of the two handlebodies gives rise to a Heegaard splitting of the manifold $\mathbb{S}^1 \times \mathbb{S}^2$ which is the orbifold cover of the mirror triangular prism. In order to understand this Heegard splitting we analyze the gluing of the two handlebodies by their boundary surface $F_5$ given by the codes $(p_5, \{r_1, r_2\})$ and $(p_5, \{r'_1, r'_2\})$, see Figure~\ref{prismab}.
Therefore
\[
  Z(TP)= H_{5} \bigcup _{\partial H_5  \equiv \partial H'_{5}} H'_{5} = \mathbb{S}^1\times \mathbb{S}^2.
\]

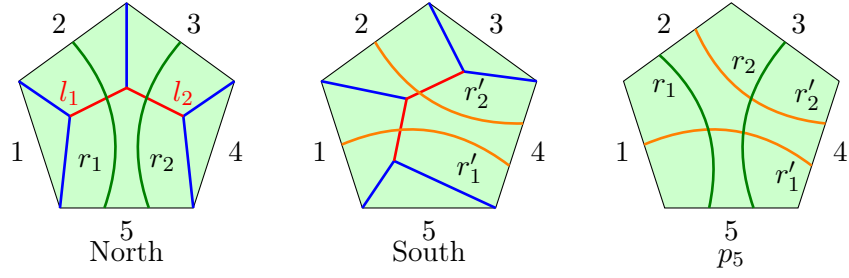
\begin{figure}[ht]
 \begin{center}
\begin{tikzpicture}
	\begin{scope}[scale=1.5]
		\foreach \x in {0, ..., 5}{
			\coordinate (A\x) at (\x*72 + 18:1);
		}

\fill[green!20!white] (A0) -- (A1) -- (A2) -- (A3) -- (A4) -- (A5);

		\draw (A0) -- (A1) -- (A2) -- (A3) -- (A4) -- (A5);
		\coordinate (A) at (-1/2,0);
		\coordinate (B) at (0, 1/4);
		\coordinate (C) at (1/2,0);

		\draw[line width=1pt, blue] (A0) -- (C);
		\draw[line width=1pt, blue] (A1) -- (B);
		\draw[line width=1pt, blue] (A2) -- (A);
		\draw[line width=1pt, blue] (A3) -- (A);
		\draw[line width=1pt, blue] (A4) -- (C);

		\draw[line width=1pt, red] (A) node[above] {$l_1$} -- (B) -- (C) node[above] {$l_2$} ;

		\draw[line width=1pt, green!50!black] ($2/3*(A3) + 1/3*(A4)$) to[bend right]
        node[left, black, pos=.25] {$r_1$} ($1/2*(A1) + 1/2*(A2)$);
		\draw[line width=1pt, green!50!black] ($1/3*(A3) + 2/3*(A4)$) to[bend left] node[right, black, pos=.25] {$r_2$} ($1/2*(A1) + 1/2*(A0)$);

		\draw ($1.25/2*(A2) + 1.25/2*(A3)$) node {$1$};
		\draw ($1.25/2*(A4) + 1.25/2*(A0)$) node {$4$};
		\draw ($1.25/2*(A2) + 1.25/2*(A1)$) node {$2$};
		\draw ($1.25/2*(A0) + 1.25/2*(A1)$) node {$3$};
		\draw ($1.25/2*(A3) + 1.25/2*(A4)$) node {$5$};
		\draw[below] ($1.25/2*(A3) + 1.25/2*(A4)$) node {North};
	\end{scope}

	\begin{scope}[xshift=4cm, scale=1.5]
		\foreach \x in {0, ..., 5}{
			\coordinate (A\x) at (\x*72 + 18:1);
		}

\fill[green!20!white] (A0) -- (A1) -- (A2) -- (A3) -- (A4) -- (A5);

		\draw (A0) -- (A1) -- (A2) -- (A3) -- (A4) -- (A5);
		\coordinate (A) at ([rotate=52] -1/2,0);
		\coordinate (B) at ([rotate=52] 0, 1/4);
		\coordinate (C) at ([rotate=52] 1/2,0);

		\draw[line width=1pt, blue] (A0) -- (C);
		\draw[line width=1pt, blue] (A1) -- (C);
		\draw[line width=1pt, blue] (A2) -- (B);
		\draw[line width=1pt, blue] (A3) -- (A);
		\draw[line width=1pt, blue] (A4) -- (A);

		\draw[line width=1pt, red] (A) -- (B) -- (C) ;

		\draw[line width=1pt, orange] ($2/3*(A4) + 1/3*(A0)$) to[bend right] node[below, black, pos=.25] {$r'_1$}  ($1/2*(A3) + 1/2*(A2)$);
		\draw[line width=1pt, orange] ($1/3*(A4) + 2/3*(A0)$) to[bend left] node[above, black, pos=.25] {$r'_2$} ($1/2*(A1) + 1/2*(A2)$);

		\draw ($1.25/2*(A2) + 1.25/2*(A3)$) node {$1$};
		\draw ($1.25/2*(A4) + 1.25/2*(A0)$) node {$4$};
		\draw ($1.25/2*(A2) + 1.25/2*(A1)$) node {$2$};
		\draw ($1.25/2*(A0) + 1.25/2*(A1)$) node {$3$};
		\draw ($1.25/2*(A3) + 1.25/2*(A4)$) node {$5$};
	\draw[below] ($1.25/2*(A3) + 1.25/2*(A4)$) node {South};
\end{scope}

	\begin{scope}[xshift=8cm, scale=1.5]
		\foreach \x in {0, ..., 5}{
			\coordinate (A\x) at (\x*72 + 18:1);
		}

\fill[green!20!white] (A0) -- (A1) -- (A2) -- (A3) -- (A4) -- (A5);

		\draw (A0) -- (A1) -- (A2) -- (A3) -- (A4) -- (A5);
		\coordinate (A) at ([rotate=52] -1/2,0);
		\coordinate (B) at ([rotate=52] 0, 1/4);
		\coordinate (C) at ([rotate=52] 1/2,0);

		\draw[line width=1pt, orange] ($2/3*(A4) + 1/3*(A0)$) to[bend right] node[below, black, pos=.15] {$r'_1$}  ($1/2*(A3) + 1/2*(A2)$);
		\draw[line width=1pt, orange] ($1/3*(A4) + 2/3*(A0)$) to[bend left]  node[above, black, pos=.1] {$r'_2$}  ($2/3*(A1) + 1/3*(A2)$);

		\draw[line width=1pt, green!50!black] ($2/3*(A3) + 1/3*(A4)$) to[bend right] node[left, black, pos=.75] {$r_1$}  ($1/3*(A1) + 2/3*(A2)$) ;
		\draw[line width=1pt, green!50!black] ($1/3*(A3) + 2/3*(A4)$) to[bend left] node[left, black, pos=.85] {$r_2$} ($1/2*(A1) + 1/2*(A0)$);

		\draw ($1.25/2*(A2) + 1.25/2*(A3)$) node {$1$};
		\draw ($1.25/2*(A4) + 1.25/2*(A0)$) node {$4$};
		\draw ($1.25/2*(A2) + 1.25/2*(A1)$) node {$2$};
		\draw ($1.25/2*(A0) + 1.25/2*(A1)$) node {$3$};
		\draw ($1.25/2*(A3) + 1.25/2*(A4)$) node {$5$};
	\draw ($1.25/2*(A3) + 1.25/2*(A4)$) node[below=3pt] {$p_5$};
\end{scope}

\end{tikzpicture}
\caption{The code for North and South parts and in the polygon $p_5$.}  \label{prismab}
\end{center}
\end{figure}

\end{ejm}

\begin{ejm}

The cube $C$ is a $6$-bipyramitoid, it has $6$ faces and  a plane passing through its center and perpendicular to the diagonal joining two opposite vertices of $C$  intersects all the faces in a hexagon $p_6$ with boundary the line $L$. This plane divides the cube into two parts, North and South, which are smoothings of the $6$-pyramid (Fig. \ref{cubo1NS}). Each part has a core tree with $3$ edges, $(l_1, l_2, l_3)\subset $ North and $(l'_1, l'_2, l'_3)\subset $ South.
 \begin{figure}[ht]
\begin{center}
\begin{tikzpicture}
	\begin{scope}[scale=1.5]
		\foreach \x in {0, ..., 6}{
			\coordinate (A\x) at (\x*60:1);
		}

		\draw (A0) -- (A1) -- (A2) -- (A3) -- (A4) -- (A5) -- (A6);
		\coordinate (O) at (0,0);
		\coordinate (A) at (30:1/2);
		\coordinate (B) at (150:1/2);
		\coordinate (C) at (-90:1/2);

		\fill[green!20!white] (A0) -- (A1) -- (A2) -- (A3) -- (A4) -- (A5) -- (A6);

		\draw[line width=1pt, blue] (A0) -- (A);
		\draw[line width=1pt, blue] (A1) -- (A);
		\draw[line width=1pt, blue] (A2) -- (B);
		\draw[line width=1pt, blue] (A3) -- (B);
		\draw[line width=1pt, blue] (A4) -- (C);
		\draw[line width=1pt, blue] (A5) -- (C);

		\draw[line width=1pt, red] (O) -- (A) node[below] {$l_2$}  (O) -- (C) node[left] {$l_3$} (O) -- (B) node[above right] {$l_1$} ;

		\draw[below] ($1.25/2*(A4) + 1.25/2*(A5)$) node {North};
	\end{scope}

	\begin{scope}[xshift=4cm, scale=1.5]
		\foreach \x in {0, ..., 6}{
			\coordinate (A\x) at (\x*60:1);
		}

		\draw (A0) -- (A1) -- (A2) -- (A3) -- (A4) -- (A5) -- (A6);
		\coordinate (O) at (0,0);
		\coordinate (A) at (90:1/2);
		\coordinate (B) at (210:1/2);
		\coordinate (C) at (-30:1/2);

		\fill[green!20!white] (A0) -- (A1) -- (A2) -- (A3) -- (A4) -- (A5) -- (A6);

		\draw[line width=1pt, blue] (A0) -- (C);
		\draw[line width=1pt, blue] (A1) -- (A);
		\draw[line width=1pt, blue] (A2) -- (A);
		\draw[line width=1pt, blue] (A3) -- (B);
		\draw[line width=1pt, blue] (A4) -- (B);
		\draw[line width=1pt, blue] (A5) -- (C);

		\draw[line width=1pt, red] (O) -- (A) node[below right] {$l'_1$}  (O) -- (C) node[below left] {$l'_2$} (O) -- (B) node[above] {$l'_3$} ;

		\draw[below] ($1.25/2*(A4) + 1.25/2*(A5)$) node {South};
\end{scope}

	\begin{scope}[xshift=-4cm, yshift=-1.25cm, scale=1.6]

	\coordinate (A1) at (0,0);
	\coordinate (A2) at (1,0);
	\coordinate (A3) at (1,1);
	\coordinate (A4) at (0, 1);

	\coordinate (B1) at ($(A1) + (-1/2,1/2)$);
	\coordinate (B2) at ($(A2) + (-1/2,1/2)$);
	\coordinate (B3) at ($(A3) + (-1/2,1/2)$);
	\coordinate (B4) at ($(A4) + (-1/2,1/2)$);

	\coordinate (A12) at ($.5*(A1) + .5*(A2)$);
	\coordinate (A23) at ($.5*(A2) + .5*(A3)$);
	\coordinate (C3) at ($.5*(A3) + .5*(B3)$);
	\coordinate (B34) at ($.5*(B3) + .5*(B4)$);
	\coordinate (B41) at ($.25*(B4) + .75*(B1)$);
	\coordinate (B23) at ($.5*(B2) + .5*(B3)$);

	\coordinate (C1) at ($.5*(A1) + .5*(B1)$);

	\fill[yellow!30!white] (A12) -- (A23) -- (C3) -- (B34) -- (B41) -- (C1) -- cycle;

	\draw  (B1) -- node[left] {$L$} (A1) -- (A2) -- (A3) -- (A4) -- (A1)  (A4)-- (B4) -- (B1) (B4) -- (B3) -- (A3) --  (B3);
\draw[dotted] (B1) -- (B2) -- (B3) (B2) -- (A2);

\draw[green!50!black, line width=1pt] (A12) -- (A23)  (C3) -- (B34)  (B41) -- (C1);
\draw[green!50!black, line width=1pt, dotted] (C1) -- (A12)  (A23) --  (C3)  (B34) -- (B41);

\end{scope}

\end{tikzpicture}
\caption{Cube and it division by the line $L$.}\label{cubo1NS}
\end{center}
\end{figure}
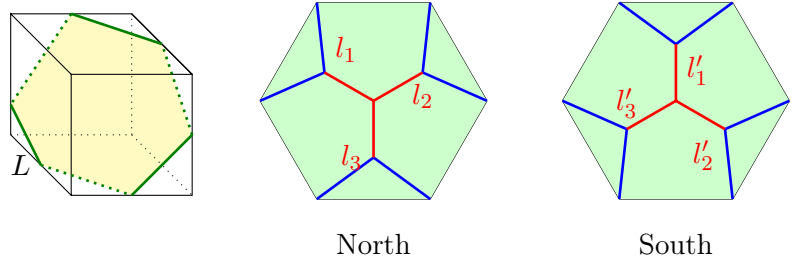

 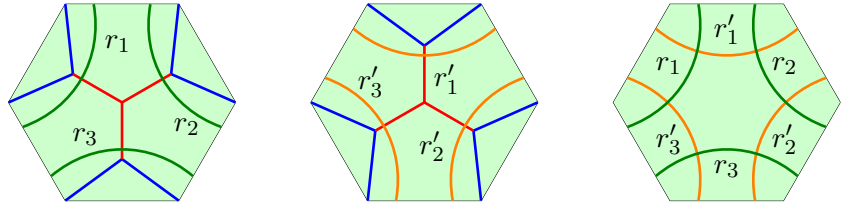
\begin{figure}[ht]
 \begin{center}
\begin{tikzpicture}
	\begin{scope}[scale=1.5]
		\foreach \x in {0, ..., 6}{
			\coordinate (A\x) at (\x*60:1);
		}

		\draw (A0) -- (A1) -- (A2) -- (A3) -- (A4) -- (A5) -- (A6);
		\coordinate (O) at (0,0);
		\coordinate (A) at (30:1/2);
		\coordinate (B) at (150:1/2);
		\coordinate (C) at (-90:1/2);

		\fill[green!20!white] (A0) -- (A1) -- (A2) -- (A3) -- (A4) -- (A5) -- (A6);

		\draw[line width=1pt, blue] (A0) -- (A);
		\draw[line width=1pt, blue] (A1) -- (A);
		\draw[line width=1pt, blue] (A2) -- (B);
		\draw[line width=1pt, blue] (A3) -- (B);
		\draw[line width=1pt, blue] (A4) -- (C);
		\draw[line width=1pt, blue] (A5) -- (C);

		\draw[line width=1pt, red] (O) -- (A)   (O) -- (C)  (O) -- (B)  ;

\draw[green!50!black, line width=1pt] ($1/4*(A4) -- 3/4*(A3)$) to[bend right=40pt] node[pos=.75, right, black] {$r_1$}($1/4*(A1) -- 3/4*(A2)$);

\draw[green!50!black, line width=1pt] ($1/4*(A3) -- 3/4*(A4)$) to[bend left=40pt] node[pos=.25, above, black] {$r_3$}($1/4*(A0) -- 3/4*(A5)$);

\draw[green!50!black, line width=1pt] ($1/4*(A5) -- 3/4*(A0)$) to[bend left=40pt] node[pos=.25, below, black] {$r_2$}($1/4*(A2) -- 3/4*(A1)$);

	\end{scope}

	\begin{scope}[xshift=4cm, scale=1.5]
		\foreach \x in {0, ..., 6}{
			\coordinate (A\x) at (\x*60:1);
		}

		\draw (A0) -- (A1) -- (A2) -- (A3) -- (A4) -- (A5) -- (A6);
		\coordinate (O) at (0,0);
		\coordinate (A) at (90:1/2);
		\coordinate (B) at (210:1/2);
		\coordinate (C) at (-30:1/2);

		\fill[green!20!white] (A0) -- (A1) -- (A2) -- (A3) -- (A4) -- (A5) -- (A6);

		\draw[line width=1pt, blue] (A0) -- (C);
		\draw[line width=1pt, blue] (A1) -- (A);
		\draw[line width=1pt, blue] (A2) -- (A);
		\draw[line width=1pt, blue] (A3) -- (B);
		\draw[line width=1pt, blue] (A4) -- (B);
		\draw[line width=1pt, blue] (A5) -- (C);

		\draw[line width=1pt, red] (O) -- (A)  (O) -- (C)  (O) -- (B) ;

\draw[orange, line width=1pt] ($1/4*(A2) -- 3/4*(A3)$) to[bend left=40pt] node[pos=.35, above, black] {$r'_3$}($1/4*(A5) -- 3/4*(A4)$);

\draw[orange, line width=1pt] ($1/4*(A4) -- 3/4*(A5)$) to[bend left=40pt] node[pos=.35, left, black] {$r'_2$}($1/4*(A1) -- 3/4*(A0)$);

\draw[orange, line width=1pt] ($1/4*(A0) -- 3/4*(A1)$) to[bend left=40pt] node[pos=.35, below, black] {$r'_1$}($1/4*(A3) -- 3/4*(A2)$);

\end{scope}

	\begin{scope}[xshift=8cm, scale=1.5]
		\foreach \x in {0, ..., 6}{
			\coordinate (A\x) at (\x*60:1);
		}

		\draw (A0) -- (A1) -- (A2) -- (A3) -- (A4) -- (A5) -- (A6);
		\coordinate (O) at (0,0);
		\coordinate (A) at (90:1/2);
		\coordinate (B) at (210:1/2);
		\coordinate (C) at (-30:1/2);

		\fill[green!20!white] (A0) -- (A1) -- (A2) -- (A3) -- (A4) -- (A5) -- (A6);

\draw[orange, line width=1pt] ($1/4*(A2) -- 3/4*(A3)$) to[bend left=40pt] node[pos=.6, left, black] {$r'_3$}($1/4*(A5) -- 3/4*(A4)$);

\draw[orange, line width=1pt] ($1/4*(A4) -- 3/4*(A5)$) to[bend left=40pt] node[pos=.4, right, black] {$r'_2$}($1/4*(A1) -- 3/4*(A0)$);

\draw[orange, line width=1pt] ($1/4*(A0) -- 3/4*(A1)$) to[bend left=40pt] node[pos=.5, above, black] {$r'_1$}($1/4*(A3) -- 3/4*(A2)$);

\draw[green!50!black, line width=1pt] ($1/4*(A4) -- 3/4*(A3)$) to[bend right=40pt] node[pos=.6, left, black] {$r_1$}($1/4*(A1) -- 3/4*(A2)$);

\draw[green!50!black, line width=1pt] ($1/4*(A3) -- 3/4*(A4)$) to[bend left=40pt] node[pos=.5, below, black] {$r_3$}($1/4*(A0) -- 3/4*(A5)$);

\draw[green!50!black, line width=1pt] ($1/4*(A5) -- 3/4*(A0)$) to[bend left=40pt] node[pos=.6, right, black] {$r_2$}($1/4*(A2) -- 3/4*(A1)$);

\end{scope}

\end{tikzpicture}

\caption{The code for north and south parts and in the polygon $p_6$. }  \label{fcuboCodigo}
\end{center}
\end{figure}
 Let
\[
\pi_C : Z(C) \longrightarrow C
\]
be the orbifold cover. Then $\pi_C^{-1}(p_6) = F_{17}$ divides $Z(C)$ into two handlebodies $H_{17}$ and $H'_{17}$. That is a Heegaard splitting of $Z(C)$. The preimage $\pi_C^{-1}(r_i)\subset F_{17}$ consists of meridians of $H_{17}$,
and analogously, $\pi_C^{-1}(r'_i)\subset F_{17}$ consists of meridians of $H'_{17}$. By an analysis similar to the one made in the case of the triangular prism, studying  the intersection points between those families of curves $r_i$ and $r'_i$, we can find  canceling pairs reducing the genus of the Heegaard splitting of $Z(C)$. There exist $14$ canceling pairs and
$Z(C)=\mathbb{S}^1 \times \mathbb{S}^1\times \mathbb{S}^1 $.
\end{ejm}

\subsection{Trapezohedron (antibipyramid or Gyrobipyramid)}
\mbox{}

The study of the cube as a bipyramitoid (Figure~\ref{cubo1NS}) suggests the following generalization:
\begin{dfn}
An \emph{$n$-trapezohedron}, (\emph{$n$-antibipyramid} or \emph{$n$-gyrobipyramid}) $\girobipy{Y}{n}$  is the result of pasting together two $Rn$-pyramids (pyramids in a geometry such that all its dihedral angles are  $\frac{\pi }{2}$ angles)  along a plane parallel to a  basis face by a $ \frac{\pi}{n}$-turn.
 It has $2n$ faces (all quadrangles) and $2$~apices. In the equatorial zone of the polyhedra the faces of both copies of the $n$-pyramid are interleaved as in a gear.
\end{dfn}

Note that \emph{$n$-trapezohedrons} are particular cases of $2n$-bipyramitoids (in general non-simple at the apices).

\begin{ejm}
    The cube is the $3$-trapezohedron $\girobipy{Y}{3}$, the only $n$-trapezohedron that is a simple polyhedron, see Figure~\ref{FBP3y4}.
\end{ejm}

 \begin{figure}[ht]
\begin{center}
\begin{tikzpicture}
\begin{scope}[yshift=.25cm]
\coordinate (N) at (0, 1);
\coordinate (S) at (0, -2);
\coordinate (WN) at (-1, 0);
\coordinate (EN) at (1, 0);
\coordinate (B) at (0, 0);
\coordinate (F) at (0, -1);
\coordinate (WS) at (-1, -1);
\coordinate (ES) at (1, -1);

\draw[line width=1pt] (N) node[above] {$N$}-- (EN) -- (F) -- (WN) -- cycle
(WN) -- (WS) -- (S) node[below] {$S$} -- (ES) -- (EN) (F) -- (S);
\draw[dotted] (ES) -- (B) -- (WS) (N) -- (B);
\draw[line width=1pt, red] ($2/3*(WN) + 1/3*(WS)$) -- ($2/3*(F) + 1/3*(WN)$) -- ($2/3*(F) + 1/3*(EN)$) -- ($2/3*(EN) + 1/3*(ES)$);
\draw[line width=1pt, red, dotted] ($2/3*(WN) + 1/3*(WS)$) -- ($3/4*(B) + 1/4*(WS)$) -- ($3/4*(B) + 1/4*(ES)$) -- ($2/3*(EN) + 1/3*(ES)$);
\end{scope}
\begin{scope}[xshift=3cm, yshift=-.25cm]
\coordinate (A) at (-30:.75);
\coordinate (B) at (90:.75);
\coordinate (C) at (-150:.75);
\coordinate (P) at (30:1.25);
\coordinate (O) at (0, 0);
\coordinate (Q) at (150:1.25);
\coordinate (R) at (-90:1.25);
\coordinate (U) at (30:1.75);
\coordinate (V) at (150:1.75);
\coordinate (W) at (-90:1.75);

\fill[green!25!white] (O) circle [radius=1.75];
\draw[red] (O) circle [radius=1];

\draw (O) node[above right] {$N$} -- (A) (C) -- (O) -- (B)
(A) -- (P) -- (B) -- (Q) -- (C) -- (R) -- cycle
(P) -- (U) (Q) -- (V) (R) -- (W);

\draw[line width=1.5pt] (2, -2) -- (2, 2);

\end{scope}

\begin{scope}[xshift=6.5cm]
\coordinate (N) at (0, 2);
\coordinate (S) at (0, -2);
\coordinate (A1) at (-.25, .5);
\coordinate[shift={(0,.5)}] (A2) at (0:1);
\coordinate (A3) at (.25, .5);
\coordinate[shift={(0,.5)}] (A4) at (180:1);
\coordinate (B1) at (.25, -.5);
\coordinate[shift={(0,-.5)}]  (B2) at (0:1);
\coordinate (B3) at (-.25,-.5);
\coordinate[shift={(0,-.5)}]  (B4) at (180:1);
\coordinate (C1) at ($1/2*(A4) + 1/2*(B4)$);
\coordinate (C2) at ($1/4*(A1) + 3/4*(B4)$);
\coordinate (C3) at ($4/5*(B1) + 1/5*(A1)$);
\coordinate (C4) at ($2/3*(B1) + 1/3*(A2)$);
\coordinate (C5) at ($3/5*(A2) + 2/5*(B2)$);
\coordinate (C6) at ($3/4*(A3) + 1/4*(B2)$);
\coordinate (C7) at ($3/4*(A3) + 1/4*(B3)$);
\coordinate (C8) at ($2/3*(A4) + 1/3*(B3)$);

\draw[line width=1pt] (N) node[above] {$N$}-- (A4) -- (B4) -- (S) node[below] {$S$} -- (B2) -- (A2) -- cycle (N) -- (A1) (S) -- (B1)
(B4) -- (A1) -- (B1) -- (A2);
\draw[dotted] (N) -- (A3) -- (B3) -- (S) (A4) -- (B3) (A3) -- (B2);
\draw[line width=1.2pt, red] (C1) -- (C2) -- (C3) -- (C4) -- (C5);
\draw[line width=1pt, red, dotted] (C5) -- (C6) -- (C7) -- (C8) -- (C1) ;
\end{scope}

\begin{scope}[xshift=10cm, yshift=-.25cm]
\coordinate (A) at (0:.75);
\coordinate (B) at (90:.75);
\coordinate (C) at (180:.75);
\coordinate (D) at (-90:.75);
\coordinate (P) at (45:1.25);
\coordinate (O) at (0, 0);
\coordinate (Q) at (135:1.25);
\coordinate (R) at (-135:1.25);
\coordinate (S) at (-45:1.25);
\coordinate (U) at (45:1.75);
\coordinate (V) at (135:1.75);
\coordinate (W) at (-135:1.75);
\coordinate (X) at (-45:1.75);

\fill[green!25!white] (O) circle [radius=1.75];
\draw[red] (O) circle [radius=1];

\draw (D) -- (O) node[above right] {$N$} -- (A) (C) -- (O) -- (B)
(A) -- (P) -- (B) -- (Q) -- (C) -- (R) -- (D) -- (S) -- cycle
(P) -- (U) (Q) -- (V) (R) -- (W) (S) -- (X);

\end{scope}
\end{tikzpicture}

\caption{The bipyramitoids $\girobipy{Y}{3}$ and $\girobipy{Y}{4}$ with the stereographic projection and a equatorial cut in red.}
\label{FBP3y4}
\end{center}
\end{figure}
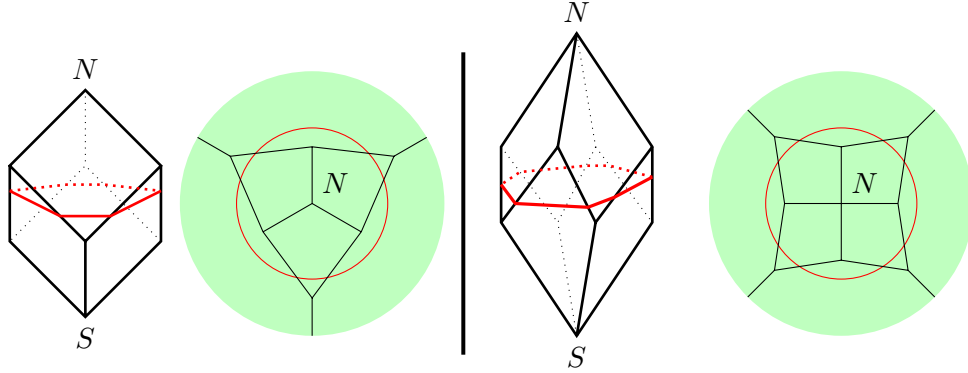

 \begin{figure}[ht]
\begin{center}
\begin{tikzpicture}

\begin{scope}
\coordinate (O) at (0, 0);
\fill[green!25!white] (O) circle [radius=1.75];
\foreach \x in {0,...,5}
{
\coordinate (A\x) at ({72*\x+18}:.75);
\draw (O) -- (A\x);
}
\foreach \x in {0,...,5}
{
\coordinate (B\x) at ({72*\x+54}:1.25);
\coordinate (C\x) at ({72*\x+54}:1.75);
\draw (B\x) -- (C\x);
}

\foreach \x[evaluate={\y=\x+1}] in {0,...,4}
{
\draw (A\x) -- (B\x) -- (A\y);
}

\draw[red] (O) circle [radius=1];

\end{scope}

\begin{scope}[xshift=3cm]

\coordinate (O) at (0, 0);
\draw[red, line width=1.2pt] (O) circle [radius=1];
\clip (0,0) circle [radius=1cm];
\fill[green!25!white] (O) circle [radius=1.75];
\foreach \x in {0,...,5}
{
\coordinate (A\x) at ({72*\x+18}:.75);
\draw (O) -- (A\x);
}
\foreach \x in {0,...,5}
{
\coordinate (B\x) at ({72*\x+54}:1.25);
\coordinate (C\x) at ({72*\x+54}:1.75);
\draw (B\x) -- (C\x);
}

\foreach \x[evaluate={\y=\x+1}] in {0,...,4}
{
\draw (A\x) -- (B\x) -- (A\y);
}

\end{scope}
\end{tikzpicture}

\caption{Stereographic projection of a $5$-trapezohedron with the equatorial cut and the north-half part.  }\label{bipyre}
\end{center}
\end{figure}
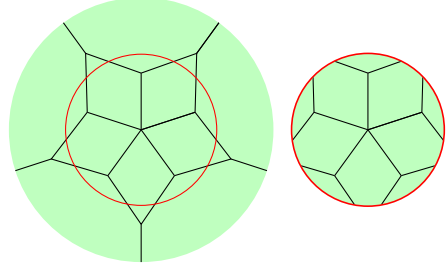

Any smoothing of the bipyramitoid $\girobipy{Y}{n}$ is obtained by pasting together two smoothings of a $2n$-pyramitoid obtained by truncation of all the vertices in the basis of an  $n$-pyramid. The gluing is done  along  the  basis face in such a way that an edge in the basis bounds a triangular face in only one side.
Note that a smoothing of~$\girobipy{Y}{n}$ is a simple $2n$\nobreakdash-bi\-pyramitoid, and
Corollary~\ref{tzdegyrobipiramoide} below holds.

\begin{coro}\label{tzdegyrobipiramoide}
  Let $s\girobipy{Y}{n}$ be a smoothing of an $n$-trapezohedron and
  let $Z(\sgirobipy{Y}{n})$ be the manifold obtained by reflection on all their faces. Then $Z(\sgirobipy{Y}{n})$ is the result of pasting together $2$~handlebodies $H_{b_{2n}}$ by a homeomorphism defined by the code associated to the smoothing of the two pyramids.
\end{coro}

\begin{proof}
  It is a direct consequence of Theorem~\ref{tzdebipiramoide}. Note that the two
  involved $2n$\nobreakdash-pyr\-amitoids have $n$ triangular faces corresponding to alternating edges
  in the equatorial plane
(the red line in Figure \ref{bipyre}).
\end{proof}

The fact that the $2n$-pyramitoids have $n$ alternating triangular faces makes them special. It is also interesting to perform the same smoothing in both pyramitoids. Let us study a special case.

\begin{ejm}{A smoothing of $gb\mathcal{Y}_{4}$}.

The polyhedron $gb\mathcal{Y}_{4}$ has $8$ cuadrangular faces, $8$ vertices ($6$ simples and $2$ with valence $4$) and $12$ edges. Is $44444444$ in the notation of \cite{Dutch20}.

In $gb\mathcal{Y}_{4}$, the North and South parts defined by the equatorial cut have one smoothing (Figure \ref{fBP4NSd}).

 \begin{figure}[ht]
\begin{center}
\begin{tikzpicture}

\begin{scope}[]

\coordinate (O) at (0, 0);
\node[regular polygon, fill=green!25!white, minimum size=2.5cm, regular polygon sides=8] {};
\foreach \x in {0,...,8}
{
\coordinate (A\x) at ({45*\x-22.5}:1.25);
}
\foreach \x in {0,...,8}
{
\coordinate (B\x) at ({90*\x}:.75);
\draw (O) -- (B\x);
}
\foreach \x[evaluate={\y=\x+1}] in {0,...,7}
{
\draw[line width=1pt, red] (A\x) -- (A\y);
}
\foreach \x[evaluate={\y=2*\x+1}, evaluate={\t=2*\x}] in {0, ..., 3}
{
\draw (A\t) -- (B\x) -- (A\y);
}
\node[above right] at (O) {$N$};
\foreach \j in {1, ..., 8}
{
\node at ({-45*\j-90}:1.375) {$\j$};}
\end{scope}

\begin{scope}[xshift=3.5cm]

\coordinate (O) at (0, 0);
\node[regular polygon, fill=green!25!white, minimum size=2.5cm, regular polygon sides=8] {};
\foreach \x in {0,...,8}
{
\coordinate (A\x) at ({45*\x-22.5}:1.25);
}
\foreach \x in {0,...,8}
{
\coordinate (B\x) at ({90*\x + 45}:.75);
\draw (O) -- (B\x);
}

\foreach \x[evaluate={\y=\x+1}] in {0,...,7}
{
\draw[line width=1pt, red] (A\x) -- (A\y);
}
\foreach \x[evaluate={\y=2*\x+1}, evaluate={\t=2*\x + 2}] in {0, ..., 3}
{
\draw (A\t) -- (B\x) -- (A\y);
}
\node[right] at (O) {$S$};
\foreach \j in {1, ..., 8}
{
\node at ({45*\j-180}:1.375) {$\j$};}
\end{scope}

\begin{scope}[xshift=7cm]

\coordinate (O) at (0, 0);
\coordinate (O1) at (45:.25);
\coordinate (O2) at (225:.25);
\node[regular polygon, fill=green!25!white, minimum size=2.5cm, regular polygon sides=8] {};
\draw[line width=1pt,red] (O1) --(O2);

\foreach \x in {0,...,8}
{
\coordinate (A\x) at ({45*\x-22.5}:1.25);
}
\foreach \x in {0,...,8}
{
\coordinate (B\x) at ({90*\x}:.75);
}
\draw (O1) -- (B0);
\draw (O1) -- (B1);
\draw (O2) -- (B2);
\draw (O2) -- (B3);
\foreach \x[evaluate={\y=\x+1}] in {0,...,7}
{
\draw[line width=1pt, red] (A\x) -- (A\y);
}
\foreach \x[evaluate={\y=2*\x+1}, evaluate={\t=2*\x}] in {0, ..., 3}
{
\draw (A\t) -- (B\x) -- (A\y);
}

\foreach \j in {1, ..., 8}
{
\node at ({-45*\j-90}:1.375) {$\j$};}
\end{scope}

\begin{scope}[xshift=10.5cm]

\coordinate (O) at (0, 0);
\coordinate (O1) at (90:.25);
\coordinate (O2) at (-90:.25);
\node[regular polygon, fill=green!25!white, minimum size=2.5cm, regular polygon sides=8] {};
\draw[line width=1pt,red] (O1) --(O2);
\foreach \x in {0,...,8}
{
\coordinate (A\x) at ({45*\x-22.5}:1.25);
}
\foreach \x in {0,...,8}
{
\coordinate (B\x) at ({90*\x + 45}:.75);
}
\draw (O1) -- (B1);
\draw (O2) -- (B2);
\draw (O2) -- (B3);
\draw (O1) -- (B0);

\foreach \x[evaluate={\y=\x+1}] in {0,...,7}
{
\draw[line width=1pt, red] (A\x) -- (A\y);
}
\foreach \x[evaluate={\y=2*\x+1}, evaluate={\t=2*\x + 2}] in {0, ..., 3}
{
\draw (A\t) -- (B\x) -- (A\y);
}
\foreach \j in {1, ..., 8}
{
\node at ({45*\j-180}:1.375) {$\j$};
}
\end{scope}

\end{tikzpicture}

\caption{The smoothing of the $4$-trapezohedron $\girobipy{Y}{4}$
in Figure~\ref{FBP3y4}
is the $5$-gyrobipentaprysm $\girobipy{P}{5}$. We show the North and South parts of $\girobipy{Y}{4}$ and one smoothing.}
\label{fBP4NSd}
\end{center}
\end{figure}
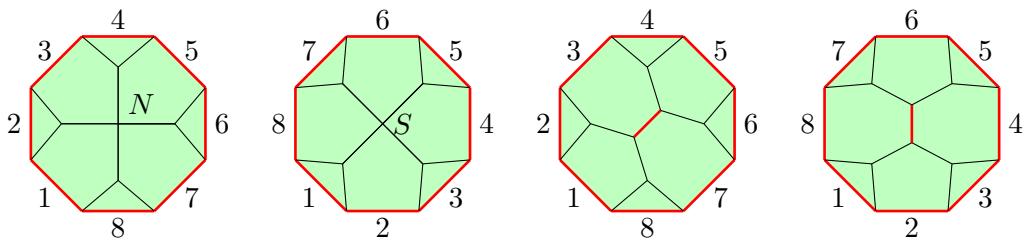

 All the possible combinations of a smoothing of the North part with a smoothing of the South part  give the same polyhedron, the \emph{Gyrobipentaprism} $\girobipy{P}{5}$ , that is the result of
pasting together two pentagonal prism along a lateral face by a $\frac{\pi}{2}$-turn, see Figure~\ref{F4(8)2}.

Theorem \ref{tzdegyrobipiramoide} gives a Heegaard splitting of the manifold  $Z(sgb\mathcal{Y}_{4})$ of genus $b_8=129$. This manifold $Z(sgb\mathcal{Y}_{4})$, being equal to $Z(gbP_5)$,  was already studied in \cite[Section 5.6]{ALdeML2025} where a complete description  was given as a Waldhausen  graph manifold. Note that another Heegaard splitting for graph manifolds was described in~\cite{aim:19}.

 \begin{figure}[ht]
\begin{center}
\begin{tikzpicture}

\begin{scope}[]
\coordinate (S) at (3, 0);
\coordinate (A) at (0, 3);
\coordinate (B) at (3, 3);
\coordinate (C) at (0, 0);
\coordinate (N) at ($1/3*(S) + 2/3*(A)$);
\coordinate (P35) at ($1/2*(S) + 1/2*(A)$);
\coordinate (P12) at ($2/3*(B) + 1/3*(C)$);
\coordinate (P67) at ($1/3*(B) + 2/3*(C)$);
\coordinate (P45) at ($2/3*(S) + 1/3*(A) + 1/15*(B)$);
\coordinate (P34) at ($2/3*(S) + 1/3*(A) - 1/15*(B)$);

\coordinate (Q1) at ($1/3*(B) + 2/3*(A)$);
\coordinate (Q2) at ($1/3*(B) + 2/3*(P12)$);
\coordinate (Q3) at ($1/3*(P12) + 2/3*(P45)$);
\coordinate (Q4) at ($3/4*(P45) + 1/4*(P35)$);
\coordinate (Q5) at ($3/4*(P34) + 1/4*(P35)$);
\coordinate (Q6) at ($1/3*(P67) + 2/3*(P34)$);
\coordinate (Q7) at ($1/3*(C) + 2/3*(P67)$);
\coordinate (Q8) at ($1/3*(C) + 2/3*(A)$);

\draw (S) node[below right] {$\scriptstyle S$} rectangle (A);

\draw (A) -- (N) node[above right=-3pt] {$\scriptstyle N$} -- (P35)
(B) -- (P12) -- (N) -- (P67) -- (C) (S) -- (P45) -- (P12)
(S) -- (P34) -- (P67) (P34) -- (P35) -- (P45);

\draw[line width=1pt, blue]
(Q1) -- node[below, blue, black, pos=.35] {$\scriptstyle 1$}
(Q2) -- node[right, blue, black, pos=.75] {$\scriptstyle 2$}
(Q3) -- node[above, blue, black, pos=1.2] {$\scriptstyle 3$}
(Q4) -- node[below, blue, black, pos=.15] {$\scriptstyle 4$}
(Q5) -- node[above, blue, black, pos=1.2] {$\scriptstyle 5$}
(Q6) -- node[below, blue, black, pos=.25] {$\scriptstyle 6$}
(Q7) --  node[right, blue, black, pos=.6] {$\scriptstyle 7$}
(Q8);

\draw[line width=1pt, blue, dotted]
(Q8) -- node[right, blue, black, pos=.15] {$\scriptstyle 8$}
(Q1);
\end{scope}

\begin{scope}[xshift=4cm]
\coordinate (S) at (3, 0);
\coordinate (A) at (0, 3);
\coordinate (B) at (3, 3);
\coordinate (C) at (0, 0);
\coordinate (S1) at ($3/4*(S) + 1/4*(C)$);
\coordinate (S2) at ($3/4*(S) + 1/4*(B)$);
\coordinate (N) at ($1/3*(S) + 2/3*(A)$);
\coordinate (N1) at ($(N)+(-.25,.25)$);
\coordinate (N2) at ($(N) + (0,-.5)$);
\coordinate (P35) at ($1/2*(S) + 1/2*(A)$);
\coordinate (P12) at ($2/3*(B) + 1/3*(C)$);
\coordinate (P67) at ($1/3*(B) + 2/3*(C)$);
\coordinate (P45) at ($2/3*(S) + 1/3*(A) + 1/15*(B)$);
\coordinate (P34) at ($2/3*(S) + 1/3*(A) - 1/15*(B)$);

\coordinate (Q1) at ($1/3*(B) + 2/3*(A)$);
\coordinate (Q2) at ($1/3*(B) + 2/3*(P12)$);
\coordinate (Q3) at ($1/3*(P12) + 2/3*(P45)$);
\coordinate (Q4) at ($3/4*(P45) + 1/4*(P35)$);
\coordinate (Q5) at ($3/4*(P34) + 1/4*(P35)$);
\coordinate (Q6) at ($1/3*(P67) + 2/3*(P34)$);
\coordinate (Q7) at ($1/3*(C) + 2/3*(P67)$);
\coordinate (Q8) at ($1/3*(C) + 2/3*(A)$);

\node[] at (S) {$\scriptstyle S$};
\node[] at (N) {$\scriptstyle N$};
\draw (S2) -- (B) -- (A) -- (C) -- (S1);
\draw[line width=1pt, red] (S1) -- (S2);

\draw (A) -- (N1)
(B) -- (P12) -- (N1) (P35) -- (N2) -- (P67) -- (C) (S2) -- (P45) -- (P12)
(S1) -- (P34) -- (P67) (P34) -- (P35) -- (P45);
\draw[line width=1pt, red] (N1) -- (N2);

\draw[line width=1pt, blue]
(Q1) -- node[below, blue, black, pos=.35] {$\scriptstyle 1$}
(Q2) -- node[right, blue, black, pos=.75] {$\scriptstyle 2$}
(Q3) -- node[above, blue, black, pos=1.2] {$\scriptstyle 3$}
(Q4) -- node[below, blue, black, pos=.15] {$\scriptstyle 4$}
(Q5) -- node[above, blue, black, pos=1.2] {$\scriptstyle 5$}
(Q6) -- node[below, blue, black, pos=.25] {$\scriptstyle 6$}
(Q7) --  node[right, blue, black, pos=.6] {$\scriptstyle 7$}
(Q8);

\draw[line width=1pt, blue, dotted]
(Q8) -- node[right, blue, black, pos=.15] {$\scriptstyle 8$}
(Q1);
\end{scope}

\begin{scope}[xshift=9cm, yshift=1.5cm]
\foreach \x in {0, ..., 5}
{
\coordinate (A\x) at ({72*\x-54}:.75);
\coordinate (B\x) at ({72*\x-54}:1.5);
}
\coordinate (C0) at ($.5*(A0) + .5*(B0) - .1*(1,0)$);
\coordinate (C4) at ($.5*(A4) + .5*(B4) + .1*(1,0)$);

\coordinate (Q1) at ($.5*(B2)+.5*(B3)$);
\coordinate (Q2) at ($.5*(B2)+.5*(A2)$);
\coordinate (Q3) at ($.3*(A2)+.7*(A1)$);
\coordinate (Q4) at ($.5*(A1)+.5*(A0)$);
\coordinate (Q5) at ($.5*(A0)+.5*(C0)$);
\coordinate (Q6) at ($.7*(C0)+.3*(C4)$);
\coordinate (Q7) at ($.5*(C4)+.5*(B4)$);
\coordinate (Q8) at ($.5*(B4)+.5*(B3)$);
\foreach \x[evaluate={\y=\x+1}] in {0, ..., 4}
{
\draw (A\x) -- (A\y)  (B\x) -- (B\y);
}
\foreach \x in {1, 2, 3}
{
\draw (A\x) -- (B\x);
}
\draw (A0) -- (C0) -- (B0)
(A4) -- (C4) -- (B4)
(C0) -- (C4);

\draw[red, line width=1pt]
(B0) --  node[right, black] {$\scriptstyle S$} (B1)
(A3) --  node[right, black] {$\scriptstyle N$} (A4);

\draw[line width=1pt, blue] (Q1) -- node[below, black] {$\scriptstyle 1$}
(Q2) -- node[right, black] {$\scriptstyle 2$}
(Q3) -- node[left, black] {$\scriptstyle 3$}
(Q4) -- node[right, black] {$\scriptstyle 4$}
(Q5) -- node[above left=-5pt, black] {$\scriptstyle 5$}
(Q6) -- node[below, black] {$\scriptstyle 6$}
(Q7) -- node[right, black] {$\scriptstyle 7$} (Q8);
\draw[line width=1pt, blue, dotted]
(Q8) -- node[left, black] {$\scriptstyle 8$} (Q1);
\end{scope}
\end{tikzpicture}

\caption{The smoothing of $gb\mathcal{Y}_{4}$ is the \emph{Gyrobipentaprism} $\girobipy{P}{5} $.  }\label{F4(8)2}
\end{center}
\end{figure}
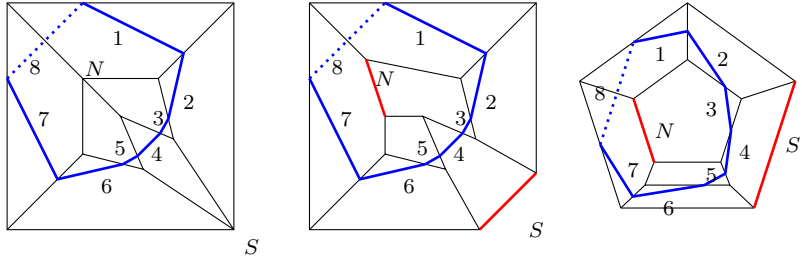

\end{ejm}
 \section{General case} \label{sec:general}

Let $Q$ be a polyhedron whith $r$ faces associated to an intersection of ellipsoids $Z(Q)$ such that a vertex $v$ of $Q$ has valence $n\geq 4$.
\begin{equation}\label{piZQ}
 \pi : Z(Q) \longrightarrow Q
\end{equation}
The vertex $v$ corresponds to $2^{r-n}$ singular points in $Z(Q)$. A neighborhood $U_v$ of $v$ in $Q$ is a $n$-pyramid.  Then  $\pi ^{-1} \overline{(Q-U_v)} $ has a boundary composed by $2^{r-n}$ surfaces $(F_{b_n})_i$, $i=1,\dots ,2^{r-n})$, where $(F_{b_n})_i$ is a  surface of genus $b_n$. There are $N_n$ possible ways to obtain a smoothing of $Q$, corresponding to the number of $n$-pyramitoids, and each one of them is determined by code (or the core tree, or the Coxeter graph, or the triangulation of the $n$-polygon).  Those  codes determine  the different ways of attaching a handlebody $H_{b_n}$ to each boundary component surface $(F_{b_n})_i$.
Therefore there are $N_n$ ways to obtain a smoothing  for $Z(Q)$ in the $2^{r-n}$ singular vertices $\pi ^{-1}(v)$.
%

 \providecommand\noopsort[1]{}
\providecommand{\bysame}{\leavevmode\hbox to3em{\hrulefill}\thinspace}
\providecommand{\MR}{\relax\ifhmode\unskip\space\fi MR }
\providecommand{\MRhref}[2]{%
  \href{http://www.ams.org/mathscinet-getitem?mr=#1}{#2}
}
\providecommand{\href}[2]{#2}

\end{document}